\documentclass[10pt]{amsart}
\usepackage[utf8]{inputenc}
\usepackage[margin=1in]{geometry}
\usepackage{amsfonts,amssymb,amsmath,amsthm,tikz,comment,mathtools,setspace,float,stmaryrd,datetime}
\usepackage[toc]{appendix}
\usepackage{enumitem}
\usepackage[hidelinks]{hyperref}
\usepackage{scalerel} 
\numberwithin{equation}{section}
\usepackage{chngcntr}
\counterwithin{figure}{section}
\allowdisplaybreaks

\newcommand{\R}{\mathbb{R}}
\newcommand{\Q}{\mathbb{Q}}
\newcommand{\Z}{\mathbb{Z}}
\newcommand{\F}{\mathcal{F}}

\newcommand{\A}{\mathcal{A}}

\newcommand{\B}{\mathcal{B}}

\newcommand{\E}{\mathcal{E}}

\newcommand{\Nor}{\mathcal N}

\newcommand{\wt}{\widetilde}

\newcommand{\deq}{\overset{d}{=}}
\newcommand{\Da}{\overleftarrow{D}}
\newcommand{\Ra}{\overleftarrow{R}}
\newcommand{\Qa}{\overleftarrow{Q}}

\DeclarePairedDelimiter\floor{\lfloor}{\rfloor}
\newcommand{\ve}{\varepsilon}
\newcommand{\sqtwopi}{\sqrt{2\pi}}

\newcommand{\zfraclambda}{\frac{z -\lambda}{\sqrt 2}}
\newcommand{\negzfraclambda}{-\frac{z + \lambda}{\sqrt 2}}
\newcommand{\f}{\frac}
\newcommand{\mbf}{\mathbf}

\newcommand{\Pp}{\mathbb P}
\newcommand{\Ee}{\mathbb E}
\newcommand{\NU}{\operatorname{NU}}
\newcommand{\h}{h}
\newcommand{\vv}{v}
\newcommand{\Largsup}{L \arg \sup}

\newcommand{\sig}{{\scaleobj{0.8}{\boxempty}}} 
\newcommand{\sigg}{{\scaleobj{0.9}{\boxempty}}}

\theoremstyle{plain}
\newtheorem{theorem}{Theorem}[section]

\newtheorem{lemma}[theorem]{Lemma}

\theoremstyle{definition}
\newtheorem{definition}[theorem]{Definition}

\theoremstyle{remark}
\newtheorem{remark}[theorem]{Remark}

\newcommand{\be}{\begin{equation}}
\newcommand{\ee}{\end{equation}}

\title[Brownian last-passage percolation]{Busemann process and semi-infinite geodesics in Brownian last-passage percolation}
\author{Timo Sepp{\"a}l{\"a}inen}
\address{Timo Sepp{\"a}l{\"a}inen, University of Wisconsin-Madison, Mathematics Department, Van Vleck Hall, 480
Lincoln Dr., Madison WI 53706-1388, USA.}
\email{seppalai@math.wisc.edu}
\author{Evan Sorensen}
\address{Evan Sorensen, University of Wisconsin-Madison, Mathematics Department, Van Vleck Hall, 480
Lincoln Dr., Madison WI 53706-1388, USA.}
\email{elsorensen@wisc.edu}

\date{\today}

\begin{document}
\maketitle
\begin{abstract}
We prove the existence of semi-infinite geodesics for Brownian last-passage percolation (BLPP). Specifically, on a single event of probability one, there exist  semi-infinite geodesics  started from every space-time point and traveling in every asymptotic direction. Properties of these geodesics include uniqueness for a fixed initial point and direction, non-uniqueness for fixed direction but random initial points, and coalescence of all geodesics traveling in a common, fixed direction. Along the way, we prove that for fixed northeast and southwest directions, there almost surely exist no bi-infinite geodesics in the given directions. The semi-infinite geodesics are constructed from Busemann functions.  Our starting point is a result of Alberts, Rassoul-Agha and Simper that established Busemann functions for  fixed points and directions.  Out of this, we construct the global process of Busemann functions simultaneously for all initial points and directions, and then the family of semi-infinite Busemann geodesics.  The  uncountable space  of the semi-discrete setting requires extra consideration and leads to new phenomena, compared to discrete models.  
\end{abstract}
\tableofcontents

\section{Introduction}

\subsection{Brownian last-passage percolation}
 Brownian last-passage percolation (BLPP) dates back to  the 1991 work of Glynn and Whitt~\cite{glynn1991}, where  the model appeared as a large-scale limit of multiple queues in series under heavy traffic conditions. Harrison and Williams~\cite{Harrison1985,harrison1990,harrison1992} also studied what is known as the Brownian queue and developed a stability result for that model. More specifically, if the arrivals process is given by increments of Brownian motion and the  service process is given by increments of an independent Brownian motion with drift, the departures process is also given by increments of Brownian motion. The connection between BLPP and the Brownian queue is expounded on by O'Connell and Yor~\cite{brownian_queues}, who also introduced the positive temperature version of this model, known as the Brownian polymer. We discuss the connection to queuing theory in Section~\ref{section:queue} and Appendix~\ref{section:queue and stationary}.

\subsection{BLPP in the Kardar-Parisi-Zhang universality class}
 In the early 2000s, Gravner, Tracy, and Widom~\cite{Gravner} and Baryshnikov~\cite{Baryshnikov}   discovered that the Brownian last-passage value $L_{(1,0),(n,1)}(\mbf B)$ (to be defined in \eqref{BLPP formula} below) has the same distribution as the largest eigenvalue of an $n \times n$ GUE random matrix. Soon after, O'Connell and Yor~\cite{rep_non_colliding} provided an alternate proof relying on the queuing interpretation of the model. Since then, Brownian last-passage percolation and the Brownian polymer have been widely studied as gateways to properties of the KPZ universality class. As a semi-discrete model with one continuous and one discrete parameter, BLPP serves as an intermediary between discrete models such as LPP  on the planar integer lattice  and continuum models such as the stochastic heat equation and the KPZ equation.  In~\cite{CorwinHammond}, Corwin and Hammond constructed the Airy line ensemble and proved that certain statistics of Brownian last-passage percolation converge in distribution to the Airy line ensemble. 
 Dauvergne, Nica, and  Vir{\'a}g (\cite{Dauvergne2019UniformCT}, Corollary 6.4) provided an alternate method to prove this fact, which couples Brownian last-passage percolation with geometric random walks. 

Recently, Dauvergne, Ortmann, and Vir{\'a}g~\cite{Directed_Landscape} constructed the directed landscape, a central object in the KPZ universality class, as the scaling limit of Brownian last-passage percolation. Another central object is the KPZ fixed point, constructed by  Matetski,  Quastel, and Remenik~\cite{KPZfixed}, as the limit of the totally asymmetric simple exclusion process. In~\cite{Directed_Landscape}, a variational duality is described between the directed landscape and KPZ fixed point, which was rigorously proven in~\cite{reflected_KPZfixed}. Even more recently, convergence to the KPZ fixed point was shown for a larger class of models, including the height function of the KPZ equation. This was done independently by Quastel and Sarkar~\cite{KPZ_equation_convergence} and Vir{\'a}g~\cite{heat_and_landscape}.

\subsection{Semi-infinite geodesics in discrete models} The study of infinite geodesics in planar random growth models has gone through a number of stages over the last 30 years, beginning with the work of Licea and Newman~\cite{licea1996, Newman} on  first-passage percolation with i.i.d.\ edge weights. 
Under a global curvature assumption on the limit shape, for  continuously distributed edge weights, they proved existence of a deterministic full-Lebesgue measure set of directions in which  there is a unique semi-infinite geodesic out of every lattice point. They also showed that, for each direction in this set, the semi-infinite geodesics in that direction all coalesce.

A separate strand of work consists of long-term efforts to prove the nonexistence of bi-infinite geodesics.  In first-passage percolation, Licea and Newman~\cite{licea1996}  showed that there are no bi-infinite geodesics in fixed northeast and southwest directions. Howard and Newman~\cite{howard2001} later proved similar results for Euclidean last-passage percolation. Around this time, Wehr and Woo~\cite{wehr_woo_1998} proved that, under a first moment assumption on the edge weights, there are no bi-infinite geodesics for first-passage percolation that lie entirely in the upper-half plane. 

In 2016, Damron and Hanson~\cite{Damron_Hanson2016} strengthened the result of Licea and Newman by proving that, if the weights have continuous distribution and the boundary of the limit shape is differentiable, for each fixed direction, there are no bi-infinite geodesics with one end having that direction. The conjectured nonexistence was finally resolved in exponential last-passage percolation (LPP), known also as the exponential corner growth model (CGM). The proofs came in two independent works: first by Basu, Hoffman, and Sly~\cite{SlyNonexistenceOB} and shortly thereafter by Bal{\'a}zs, Busani, and the first author~\cite{Balzs2019NonexistenceOB}. The latter proof is in spirit aligned with the development in the present paper, as it rested on understanding  the joint distribution of the Busemann functions from~\cite{CGM_Joint_Buse}. 

Another focus of research has been the coalescence structure of geodesics.
Ferrari and Pimentel \cite{ferr-pime-05}  imported the  Licea--Newman approach~\cite{licea1996,Newman} to exponential LPP on the lattice.  Later Pimentel~\cite{pimentel2016} developed a probabilistic duality between the coalescence time of two semi-infinite geodesics and their last exit times from the initial boundary, again in exponential LPP. A key idea was the equality in distribution of the tree of directed semi-infinite geodesics and the dual  tree of southwest-directed geodesics. From this,  he obtained a lower bound on tail probabilities of the coalescence time for two geodesics starting from $(-\floor{k^{2/3}},\floor{k^{2/3}})$ and $(\floor{k^{2/3}},-\floor{k^{2/3}})$.  Coalescence bounds have seen significant recent  improvement in \cite{BasuSarkarSly_Coalescence} and  \cite{XiaoTimoCoalescence}. A study of the coalescence structure of finite geodesics in BLPP was undertaken by Hammond in four papers~\cite{Hammond1,Hammond3,Hammond4,Hammond2}. The existence of semi-infinite geodesics in the Airy line ensemble, for a countable dense set of directions and initial points, was proven recently by Sarkar and Vir{\'a}g in~\cite{Sarkar-Virag-21}.


Hoffman was the first to use ideas of Busemann functions to study geodesics in first-passage percolation. In~\cite{hoffman2008} he showed the existence of disjoint semi-infinite geodesics. Later, Damron and Hanson~\cite{Damron_Hanson2012} constructed generalized Busemann functions from weak subsequential limits of first-passage times. This allowed them to develop results for semi-infinite geodesics under weaker assumptions than the global  curvature. Under the assumption that the limit shape is strictly convex and differentiable, they proved that every semi-infinite geodesic has an asymptotic direction and that, in every direction, there exists a semi-infinite geodesic out of every lattice point.

On the side of discrete last-passage percolation with general i.i.d.\ weights, Georgiou, Rassoul-Agha, and the first author~\cite{general_Busemanns,goed_and_comp_interface} used  the  stationary  LPP process to prove existence of Busemann functions under mild moment conditions. The  Busemann functions were then used to construct  semi-infinite geodesics. They also showed that if the shape function is strictly concave, every semi-infinite geodesic has an asymptotic direction.   Further in this direction,  in~\cite{Timo_Coalescence}  the first author  introduced a new proof of the coalescence of semi-infinite geodesics that utilizes the  stationary  LPP process. 

Our work is  related to this last approach.  We use Busemann functions to construct    and  study semi-infinite geodesics in BLPP.   Existence of Busemann functions in  BLPP, for fixed initial points and directions, was recently established by Alberts, Rassoul-Agha and Simper~\cite{blpp_utah}, along with  Busemann functions and infinite polymer measures for the semi-discrete Brownian polymer.  
Their BLPP result is the  starting point of our study.

\subsection{New techniques and phenomena in  the semi-discrete model}
The present  paper develops  the global setting of Busemann functions and semi-infinite geodesics in the semi-discrete BLPP, with a view to future study of their finer properties.  The novelty lies in going beyond the  discrete set-up.  In discrete last-passage percolation, one can prove a probability-one statement about semi-infinite geodesics out of a fixed initial point, and then that statement extends to all initial points by a simple union bound. This is not the case in BLPP. To overcome this difficulty, we need new methods of proof. Additionally, the continuum of points gives rise to new results regarding non-uniqueness of semi-infinite geodesics, as seen in Item~\eqref{itm: non-uniqueness intro} below.

Specific items proved in this paper include the following, and are recorded as Theorems~\ref{thm:general_SIG} and~\ref{thm:summary of properties of Busemanns for all theta}.
\begin{enumerate}
    \item The  Busemann functions of BLPP from~\cite{blpp_utah}  are extended to a global process on a single event of probability one, for all the uncountably many initial points and directions. 
    \item  Once the Busemann process is in place, we use it to construct semi-infinite geodesics for BLPP.  With probability one, there exists a family of semi-infinite geodesics, starting from each initial point and in each asymptotic direction.
    \item \label{itm: non-uniqueness intro} With probability one, all semi-infinite geodesics, whether constructed by the Busemann functions or not, have an asymptotic direction. For a fixed initial point and direction, there is almost surely a unique semi-infinite geodesic starting from the given point and traveling asymptotically in the given direction. We also show that for a fixed direction, there is a countably infinite set of initial points whose geodesic in that direction is not unique. This non-uniqueness into a fixed direction is a new phenomenon that is not present in exponential last-passage percolation on the lattice.  
    \item \label{itm:intro_coal} For each fixed direction, we prove that all semi-infinite geodesics, traveling in that common direction, coalesce.
    \item For fixed northeast and southwest directions, we prove the almost sure nonexistence of bi-infinite geodesics in those directions.
    \end{enumerate}

To construct an infinite up-right path on the lattice $\Z^2$, one chooses at each step whether to move upward or to the right. In the BLPP setting, one chooses a real-valued increment  to the right  and then takes  a unit-size  upward step.  Section~\ref{section:geodesic construction intro} explains  informally   this construction. The locations of the upward steps of a Busemann geodesic are determined by a variational problem for Brownian motion with drift. 
This formulation is significant in at least two ways. First, this step is  where non-uniqueness of geodesics can arise.   Understanding it involves properties of Brownian paths. Furthermore,  this variational construction can be  potentially and profitably adapted to other models. This includes both continuum models of the KPZ class, such as the directed landscape, and lattice LPP. 
We address the latter in 
Section~\ref{sec:discreteLPP}.


The proof of the coalescence  in Item~\eqref{itm:intro_coal} above requires technical novelties.  The underlying idea from~\cite{Timo_Coalescence} is to construct a dual environment from the original environment and a Busemann function. If two geodesics in a given  direction do not coalesce, there exists a bi-infinite geodesic in the dual environment with given northwest and southeast directions. Then, it is proven that  there are almost surely no bi-infinite geodesics in fixed  directions.  

A key point is a non-intersection property between the original semi-infinite geodesics  and southwest-travelling semi-infinite geodesics in the dual environment. This is in Theorem~\ref{strong crossing theorem for geodesics and dual geodesics}, which is the analogue of Lemma 4.4 in~\cite{Timo_Coalescence}. The dual environment in BLPP is constructed through dual queuing mappings, first presented in~\cite{brownian_queues} and further studied in~\cite{Sepp_and_Valko,blpp_utah}. See Equation~\eqref{definition of dual weights from Busemanns} for the precise definition. While the general approach is not new, our contribution comes first from  Theorem~\ref{bijectivity of D R joint mapping} that  shows that joint queuing mappings can be inverted by reverse-time queuing mappings. Then, we use this theorem and variational formulas to construct southwest-travelling semi-infinite geodesics in the dual environment. Lastly, the non-intersection of Theorem~\ref{strong crossing theorem for geodesics and dual geodesics} is proved by comparing the   maximizers of the variational formulas.



\subsection{Future work}
The present paper lays the foundation  for future work on  Busemann functions and semi-infinite geodesics, which most immediately will cover the following. 
    \begin{enumerate}
    \item To study the full family of semi-infinite geodesics in  BLPP, we derive the joint distribution of Busemann functions across all directions. The analogous result was achieved in~\cite{CGM_Joint_Buse} for the exponential CGM, which led to important advances in the global structure of geodesic trees \cite{Geometry_of_Geodesics}, coalescence \cite{XiaoTimoCoalescence}, convergence \cite{bala-busa-sepp-20, busa-ferr-20},  and bi-infinite geodesics \cite{Balzs2019NonexistenceOB}. We likewise study the geometry of geodesics in BLPP and give a full characterization of exceptional directions where geodesics do not coalesce, in terms of the Busemann process.
    \item We define the competition interface for each initial point of BLPP. This has a straightforward definition in discrete models, but requires more care in the semi-discrete setting.  Specifically, we show that, on each level $m$ of the semi-discrete space $\Z\times\R$, outside a random set of Hausdorff dimension $\f{1}{2}$, all points $(m,t)$ have a degenerate competition interface. For each of the points in this set of Hausdorff dimension $\f{1}{2}$, there exists a random direction $\theta^\star$ such that there are two semi-infinite geodesics starting at $(m,t)$ in direction $\theta^\star$ that only share the common initial point.  \label{itm:comp interface}
\end{enumerate}
As can be seen from Item~\eqref{itm:comp interface}, there are many rich questions to explore related to non-uniqueness of semi-infinite geodesics in Brownian last-passage percolation. Other questions related to Hausdorff dimension have recently been explored in the KPZ fixed point and the directed landscape  in~\cite{Airy_Fractal,DirectedLanscape_Hausdorff_dim,Ganguly-Hegde-2021,KPZ_violate_Johansson}.  
Our longer-term goal is to extend the approach of this paper to the continuum models of the KPZ class.

\subsection{Organization of the paper}   Section~\ref{section: defs_and_results}  defines the Brownian last-passage percolation model and related objects. We also state results from other papers, including the existence of Busemann functions from~\cite{blpp_utah}.   Section~\ref{section:main results} contains the main results on the existence and properties of the global Busemann process and semi-infinite geodesics. 
 Section~\ref{sec:SIG_constr} states more technical theorems related to the specific construction of semi-infinite geodesics from the Busemann functions. In Section~\ref{section:connections}, we discuss connections to infinite polymer measures in the O'Connell-Yor polymer, semi-infinite geodesics in discrete last-passage percolation, and to queuing theory. The proofs of this paper are contained in Sections~\ref{section:Busemann construction} and~\ref{section:main proofs}. Section~\ref{section:Busemann construction} constructs the global Busemann process and derives its properties. All results for semi-infinite geodesics are proved in Section~\ref{section:main proofs}. 
 The Appendices collect some standard material, background from the literature, and some technical theorems.

 \section{Definitions and previous results} \label{section: defs_and_results}

\subsection{Preliminaries}
The following notation and conventions are used throughout the paper.
\begin{enumerate} [label=\rm(\roman{*}), ref=\rm(\roman{*})]  \itemsep=3pt 
    \item For a function $f:\R \rightarrow \R$ , we write 
$
f(s,t) = f(t) - f(s)
$
and 
$
\wt f(t) = -f(-t).
$
\item $\Z$, $\Q$ and $\R$ are restricted by subscripts, as in for example $\Z_{> 0}=\{1,2,3,\dotsc\}$.  

\item Whenever $m \le n \in \Z$ and $s \le t \in \R$, we say that $(m,s) \le (n,t)$.
\item Let $X \sim \Nor(\mu,\sigma^2)$ indicate that the random variable $X$ has normal distribution with mean $\mu$ and variance $\sigma^2$. For $\alpha > 0$, let $X \sim \operatorname{Exp}(\alpha)$ indicate that $X$ has exponential distribution with rate $\alpha$, or equivalently, mean $\alpha^{-1}$.
\item  Equality in distribution between random variables and  processes is denoted by $\deq$.  
\item A two-sided Brownian motion is a continuous random process $\{B(t): t \in \R\}$ such that $B(0) = 0$ almost surely and such that $\{B(t):t \ge 0\}$ and $\{B(-t):t \ge 0\}$ are two independent standard Brownian motions on $[0,\infty)$. 
\item For $\lambda \in \R$,   $\{Y(t): t \in \R\}$ is a two-sided Brownian motion with drift $\lambda$ if the process $\{Y(t) - \lambda t: t \in \R\}$ is a two-sided Brownian motion. 
\item The square $\sigg$ as a superscript represents a sign $+$ or $-$. 
\end{enumerate}
\subsection{Geodesics in Brownian last-passage percolation} \label{section:def of BLPP}
The Brownian last-passage process is defined as follows. On a probability space $(\Omega, \F,\Pp)$, let $\mathbf B = \{B_r\}_{r \in \Z}$ be a field of independent, two-sided Brownian motions. For $(m,s) \le (n,t)$, define the set 
\[
\Pi_{(m,s),(n,t)} := \{\mbf s_{m,n} = (s_{m - 1},s_m,\ldots,s_n) \in \R^{n - m + 2}: s = s_{m - 1} \le s_m \le \cdots \le s_n = t   \}.
\]
 Denote the energy of a sequence $\mbf s_{m,n} \in \Pi_{(m,s),(n,t)}$ by
\be \label{E10}
\E(\mbf s_{m,n}) = \sum_{r = m}^n B_r(s_{r - 1},s_r).
\ee
Now, for $\mbf x = (m,s) \le (n,t) = \mbf y$, define the Brownian last-passage time as
\begin{equation} \label{BLPP formula}
L_{\mbf x,\mbf y}(\mbf B) = \sup\{\E(\mbf s_{m,n}): \mbf s_{m,n} \in \Pi_{\mbf x,\mbf y}\}.
\end{equation}
Whenever the specific field of Brownian motions used is either clear from or not important in the context, we write $L_{(m,s),(n,t)}$.

\begin{figure}[ht!]
    \centering
    \begin{tikzpicture}
\draw[gray,thin] (0.5,0) -- (15.5,0);
\draw[gray,thin] (0.5,0.5) --(15.5,0.5);
\draw[gray, thin] (0.5,1)--(15.5,1);
\draw[gray,thin] (0.5,1.5)--(15.5,1.5);
\draw[gray,thin] (0.5,2)--(15.5,2);
\draw[black,thick] plot coordinates {(1.5,-0.1)(1.7,-0.2)(1.9,0.1)(2.1,0.4)(2.3,-0.3)(2.5,0.2)(2.7,0.1)(2.9,-0.3)(3.1,-0.4)(3.3,0.2)(3.5,0.3)(3.7,-0.4)(3.9,0.1)(4.1,-0.2)(4.3,-0.3)(4.45,0.2)};
\draw[red, ultra thick] plot coordinates {(4.5,0.4)(4.7,0.7)(4.9,0.3)(5.1,0.1)(5.4,0.6)(5.7,0.3)(5.7,0.6)(5.9,0.8)(6.1,0.4)(6.3,0.3)(6.6,0.4)(6.95,0.8)};
\draw[blue,thick] plot coordinates {(7,0.8)(7.3,1.2)(7.7,1.1)(8,0.6)(8.2,0.8)(8.5,1.3)(8.7,1.1)(9,0.7)(9.2,1.1)(9.45,1.3)};
\draw[green,thick] plot coordinates
{(9.5,1.1)(9.7,1.3)(9.9,1.7)(10.2,1.1)(10.4,1.3)(10.6,1.9)(10.8,1.4)
(11.1,1.2)(11.3,1.6)(11.5,1.9)(11.8,1.4)(12,1.1)(12.2,1.8)(12.4,1.3)(12.6,1.2)(12.8,1.1)(12.95,1.7)};
\draw[brown,thick] plot coordinates {(13,2.3)(13.3,1.7)(13.5,2.4)(13.7,2.3)(13.9,2.1)(14.1,1.9)(14.3,1.7)(14.4,2.1)(14.6,1.8)(14.8,1.6)(15,2.1)};
\node at (1.5,-0.5) {$s$};
\node at (4.5,-0.5) {$s_0$};
\node at (7,-0.5) {$s_1$};
\node at (9.5,-0.5) {$s_2$};
\node at (13,-0.5) {$s_3$};
\node at (15,-0.5) {$t$};
\node at (0,0) {$0$};
\node at (0,0.5) {$1$};
\node at (0,1) {$2$};
\node at (0,1.5) {$3$};
\node at (0,2) {$4$};
\end{tikzpicture}
    \caption{\small The Brownian increments $B_r(s_{r - 1},s_r)$ for $r=0,\dotsc,4$ in \eqref{E10} that make up the energy of the path depicted in Figure \ref{fig:BLPP_geodesic}. }
    \label{fig:BLPP maximizing path}
\end{figure}
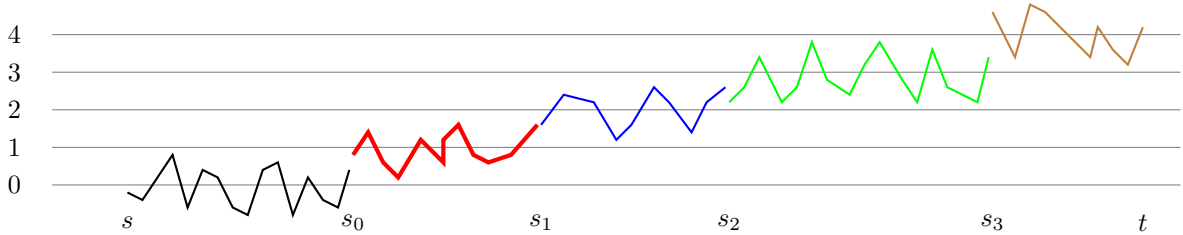

The elements  of $\Pi_{\mbf x,\mbf y}$ are in bijection with paths in $\R^2$ between $\mbf x$ and $\mbf y$, that move to the right in real-valued increments, and move upwards in integer increments. For a given $\mbf s_{m,n} \in \Pi_{(m,s),(n,t)}$, the path consists of the following points: 
\begin{equation} \label{set of points in path}
 \bigcup_{r = m}^n\bigl\{(r,u): u \in [s_{r - 1},s_r]\bigr\} \cup \bigcup_{r = m}^{n - 1} \bigl\{(v,s_r):  v \in [r,r + 1]\bigr\}.
\end{equation}
This set consists of horizontal and vertical line segments, such that the vertical segments occur at the points $s_r$. Because of this bijection, we sometimes say $\Gamma \in \Pi_{\mbf x,\mbf y}$ for such an up-right path. For $(m,t) \in \Z \times \R$, we graphically represent the $t$-coordinate as the horizontal coordinate and the $m$-coordinate as the vertical coordinate in the plane. Since $\Pi_{\mbf x,\mbf y}$ is a compact set and Brownian motion is continuous, on a single event of probability one, for all $(m,t) = \mbf x \le \mbf y = (n,t) \in \Z \times \R$, there exists a sequence $\mbf s_{m,n} \in \Pi_{\mbf x,\mbf y}$ such that $\E(\mbf s_{m,n}) = L_{\mbf x,\mbf y}$. The associated path is called a \textit{geodesic} between the points.

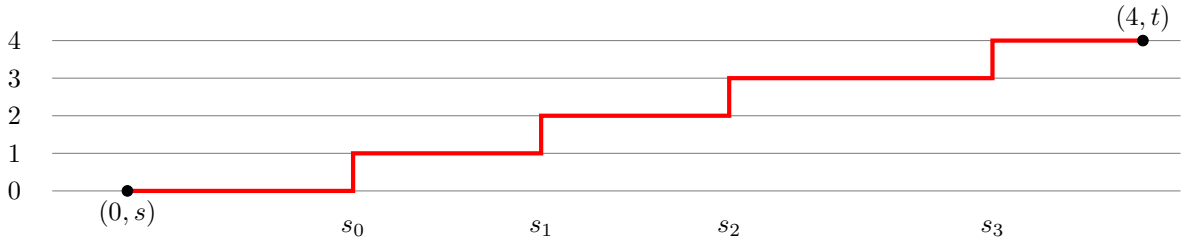
\begin{figure}[ht]
\begin{tikzpicture}
\draw[gray,thin] (0.5,0) -- (15.5,0);
\draw[gray,thin] (0.5,0.5) --(15.5,0.5);
\draw[gray, thin] (0.5,1)--(15.5,1);
\draw[gray,thin] (0.5,1.5)--(15.5,1.5);
\draw[gray,thin] (0.5,2)--(15.5,2);
\draw[red, ultra thick] (1.5,0)--(4.5,0)--(4.5,0.5)--(7,0.5)--(7,1)--(9.5,1)--(9.5,1.5)--(13,1.5)--(13,2)--(15,2);
\filldraw[black] (1.5,0) circle (2pt) node[anchor = north] {$(0,s)$};
\filldraw[black] (15,2) circle (2pt) node[anchor = south] {$(4,t)$};
\node at (4.5,-0.5) {$s_0$};
\node at (7,-0.5) {$s_1$};
\node at (9.5,-0.5) {$s_2$};
\node at (13,-0.5) {$s_3$};
\node at (0,0) {$0$};
\node at (0,0.5) {$1$};
\node at (0,1) {$2$};
\node at (0,1.5) {$3$};
\node at (0,2) {$4$};
\end{tikzpicture}
\caption{\small Example of a planar path from $(0,s)$ to $(4,t)$, represented by the sequence $(s=s_{-1}, s_0, s_1, s_2, s_3, s_4=t)\in\Pi_{(0,s),(4,t)}$.}
\label{fig:BLPP_geodesic}
\end{figure}

To an infinite sequence, $s = s_{m - 1} \le s_m \le s_{m + 1} \le \cdots$ we similarly associate a semi-infinite path. It is possible that $s_r = \infty$ for some $r \ge m$, in which case the last segment of the path is the ray $[s_{r - 1},\infty) \times \{r\}$, where $r$ is the first index with $s_r = \infty$. 
The infinite path has \textit{direction} $\theta \in [0,\infty]$ or is \textit{$\theta$-directed} if 
\[
\lim_{n \rightarrow \infty}\f{s_n}{n} \qquad\text{exists and equals }\theta.
\]
We call an up-right semi-infinite path a \textit{semi-infinite geodesic} if, for any two points $\mbf x\le \mbf y \in \Z \times \R$ that lie along the path, the portion of the path between the two points is a geodesic between the two points. Similarly, a \textit{bi-infinite geodesic} is a bi-infinite path that forms a geodesic between any two of its points. Two  up-right, semi-infinite paths \textit{coalesce} if there exists $\mbf z \in \Z \times \R$ such that the two paths agree above and to the right of $\mbf z$, as shown in Figure~\ref{fig:coalescence of geodesics}. Alternatively, if the paths are defined by sequences of jump times $s^1 = s_{m_1 - 1}^1 \le s_{m_1}^1\cdots$ and $s^2 = s_{m_2 -1}^2 \le s_{m_2}^2 \le \cdots$, then the two paths coalesce if and only if there exists $N \in \Z$ such that $s_r^1 = s_r^2$ for all $r \ge N$. 
  
  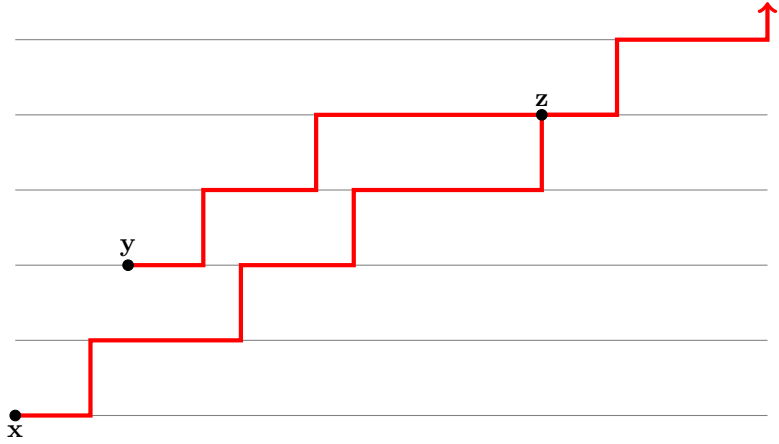
\begin{figure}[ht!]
 \centering
            \begin{tikzpicture}
            \draw[gray,thin] (0,0)--(10,0);
            \draw[gray,thin] (0,1)--(10,1);
            \draw[gray,thin] (0,2)--(10,2);
            \draw[gray,thin] (0,3)--(10,3);
            \draw[gray,thin] (0,4)--(10,4);
            \draw[gray,thin] (0,5)--(10,5);
            \draw[red, ultra thick,->] plot coordinates {(0,0)(1,0)(1,1)(3,1)(3,2)(4.5,2)(4.5,3)(7,3)(7,4)(8,4)(8,5)(10,5)(10,5.5)};
            \draw[red, ultra thick] plot coordinates {(1.5,2)(2.5,2)(2.5,3)(4,3)(4,4)(8,4)};
            \filldraw[black] (0,0) circle (2pt) node[anchor = north] {$\mbf x$};
            \filldraw[black] (1.5,2) circle (2pt) node[anchor = south] {$\mbf y$};
            \filldraw[black] (7,4) circle (2pt) node[anchor = south] {$\mbf z$};
            \end{tikzpicture}
            \caption{\small Coalescence of semi-infinite paths}
            \label{fig:coalescence of geodesics}
            \bigskip
        \end{figure}
        
The following lemma, due to Hammond~\cite{Hammond4}, establishes uniqueness of geodesics for a fixed initial and terminal point. 
\begin{lemma}[\cite{Hammond4}, Theorem B.1] \label{thm:uniqueness of LPP time}
Fix endpoints $\mbf x \le \mbf y \in \Z \times \R$. Then, there is almost surely a unique path whose energy achieves $L_{\mbf x,\mbf y}(\mbf B)$.
\end{lemma}
However, it is also true that for each fixed initial point $\mbf x \in \Z \times \R$, with probability one, there exist points $\mbf y \ge \mbf x$, such that the geodesic between $\mbf x$ and $\mbf y$ is not unique. Hence, the following lemma  is important for our understanding.  It is a deterministic statement which holds for last-passage percolation across any field of continuous functions.

\begin{lemma}[\cite{Directed_Landscape}, Lemma 3.5] \label{existence of leftmost and rightmost geodesics}
Between any two points $(m,s) \le (n,t) \in \Z \times \R$,  there is a rightmost and a leftmost Brownian last-passage geodesic. That is, there exist $\mbf s_{m,n}^L,\mbf s_{m,n}^R \in \mbf \Pi_{(m,s),(n,t)}$, that are maximal for $\E(\mbf s_{m,n})$, such that, for any other maximal sequence $\mbf s_{m,n}$, $s_{r}^L \le s_r \le s_r^R$ holds for $m \le r \le n$.
\end{lemma}

\subsection{Busemann functions}
In the present paper, semi-infinite geodesics are constructed from Busemann functions. Busemann functions are defined to be the asymptotic difference of last-passage times from two different starting points to a common terminal point that is traveling to $\infty$ in a given direction. See Figure~\ref{fig:Busemann functions}. The direction is indexed by a parameter $\theta > 0$. 
  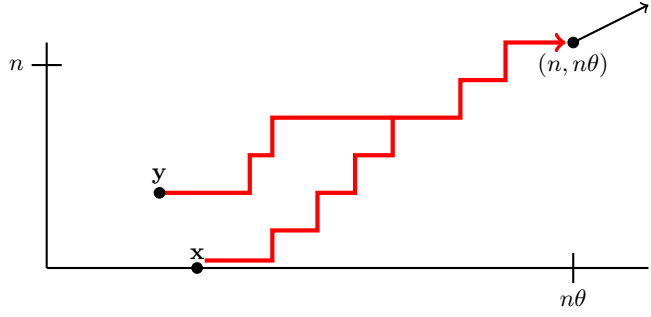
\begin{figure}[t]
  \centering
            \begin{tikzpicture}
            \draw[black,thick] (0,0)--(8,0);
            \draw[black,thick] (0,0)--(0,3);
            \draw[red, ultra thick] plot coordinates {(1.5,1)(2.7,1)(2.7,1.5)(3,1.5)(3,2)(4.6,2)};
            \draw[red, ultra thick,arrows = ->] plot  coordinates {(2.1,0.1)(3,0.1)(3,0.5)(3.6,0.5)(3.6,1)(4.1,1)(4.1,1.5)(4.6,1.5)(4.6,2)(5.5,2)(5.5,2.5)(6.1,2.5)(6.1,3)(6.9,3)};
            \draw[black,thick,arrows = ->] (7,3) --(8,3.5);
            \filldraw[black] (1.5,1) circle (2pt) node[anchor =  south] {\small $\mbf y$};
            \filldraw[black] (2,0) circle (2pt) node[anchor = south] {\small $\mbf x$};
            \filldraw[black] (7,3) circle (2pt);
            \node at (7,2.7) {\small $(n,n\theta)$};
            \draw[black,thick] (-0.2,2.7)--(0.2,2.7);
            \node at (-0.4,2.7) {\small $n$};
            \draw[black,thick] (7,-0.2)--(7,0.2);
            \node at (7,-0.4) {\small $n\theta$};
            \end{tikzpicture}
            \caption{\small Two geodesics with a common terminal point $(n,n\theta)$.}
            \label{fig:Busemann functions}
        \end{figure}
Existence of these Busemann functions was proven in~\cite{blpp_utah}, both in the positive temperature and zero temperature cases. We state the zero-temperature  result: 
\begin{theorem}[\cite{blpp_utah}, Theorem 4.2] \label{thm:existence of Busemann functions for fixed points}
Fix $\theta > 0$ and $\mathbf x,\mathbf y \in \Z \times \R$. Then, there exists a random variable $\B^\theta(\mbf x,\mbf y)$ and an event $\Omega_{\mbf x,\mbf y}^{(\theta)}$ of probability one, on which
\begin{equation} \label{eqn:limit definition of Busemann functions}
\B^\theta(\mathbf x,\mathbf y) = \lim_{n \rightarrow \infty} \bigl[L_{\mathbf x,(n,t_n)} - L_{\mathbf y,(n,t_n)}\bigr]
\end{equation}
holds for any sequence $\{t_n\} \subseteq \R$ satisfying ${t_n}/{n}\to \theta$.
Further, if
\begin{align} 
    &\vv_{m}^\theta(t) := \B^\theta((m - 1,t),(m,t)), \text{ and} \label{vertical Busemann simple expression}\\[1em]
    &\h_m^\theta(t) := \B^\theta((m,0),(m,t)), \label{horizontal Busemann simple expression}
\end{align}
then $\vv_m^\theta(t) \sim \operatorname{Exp}\bigl(\f{1}{\sqrt \theta}\bigr)$ and $\h_m^\theta(s,t) \sim \mathcal N\bigl(\f{t - s}{\sqrt \theta},|t - s|\bigr)$ for all $s,t \in \R$ and $m \in \Z$. 
\end{theorem}

\section{Main results} \label{section:main results}
\subsection{Semi-infinite geodesics} \label{sec:general_geodesics}The following result summarizes the contributions of this paper related to semi-infinite geodesics in BLPP.
\begin{theorem} \label{thm:general_SIG}
The following hold
\begin{enumerate}[label=\rm(\roman{*}), ref=\rm(\roman{*})]  \itemsep=3pt 
     \item \label{itm:SIG_existence} With probability one, for every initial point $\mbf x \in \Z \times \R$ and every direction $\theta > 0$, there exists a $\theta-$directed semi-infinite geodesic starting from $\mbf x$.
    \item{\rm(}Uniqueness for fixed initial points and directions{\rm)} \label{itm:uniqueness of geodesic for fixed point and direction} For each fixed $\mbf x \in \Z \times \R$ and $\theta > 0$, there exists an event, $\Omega_{\mbf x}^{(\theta)}$, of probability one, on which there is exactly one $\theta$-directed semi-infinite geodesic starting from $\mbf x$.
    \item{\rm(}Non-uniqueness for fixed direction and random initial points{\rm)} \label{itm:size of non-uniqueness}For each $\theta > 0$, there exists an event $\wt \Omega^{(\theta)}$, of probability one, on which the set
    \[
    \{\mbf x \in \Z \times \R: \text{the }\theta\text{-directed semi-infinite geodesic starting from }\mbf x \text{ is not unique}\}
    \]
    is countably infinite. For every $(m,t)$ in this set, at most one of the $\theta$-directed semi-infinite geodesics passes through the point $(m,t +\ve)$ for some $\ve > 0$. All the others pass through $(m +1, t)$. 
    \item \label{itm:all geodesics are directed} With probability one, every semi-infinite geodesic is $\theta$-directed for some $\theta \in [0,\infty]$. That is, for any infinite sequence $t = t_{m - 1} \le t_m \le t_{m + 1} \le \cdots$ defining a semi-infinite geodesic starting from some point $(m,t) \in \Z \times \R$, the limit
    \be \label{eqn:general_sig_limit}
    \lim_{n \rightarrow \infty}\f{t_n}{n} \ \text{ exists in } [0,\infty]. 
    \ee
    \item \label{itm:only vertical or horizontal geodesics are trivial}
    With probability one, if for any $(m,t) \in \Z \times \R$ and any such sequence, the limit~\eqref{eqn:general_sig_limit} equals $0$, then $t_r = t$ for all $r \ge m$. Similarly, if this limit is $\infty$, then $t_m =t_{m + 1} = \cdots = \infty$. That is, the only semi-infinite geodesics that are asymptotically vertical or horizontal are trivial {\rm(}i.e. straight lines{\rm)}.
     \item{\rm(}Non-existence of bi-infinite geodesics for fixed directions{\rm)} \label{itm:no bi-infinite geodesics in given direction} Fix $\theta,\eta > 0$. Then, there exists an event, $\Omega^{(\theta,\eta)}$, of probability one, on which there are no bi-infinite geodesics defined by jump times $\cdots \le \tau_{-1} \le \tau_0 \le \tau_1 \le \cdots$ such that
 \begin{equation} \label{eqn:limit condition for bi-infintie geodesics}
 \lim_{n \rightarrow \infty}\f{\tau_n}{n} = \theta\qquad\text{and}\qquad\lim_{n \rightarrow \infty} \f{\tau_{-n}}{n}  = -\eta.
 \end{equation}
    \item \label{itm:coalescence} {\rm(}Coalescence of geodesics in a fixed direction{\rm)} For each $\theta > 0$, there exists an event $\widehat \Omega^{(\theta)}$, of probability one, on which all $\theta$-directed semi-infinite geodesics coalesce. 
\end{enumerate}
\end{theorem}
\begin{remark}
As discussed in the introduction, the non-uniqueness stated in Part~\ref{itm:size of non-uniqueness} is a new phenomenon that arises from the semi-discrete nature of the model. We refer the reader to Section~\ref{sec:non_unique} for further discussion on non-uniqueness.
\end{remark}
 Theorem~\ref{thm:general_SIG}\ref{itm:no bi-infinite geodesics in given direction} is proven by first solving the ``midpoint problem." This problem first appeared in the context of first-passage percolation in a paper of Benjamini, Kalai, and Schramm~\cite{benjamini2003}. In that context the problem asks whether
\[
\lim_{n \rightarrow \infty} \Pp\Bigl(\Bigl\lfloor\f{n}{2}\Bigr \rfloor e_1 \text{ lies on some geodesic between } 0 \text{ and } ne_1\Bigr) = 0.
\]
In 2016, Damron and Hanson~\cite{Damron_Hanson2016} proved that this is true, under the assumption that the boundary of the limit shape is differentiable. Later Ahlberg and Hoffman~\cite{Ahlberg_Hoffman} proved this result without the assumption of differentiability. The following formulation of the midpoint problem more closely matches that for exponential last-passage percolation in the arXiv version of~\cite{Timo_Coalescence}, Theorem 4.12.  
\begin{lemma}[Midpoint problem] \label{lemma:midpoint prob for BLPP}
 Let $\theta,\eta > 0$ and $(m,t) \in \Z \times \R$. Then, the following subset of $\Omega$ is contained in an event of probability zero: 
 \begin{align*}
     &\Big\{\text{there exists a sequence } \{t_n\}_{n \in \Z}  \text{ satisfying } \lim_{n \rightarrow \infty} \f{t_n}{n} = \theta \text{ and }\lim_{n \rightarrow \infty} \f{t_{-n}}{-n} = \eta \text{ and }  \\[1em]
     &\qquad\text{such that, for each } n \in \Z_{> 0}, \text{ some geodesic between } (-n,t_{-n}) \text{ and } (n,t_n) \text{ passes through } (m,t)\Big\}
 \end{align*}
 \end{lemma}

\subsection{Existence and properties of the Busemann process}
To prove Theorem~\ref{thm:general_SIG}, we extend the individual Busemann functions of  Theorem~\ref{thm:existence of Busemann functions for fixed points} to a global {\it Busemann process}. The following transformations are used to understand the structure of this process. For functions  $Z,B:\R \rightarrow \R$ satisfying 
$Z(0) = B(0) = 0$ and $
\limsup_{s \rightarrow \infty} (B(s) - Z(s)) = -\infty
$, 
define 
\begin{align}
Q(Z,B)(t) &= \sup_{t \le s < \infty}\{B(t,s)-Z(t,s)\}, \label{definition of Q} \\[1em]
D(Z,B)(t) &= Z(t) + Q(Z,B)(0) - Q(Z,B)(t), \label{definition of D}\\[1em]
R(Z,B)(t) &= B(t) + Q(Z,B)(t) - Q(Z,B)(0). \label{definition of R}
\end{align}
Reverse-time analogues of these  transformations are defined for continuous functions $Y,C:\R \rightarrow \R$ satisfying $Y(0) = C(0) = 0$ and $\limsup_{s \rightarrow -\infty} (Y(s) - C(s)) = -\infty$:  
\begin{align} 
\Qa(Y,C)(t) &= \sup_{-\infty < s \le t} \{C(s,t)-Y(s,t)\}, \label{reverse definition of Q} \\[1em]
\Da(Y,C)(t) &= Y(t) + \Qa(Y,C)(t) - \Qa(Y,C)(0),\label{reverse definition of D} \\[1em]
\Ra(Y,C)(t) &= C(t) + \Qa(Y,C)(0) - \Qa(Y,C)(t).\label{reverse definition of R}
\end{align}

These transformations originate from the Brownian queue, first studied by Glynn and Whitt~\cite{glynn1991}, and further expounded on by Harrison and Williams~\cite{Harrison1985,harrison1990,harrison1992} and O'Connell and Yor~\cite{brownian_queues}. See Section~\ref{section:queue} and Appendix~\ref{section:queue and stationary} for more about the queuing interpretation.
The  following lemma is a straightforward exercise. We state it for completeness, as we will refer to it later in the paper.  
\begin{lemma} \label{queue length is continuous function of t}
Let $B,Z:\R \rightarrow \R$ be continuous functions satisfying $\limsup_{s \rightarrow \infty} (B(s) - Z(s)) = -\infty$. Then, $Q(Z,B)$, $D(Z,B)$, and $R(Z,B)$ are continuous. Similarly, if  $Y,C:\R \rightarrow \R$ are continuous functions satisfying $\limsup_{s \rightarrow -\infty}(Y(s) - C(s)) = -\infty$, then $\Qa(Y,C),\Da(Y,C)$, and $\Ra(Y,C)$ are continuous.
\end{lemma}

The next theorem summarizes the existence and properties of the Busemann process.  This process has discontinuities in the direction parameter $\theta$. Instead of a single cadlag process, it is useful to retain both a left- and a right-continuous version indicated by $\theta-$ and $\theta+$ because this distinction captures spatial limits.   

\begin{theorem} \label{thm:summary of properties of Busemanns for all theta}
There exists a process,
\[
\{\B^{\theta \sig}(\mbf x,\mbf y): \theta > 0, \, \sigg \in \{+,-\},  \, \mbf x,\mbf y \in \Z \times \R\},
\]
and for $\theta > 0$, there exist events  $\Omega^{(\theta)} \subseteq \Omega_1 \subseteq \Omega$, each of probability one, such that the following hold. Here, $\vv_{m + 1}^{\theta \sig}, \h_m^{\theta \sig}$ are defined as in~\eqref{vertical Busemann simple expression} and~\eqref{horizontal Busemann simple expression}, placing $\sigg$ in the appropriate superscripts. 
    \begin{enumerate} [label=\rm(\roman{*}), ref=\rm(\roman{*})]  \itemsep=3pt 
    \item{\rm(}Additivity{\rm)} \label{general additivity Busemanns} On $\Omega_1$, whenever $\mathbf x,\mathbf y,\mathbf z \in (\Z \times \R), \theta > 0$, and $\sigg \in \{+,-\}$,
    \[
    \B^{\theta \sig}(\mathbf x, \mathbf y) + \B^{\theta \sig}(\mathbf y,\mathbf z) = \B^{\theta\sig}(\mathbf x, \mathbf z). 
    \] 
    \item{\rm(}Monotonicity{\rm)} \label{general monotonicity Busemanns} On $\Omega_1$, whenever $0 <\gamma < \theta < \infty$, $m \in \Z$, and $s < t \in \R$,
    \[
    0 \le \vv_{m}^{\gamma -}(s) \leq \vv_{m}^{\gamma +}(s) \leq \vv_{m}^{\theta -}(s) \le \vv_{m}^{\theta +}(s), \text{ and }
    \]
    \[
    B_m(s,t) \le \h_m^{\theta +}(s,t) \le \h_m^{\theta -}(s,t) \le \h_m^{\gamma +}(s,t) \le   \h_m^{\gamma -}(s,t). 
    \]  
    \item{\rm(}Convergence{\rm)} \label{general uniform convergence Busemanns} On $\Omega_1$, for every $m \in \Z$, $\theta > 0$ and $\sigg \in \{+,-\}$, 
    \begin{enumerate} [label=\rm(\alph{*}), ref=\rm(\alph{*})]  \itemsep=3pt 
        \item  \label{general uniform convergence:limits from left} As $\gamma \nearrow \theta$,  $\h_m^{\gamma \sig} $ and $\vv_m^{\gamma \sig}$ converge uniformly, on compact subsets of $\R$, to $\h_m^{\theta -}$ and $\vv_m^{\theta -}$, respectively. 
        \item \label{general uniform convergence:limits from right} As $\delta \searrow \theta$, $\h_m^{\delta\sig}$ and $\vv_m^{\delta\sig}$ converge uniformly, on compact subsets of $\R$, to $\h_m^{\theta +}$ and $\vv_m^{\theta +}$, respectively.
        \item \label{general uniform convergence:limits to infinity}As $\gamma \rightarrow \infty$, $\h_m^{\gamma \sig}$ converges uniformly, on compact subsets of $\R$, to $B_m$. 
        \item \label{general uniform convergence:limits to 0}As $\delta \searrow 0$, $\vv_m^{\delta \sig}$ converges uniformly, on compact subsets of $\R$, to $0$. 
    \end{enumerate}   
    \item{\rm(}Continuity{\rm)} On $\Omega_1$, for any $r,m \in \Z$,  $\theta > 0$, and $\sigg \in \{+,-\}$, $(s,t) \mapsto \B^{\theta \sig}((m,s),(r,t))$ is a continuous function $\R^2 \rightarrow \R$. \label{general continuity of Busemanns}
    \item{\rm(}Limits{\rm)}  \label{limits of B_m minus \h_{m + 1}}
    On $\Omega_1$, for each $\theta > 0$ and $\sigg \in \{+,-\}$,
    \[
    \lim_{s \rightarrow \pm \infty} \bigl[ B_m(s) - \h_{m + 1}^{\theta \sig}(s)\bigr]  = \mp \infty. 
    \]
    \item{\rm(}Queuing relationships between Busemann functions{\rm)} \label{general queuing relations Busemanns} On $\Omega_1$, for all $m \in \Z,\theta > 0$, and signs $ \sigg \in \{+,-\}$,  
    \[
    \vv_{m + 1}^{\theta \sig} = Q(\h_{m + 1}^{\theta \sig},B_m) \qquad\text{and}\qquad \h_m^{\theta \sig} =D(\h_{m + 1}^{\theta \sig},B_m). 
    \] 
    \item{\rm(}Independence{\rm)} \label{independence structure of Busemann functions on levels}  
    For any $m \in \Z$,
    \[
    \{\h_r^{\theta \sig}: \theta > 0, \sigg \in \{+,-\}, r > m\} \text{ is independent of } \{B_r: r \le m\} 
    \]
\item{\rm(}Equality for fixed directions{\rm)} \label{busemann functions agree for fixed theta} Fix $\theta > 0$. Then, on the event $\Omega^{(\theta)}$, for all $\mbf x,\mbf y \in \Z \times \R$ and all sequences $\{t_n\}$ with $t_n/n \rightarrow \theta$, 
\[
\B^{\theta-}(\mbf x,\mbf y) = \lim_{n \rightarrow \infty} \bigl[L_{\mbf x,(n,t_n)} - L_{\mbf y,(n,t_n)}\bigr] =  \B^{\theta+}(\mbf x,\mbf y). 
\] 
\item{\rm(}Shift invariance{\rm)} \label{itm:shift_invariance}
For each $\mbf z \in \Z \times \R$,
\[
\{\B^{\theta \sig}(\mbf x,\mbf y):\mbf x,\mbf y \in \Z \times \R,\theta > 0, \sigg \in \{+,-\}\} \deq \{\B^{\theta \sig}(\mbf x + \mbf z,\mbf y + \mbf z):\mbf x,\mbf y \in \Z \times \R,\theta > 0, \sigg \in \{+,-\}\}. 
\]
\end{enumerate}
\end{theorem}
\begin{remark}
 On account of Part~\ref{busemann functions agree for fixed theta},  when working on the event $\Omega^{(\theta)}$ we write $\B^\theta=\B^{\theta-}= \B^{\theta+}$. The fact that the limits exist for all initial points $\mbf x,\mbf y \in \Z \times \R$ and fixed $\theta > 0$ on a single event of probability one gives a generalization of Theorem~\ref{thm:existence of Busemann functions for fixed points}. As will be seen in Section~\ref{section:Busemann construction}, the key is Lemma~\ref{lemma:equality of Busemann function distribution at rationals}, which generalizes a proof in~\cite{blpp_utah}. 
\end{remark}

 \noindent We introduce a dual field of Brownian motions used later in the proof of coalescence of semi-infinite geodesics. Let $\theta > 0$. Recall the definition \eqref{definition of R} of the mapping $R$, and on the event $\Omega^{(\theta)}$, set
\begin{equation}\label{definition of dual weights from Busemanns} 
X_m^{\theta} = R(\h_m^{\theta},B_{m - 1}). 
\end{equation}

\noindent Denote the field of these random functions by $\mbf X^\theta := \{X_m^\theta\}_{m \in \Z}$. The following theorem describes the distribution of the Busemann functions for a fixed direction $\theta$. 
\begin{theorem} \label{thm:dist of Busemann functions}
Fix $\theta > 0$. 
\begin{enumerate} [label=\rm(\roman{*}), ref=\rm(\roman{*})]  \itemsep=3pt 
    \item The process $t \mapsto \h_m^\theta(t)$ is a two-sided Brownian motion with drift $\f{1}{\sqrt \theta}$. \label{BM_drift}
    \item $\mbf X^\theta$ is a field of independent two-sided Brownian motions. For each $n \in \Z$, $\{X_m^\theta\}_{m > n}$ is independent of $\{\h_m^\theta\}_{m  \le n}$.   \label{mutual independence of the X_m}
    \item  The process  $t \mapsto \vv_m^\theta(t)$  is a stationary and reversible strong Markov process. For each $t \in \R$, $\vv_m^\theta(t) \sim \operatorname{Exp}\bigl(\f{1}{\sqrt \theta}\bigr)$. \label{v_m process}
\item{\rm(}Burke property{\rm)} Fix $m<n \in \Z$ and $-\infty < t_n \leq t_{n - 1} \leq \cdots \leq t_{m + 1} < \infty$. Then, the following random variables and processes are mutually independent:
\begin{multline*}
\{B_m(u,t_{m + 1}):u \le t_{m + 1}\}, \;\; v_{m + 1}^\theta(t_{m + 1}),\;\;
 \{h_m^\theta(t_{m + 1},u): u \ge t_{m + 1}\}, \; \\[1em]
\{B_{r}(u,t_{r + 1}): u \leq t_{r + 1}\},\; v_{r + 1}^\theta(t_{r + 1}),\;\{h_r^\theta(t_{r + 1},u): t_{r + 1} \le u \le t_r\}, \\[1em]\text{ and } \{X_r^\theta(t_r,u): u \geq t_r  \},\; \text{ for }\; m + 1 \leq r \leq n - 1,\;\; 
 \{h_n^\theta(u,t_n):u \le t_n\},\; \text{ and }\;\{X_n^\theta(t_n,u): u \ge t_n\}.
\end{multline*}
As a special case, $\{v_m^\theta(t)\}_{m \in \Z}$ is an i.i.d sequence for each $t \in \R$. Refer to Figure~\ref{fig:Independence structure for Busemanns} for clarity. \label{Burke property}
\end{enumerate}
\end{theorem}
\begin{remark}
 Using the representation of Equation~\eqref{v_m queue notation} below, many formulas for the process $t \mapsto v_m^\theta(t)$ are well known. See~\cite{BM_handbook}, page 129 and~\cite{Norros-Salminen}, specifically equations (4) and (5), for more on this process, including the transition density. However, we caution that when comparing formulas, in the setting of~\cite{BM_handbook,Norros-Salminen}, the process is \[
 t \mapsto \sup_{-\infty < s \le t}\{B(s,t) - \mu(t- s)\},
 \]
 where $B$ is a two-sided Brownian motion and $\mu > 0$. In our setting, there is a factor of $\sqrt 2$ multiplied to $B$. 
\end{remark}
 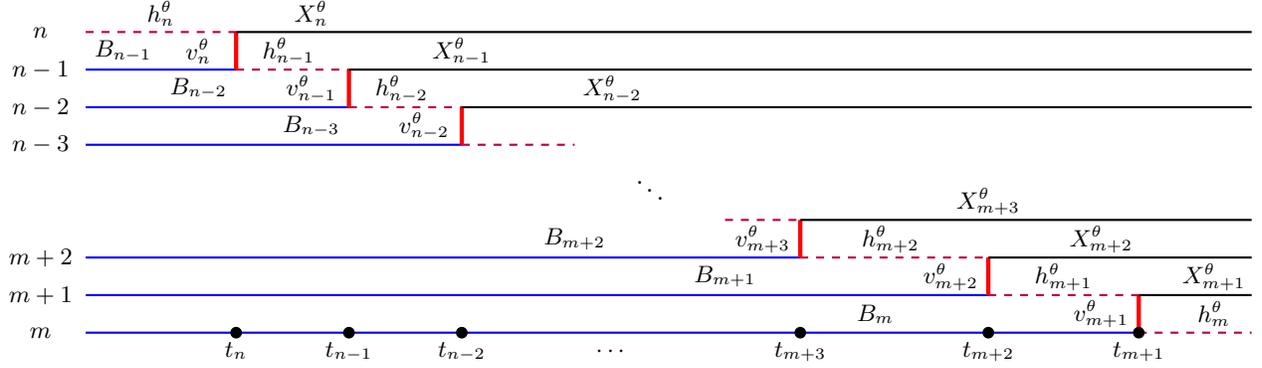
\begin{figure}[t]
\begin{tikzpicture}
\draw[blue,thick] (-0.5,0) -- (13.5,0);
\node at (-1.1,0) {\small $m$};
\node at (-1.1,0.5) {\small $m + 1$};
\node at (-1.1,1) {\small $m + 2$};
\node at (-1.1,2.5) {\small $n - 3$};
\node at (-1.1,3) {\small $n - 2$};
\node at (-1.1,3.5) {\small $n -1$};
\node at (-1.1,4) {\small $n$};
\draw[purple, thick, dashed] (13.5,0)--(15,0);
\node at (10,0.25) {\small $B_m$};
\draw[red, ultra thick] (13.5,0)--(13.5,0.5);
\node at(13,0.25) {\small $v_{m + 1}^\theta$};
\draw[blue,thick] (-0.5,0.5) --(11.5,0.5);
\draw[red, ultra thick] (11.5,0.5)--(11.5,1);
\draw[purple,thick, dashed] (11.5,0.5)--(13.5,0.5);
\draw[black,thick] (13.5,0.5)--(15,0.5);
\node at (8,0.75) {\small$B_{m + 1}$};
\node at (11,0.75) {\small$v^\theta_{m + 2}$};
\node at(12.5,0.75) {\small$h_{m +1}^\theta$};
\node at (14.5,0.75) {\small$X_{m + 1}^\theta$};
\node at (14.5,0.25) {\small $h_m^\theta$};
\draw[blue,thick] (-0.5,1) --(9,1);
\draw[purple,thick, dashed] (9,1)--(11.5,1);
\draw[black, thick] (11.5,1)--(15,1);
\draw[red, ultra thick] (9,1)--(9,1.5);
\node at (6,1.25) {\small$B_{m + 2}$};
\node at (10.2,1.25) {\small$h_{m + 2}^\theta$};
\node at (8.5,1.25) {\small$v^\theta_{m + 3}$};
\node at (13,1.25) {\small$X^\theta_{m + 2}$};
\draw[black,thick] (9,1.5)--(15,1.5);
\draw[purple, thick, dashed] (8,1.5)--(9,1.5);
\node at (11.5,1.75) {\small$X^\theta_{m + 3}$};
\draw[blue,thick] (-0.5,2.5)--(4.5,2.5);
\draw[purple, thick, dashed] (4.5,2.5)--(6,2.5);
\draw[red, ultra thick] (4.5,2.5)--(4.5,3);
\node at (4,2.75) {\small$v_{n - 2}^\theta$};
\node at(2.5,2.75) {\small$B_{n - 3}$};
\draw[blue, thick] (-0.5,3)--(3,3);
\draw[purple,thick,dashed] (3,3)--(4.5,3);
\node at (7,2) {$\ddots$};
\draw[black,thick] (4.5,3)--(15,3);
\draw[red, ultra thick] (3,3)--(3,3.5);
\node at(2.5,3.25) {\small$v^\theta_{n - 1}$};
\node at (3.7,3.25) {\small$h^\theta_{n - 2}$};
\node at (1,3.25) {\small$B_{n - 2}$};
\node at (6.5,3.25) {\small$X^\theta_{n - 2}$};
\draw[blue,thick] (-0.5,3.5)--(1.5,3.5);
\draw[purple,thick, dashed] (1.5,3.5)--(3,3.5);
\draw[black, thick] (3,3.5)--(15,3.5);
\draw[red, ultra thick] (1.5,3.5)--(1.5,4);
\node at (1,3.75) {\small$v_{n}^\theta$};
\node at (0,3.75) {\small$B_{n - 1}$};
\node at (2.2,3.75) {\small$h^\theta_{n - 1}$};
\node at (4.5,3.75) {\small$ X^\theta_{n - 1}$};
\draw[purple, thick,dashed] (-0.5,4)--(1.5,4);
\draw[black, thick] (1.5,4)--(15,4);
\node at (0.5,4.25) {\small$h_n^\theta$};
\node at (2.5,4.25) {\small$X_n^\theta$};
\filldraw[black] (1.5,0) circle (2pt) node[anchor = north] {\small $t_n$};
\filldraw[black] (3,0) circle (2pt) node[anchor = north] {\small $t_{n -1}$};
\filldraw[black] (4.5,0) circle (2pt) node[anchor = north] {\small $t_{n - 2}$};
\node at (6.5,-0.25) {\small $\cdots$};
\filldraw[black] (9,0) circle (2pt) node[anchor = north] {\small $t_{m + 3}$};
\filldraw[black] (11.5,0) circle (2pt) node[anchor = north] {\small $t_{m + 2}$};
\filldraw[black] (13.5,0) circle (2pt) node[anchor = north] {\small $t_{m + 1}$};
\end{tikzpicture}
\caption{\small Independence structure for Busemann functions. Each process $h^\theta_r$ is associated to the (purple/dotted) segment on level $r$, processes $B_r$ and $X^\theta_r$ cover the remaining portions of horizontal level $r$, and the process $v_r^\theta$ is associated to the (red/vertical) edge from level $r-1$ to $r$ at time point $t_r$.  }
\label{fig:Independence structure for Busemanns}
\end{figure}

The construction of the Busemann process and the proof of Theorem~\ref{thm:summary of properties of Busemanns for all theta} can be found in Section~\ref{section:Busemann construction}. We prove Theorem~\ref{thm:dist of Busemann functions} here, assuming Theorem~\ref{thm:summary of properties of Busemanns for all theta} and with the help of the results of the appendix. 
\begin{proof}[Proof of Theorem~\ref{thm:dist of Busemann functions}]
\noindent \textbf{Part~\ref{BM_drift}}: As will be seen from the construction in Section~\ref{section:Busemann construction}, $t \mapsto h_m^\theta(t)$ has the proper finite-dimensional distributions and is continuous by Theorem~\ref{thm:summary of properties of Busemanns for all theta}\ref{general continuity of Busemanns}.

\medskip \noindent \textbf{Part~\ref{mutual independence of the X_m}}: 
Fix integers $p<n<m$.  
By Theorem~\ref{thm:summary of properties of Busemanns for all theta}\ref{independence structure of Busemann functions on levels}, 
\[  (B_{p},\ldots,B_n, \dots, B_{m - 2},  B_{m - 1}, \h_m^\theta) \quad\text{are independent. } \]  
Theorem~\ref{thm:summary of properties of Busemanns for all theta}\ref{general queuing relations Busemanns} and definition \eqref{definition of dual weights from Busemanns} give   $\h_{m - 1}^\theta = D(\h_m^\theta,B_{m - 1})$ and  $X_m^\theta= R(\h_m^\theta,B_{m - 1})$,  so  by Theorem~\ref{O Connell Yor BM independence theorem queues}, $\h_{m - 1}^\theta$ and $X_m^\theta$ are independent  
and  $X_m^\theta$ is a two-sided Brownian motion. 
In particular, now  

 \[  (B_{p},\ldots,B_n, \dots, B_{m - 2}, \h_{m - 1}^\theta, X_m^\theta) \quad\text{are independent. } \] 
 Continue inductively by applying the  transformation $(D,R)$  to successive pairs $(B_{j-1}, \h_{j }^\theta)$ for $j= m-1, m-2,\dotsc, n+1$ after which  
 \[  (B_{p},\ldots, B_{n-1}, \h_{n}^\theta, X_{n+1}^\theta, \dotsc, X_m^\theta) \quad\text{are independent. } \]  
To conclude, note that $(\h_{p}^\theta,\ldots,   \h_{n}^\theta)$ is a function of 
$(B_{p},\ldots, B_{n-1}, \h_{n}^\theta)$  through iteration of $\h_{k}^\theta = D(\h_{k+1}^\theta,B_{k})$. 

\medskip \noindent \textbf{Part~\ref{v_m process}}:  By Theorem~\ref{thm:summary of properties of Busemanns for all theta}\ref{general queuing relations Busemanns} and   \eqref{definition of dual weights from Busemanns}, on the event $\Omega^{(\theta)}$ we have these relations
$\forall m\in\Z$:
 \be\label{380}   \h_{m - 1}^{\theta } = D(\h_m^{\theta},B_{m - 1}),  \qquad  X_m^{\theta} = R(\h_m^{\theta },B_{m - 1}), \qquad\text{and} \qquad \vv_m^{\theta }= Q(\h_m^{\theta},B_{m - 1}). \ee
 The fact that $v_m^\theta(t)$ is exponential with rate $\f{1}{\sqrt \theta}$ then follows from Lemma~\ref{lemma:sup of BM with drift}.
   Theorem~\ref{bijectivity of D R joint mapping} allows us to reverse these mappings, so $\forall m\in\Z$:
\begin{align}\label{eqn:dual weights reverse relations}
\h_m^{\theta} = \Da(\h_{m - 1}^{\theta},X_m^{\theta}),\qquad B_{m - 1} = \Ra(\h_{m - 1}^{\theta },X_m^{\theta}), \qquad\text{and}\qquad \vv_m^{\theta} = \Qa(\h_{m - 1}^{\theta },X_m^{\theta }).
\end{align}
Then, for $t \in \R$,
\be \label{v_m queue notation}
v_m^\theta(t) = \sup_{-\infty < u \le t}\{X_m^\theta(u,t) - h_{m - 1}^\theta(u,t)\}. 
\ee
By Parts~\ref{BM_drift} and~\ref{mutual independence of the X_m}, $t \mapsto X_m^\theta(t) - h_{m - 1}^\theta(t)$ is equal in distribution to a two-sided Brownian motion with negative drift, multiplied by a factor of $\sqrt{2}$. Represented this way, $t \mapsto v_m^\theta(t)$ is known as a stationary, reflected Brownian motion with drift. Stationarity follows from the stationarity of increments.  The fact that $X$ is a reversible strong Markov process is proven in~\cite{Harrison1985}, pg. 81 (see also pg. 49-50 in~\cite{Harrison1985} and Equations (4) and (5) in~\cite{Norros-Salminen} for a more directly applicable statement).

\medskip \noindent \textbf{Part~\ref{Burke property}}: By Part~\ref{mutual independence of the X_m} and~\eqref{eqn:dual weights reverse relations}, for any initial level $m$, the process $\bigl\{\h_{r + m}^\theta,\vv_{r +m +  1}^\theta, X_{r +m +  1}^\theta,B_{r + m}\bigr\}_{r \ge 0}$
    has the same distribution as
    $\bigl\{Y_{r}^{\f{1}{\sqrt \theta}}, q_{r + 1}^{\f{1}{\sqrt \theta}},B_{r + 1},W_r^\lambda\bigr\}_{r \ge 0}$ 
    as defined in~\eqref{eqn:stationary BLPP definitions}. Note that $B_{r + m}$ now plays the role of the $W_r^\lambda$, as stated in the definition. Therefore, the independence structure of Theorem~\ref{Burke-type Theorem} holds.
\end{proof}

\section{Construction and properties of the semi-infinite geodesics} \label{sec:SIG_constr}
\subsection{Heuristic for construction of semi-infinite geodesics} \label{section:geodesic construction intro}
The next task is the construction of  semi-infinite geodesics  from each initial point and   in each asymptotic direction. For each given point $(m,t) \in \Z \times \R$ and direction parameter $\theta > 0$, we want to find a semi-infinite geodesic, defined by jump times $t = \tau_{m - 1} \le \tau_m \le \cdots$ that satisfies
\[
\lim_{n \rightarrow \infty} \f{\tau_n}{n} = \theta.
\]

We argue heuristically to motivate  the useful construction. 
Start by finding a maximal path for $L_{(m,t),(n,n\theta)}$ for a large value of $n$. Note that
\[
L_{(m,t),(n,n\theta)} = \max_{s \in[t,n\theta]} \bigl(B_{m}(t,s) + L_{(m + 1,s),(n,n\theta)}\bigr),
\]
and the maximizer $s = \tau_{m}$ is the location where   the geodesic jumps from level $m$ to $m + 1$. For all $t \le s \le n\theta$,
\[
B_{m}(t,s) + L_{(m + 1,s),(n,n\theta)} \le B_{m}(t,\tau_{m}) + L_{(m + 1,\tau_{m}),(n,n\theta)}.
\]
Rearranging yields
\[
B_{m}(\tau_{m},s) \le L_{(m + 1,\tau_{m}),(n,n\theta)} - L_{(m + 1,s),(n,n\theta)}.
\]
As $n$ changes, so could $\tau_m$, but for the sake of heuristic we hold $\tau_m$ constant. Take limits as $n \rightarrow \infty$ and rearrange again to get, for some sign $\sigg \in \{+,-\}$, 
\[
 B_{m}(s) - \h_{m + 1}^{\theta \sig}(s)  \le B_{m}(\tau_{m}) - \h_{m + 1}^{\theta\sig}(\tau_{m}).  
\]

\subsection{Busemann geodesics}\label{sec:Buse_geod}
The discussion of the previous section   motivates this rigorous definition.  


\begin{definition} \label{def:semi-infinite geodesics}
On the event $\Omega_1$, for all $(m,t) \in \Z$, $\theta>0$ and  $\sig \in \{+,-\}$, let $\mbf T^{\theta\sig}_{(m,t)}$ denote the set of sequences 
\[
t = \tau_{m - 1} \le \tau_m \le \tau_{m + 1}\le \cdots
\]
that satisfy
\begin{equation} \label{semi-infinite geodesic succesive jumps}
 B_{r}(\tau_{r})- \h_{r + 1}^{\theta \sig}(\tau_{r})  = \sup_{s \in [\tau_{r - 1},\infty)}\{B_r(s) - \h_{r + 1}^{\theta \sig}(s) \} \qquad\text{for each $r \ge m$.}  
\end{equation} 
Theorem~\ref{thm:summary of properties of Busemanns for all theta}\ref{general continuity of Busemanns}--\ref{limits of B_m minus \h_{m + 1}} imply that such sequences exist.  
At each level $r$, there exist leftmost and rightmost maximizers.  Let 
\[
t = \tau_{(m,t),m - 1}^{\theta \sig,L} \le \tau_{(m,t),m}^{\theta\sig,L} \le \tau_{(m,t),m + 1}^{\theta\sig,L} \le \cdots \qquad\text{and}\qquad t = \tau_{(m,t),m - 1}^{\theta \sig,R} \le  \tau_{(m,t),m}^{\theta\sig,R} \le \tau_{(m,t),m + 1}^{\theta\sig,R} \le \cdots
\]
denote the leftmost and rightmost sequences in $\mbf T_{(m,t)}^{\theta\sig}$. Furthermore, define
\[
\mbf T_{(m,t)}^{\theta} := \mbf T_{(m,t)}^{\theta +} \cup \mbf T_{(m,t)}^{\theta -}. 
\]
\end{definition}
\begin{remark} \label{rmk:theta+ = theta-}
Since every non-decreasing sequence in $\mbf T_{(m,t)}^{\theta}$ defines a semi-infinite up-right path, $\mbf T_{(m,t)}^\theta$ will be used to denote the set of up-right paths constructed in this way. Theorem~\ref{existence of semi-infinite geodesics intro version}\ref{energy of path along semi-infinte geodesic} below shows that all sequences in $\mbf T_{(m,t)}^\theta$ are semi-infinite geodesics starting from $(m,t)$. By Theorem~\ref{thm:summary of properties of Busemanns for all theta}\ref{busemann functions agree for fixed theta}, on the event $\Omega^{(\theta)} \subseteq \Omega_1$, $\mbf T_{\mbf x}^{\theta -} = \mbf T_{\mbf x}^{\theta +} = \mbf T_{\mbf x}^{\theta}$ for all $\mbf x \in \Z \times \R$. This does not imply that $\mbf T_{\mbf x}^\theta$ contains only one element, so the leftmost and rightmost distinction is still necessary in general for a fixed direction $\theta$. See Theorem~\ref{thm:non_unique_size} below. However, by Theorem~\ref{thm:general_SIG}\ref{itm:uniqueness of geodesic for fixed point and direction}, for fixed $\theta > 0$ and fixed $
(m,t)\in \Z \times \R$, $\mbf T_{(m,t)}^\theta$ almost surely contains a single element. In this case, Theorem~\ref{distribtution of argmax BM with drift} gives the distribution of the first jump time $\tau_m$.  
\end{remark}

\begin{figure}[t]
\begin{tikzpicture}
\draw[gray,thin] (0.5,0) -- (15.5,0);
\draw[gray,thin] (0.5,0.5) --(15.5,0.5);
\draw[gray, thin] (0.5,1)--(15.5,1);
\draw[gray,thin] (0.5,1.5)--(15.5,1.5);
\draw[gray,thin] (0.5,2)--(15.5,2);
\draw[red, ultra thick,->] (1.5,0)--(4.5,0)--(4.5,0.5)--(7,0.5)--(7,1)--(9.5,1)--(9.5,1.5)--(13,1.5)--(13,2)--(15.5,2);
\filldraw[black] (1.5,0) circle (2pt) node[anchor = north] {$(m,t)$};
\node at (4.5,-0.5) {$\tau_{m}^\theta$};
\node at (7,-0.5) {$\tau_{m + 1}^\theta$};
\node at (9.5,-0.5) {$\tau_{m + 2}^\theta$};
\node at (13,-0.5) {$\cdots$};
\node at (0,0) {$m$};
\node at (0,0.5) {$m + 1$};
\node at (0,1) {$m + 2$};
\node at (0,1.5) {$\vdots$};
\end{tikzpicture}
\caption{\small Example of an element of $\mbf T_{(m,t)}^\theta$}
\label{fig:BLPP_semi-infinite_geodesic}
\end{figure}
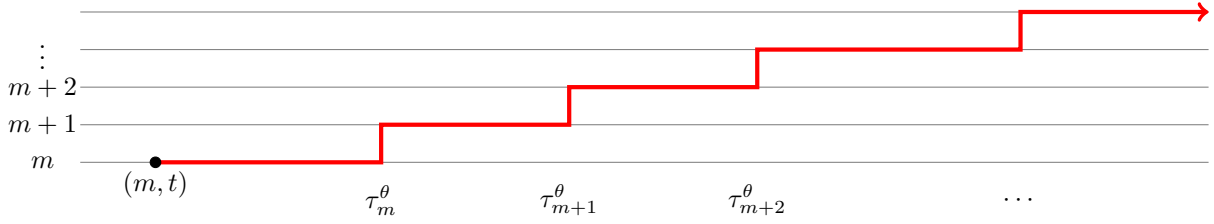

Two elements $\{\tau_r\}_{r\ge m-1}$ and  $\{\tau'_r\}_{r\ge m-1}$  of $\mbf T_{\mbf x}^{\theta \sig}$ are distinct if $\tau_r\ne\tau'_r$ for at least one index $r$.  Uniqueness of the $\theta \sig$ Busemann geodesic from $\mbf x$ means that $\mbf T_{\mbf x}^{\theta \sig}$ contains exactly one sequence.

The following theorems collect the properties of the Busemann semi-infinite geodesics, to be  proved in Section~\ref{section:main proofs}.
\begin{theorem} \label{existence of semi-infinite geodesics intro version}
There exists an event, $\Omega_2$, of full probability, on  which the following hold. 
 \begin{enumerate} [label=\rm(\roman{*}), ref=\rm(\roman{*})]  \itemsep=3pt 
     \item{\rm(}Existence{\rm)} \label{energy of path along semi-infinte geodesic} For all $\mbf x \in \Z \times \R, \theta > 0$, and $\sigg \in \{+,-\}$, every element of $\mbf T_{\mbf x}^{\theta \sig}$ defines a semi-infinite geodesic starting from $\mbf x$. More specifically, for any two points $\mbf y \le \mbf z$ along a path in $\mbf T_{\mbf x}^{\theta}$, the energy of this path between $\mbf y$ and $\mbf z$ is $\B^{\theta\sig}(\mbf y,\mbf z)$, and this energy is maximal over all paths between $\mbf y$ and $\mbf z$. 
     \item{\rm(}Leftmost and rightmost finite geodesics along paths{\rm)} \label{Leftandrightmost} If, for some $\theta > 0$, $\sigg \in \{+,-\}$, and $\mbf x \in \Z \times \R$, the points $\mbf y \le \mbf z \in \Z \times \R$ both lie on the leftmost semi-infinite geodesic in $\mbf T_{\mbf x}^{\theta \sig}$, then the portion of this geodesic between $\mbf y$ and $\mbf z$ coincides with  the leftmost finite geodesic between these two points. Similarly, the rightmost semi-infinite geodesic is the rightmost geodesic between any two of its points.  
     \item{\rm(}Monotonicity{\rm)} \label{itm:monotonicity of semi-infinite jump times} The following inequalities hold.
     \begin{enumerate} [label=\rm(\alph{*}), ref=\rm(\alph{*})]  \itemsep=3pt 
         \item \label{itm:monotonicity in theta} For all $0 < \gamma < \theta$, all $(m,t) \in \Z \times \R$, and $r \ge m$,
    \[
     t \le \tau_{(m,t),r}^{\gamma -,L} \le \tau_{(m,t),r}^{\gamma +,L} \le \tau_{(m,t),r}^{\theta -,L} \le \tau_{(m,t),r}^{\theta +,L}  \qquad\text{and}\qquad t \le \tau_{(m,t),r}^{\gamma -,R} \le \tau_{(m,t),r}^{\gamma +,R} \le \tau_{(m,t),r}^{\theta -,R} \le \tau_{(m,t),r}^{\theta +,R}.
    \]
    \item\label{itm:monotonicity in t} For all $\theta > 0$, $m \le r \in \Z$, $s < t \in \R$, and $\sig \in \{+,-\}$, 
    \[
    \tau_{(m,s),r}^{\theta \sig,L} \le \tau_{(m,t),r}^{\theta\sig,L}\qquad \text{and} \qquad \tau_{(m,s),r}^{\theta \sig,R} \le \tau_{(m,t),r}^{\theta \sig, R}.
    \] 
     \item \label{itm:strong monotonicity in t} For $\theta > 0 $, on the $\theta$-dependent full-probability event $\wt \Omega^{(\theta)}$ of Theorem~\ref{thm:general_SIG}\ref{itm:size of non-uniqueness},
    for all pairs of  initial points $(m,s)$ and $(m,t)$ in $\Z\times\R$ that satisfy $s<t$,
      we have 
    \[
    \tau_{(m,s),r}^{\theta,R} \le \tau_{(m,t),r}^{\theta,L} \quad \text{ for all $r \ge m$.} 
    \]
     \end{enumerate}
    
    \item{\rm(}Convergence{\rm)} \label{itm:convergence of geodesics}
    The following limits hold.
    \begin{enumerate} [label=\rm(\alph{*}), ref=\rm(\alph{*})]  \itemsep=3pt
         \item \label{itm:limits in theta}
         For all $(m,t) \in \Z \times \R$, $r \ge m$, $\theta > 0$, and $\sigg \in \{+,-\}$, 
        \[
     \lim_{\gamma \nearrow \theta} \tau_{(m,t),r}^{\gamma \sig,L} = \tau_{(m,t),r}^{\theta -,L}\qquad\text{and}\qquad \lim_{\delta \searrow \theta} \tau_{(m,t),r}^{\delta \sig,R} = \tau_{(m,t),r}^{\theta +,R}. 
    \]
    \item \label{itm:limits in theta to infty}
    For all $(m,t) \in \Z \times \R$, $r \ge m$, $\sigg \in \{+,-\}$, and $S \in \{L,R\}$,
    \[
    \lim_{\theta \searrow 0} \tau_{(m,t),r}^{\theta \sig,S} = t\qquad\text{and}\qquad\lim_{\theta \rightarrow \infty} \tau_{(m,t),r}^{\theta \sig,S} = \infty. 
    \]
    \item \label{itm:limits in t} For all $(m,t) \in \Z \times \R$, $r \ge m$, $\theta > 0$, and $\sigg \in \{+,-\}$, 
    \[
    \lim_{s \nearrow t} \tau_{(m,s),r}^{\theta \sig,L} =  \tau_{(m,t),r}^{\theta \sig,L}\qquad\text{and}\qquad \lim_{u \searrow t} \tau_{(m,u),r}^{\theta \sig,R} = \tau_{(m,t),r}^{\theta \sig,R},
    \]
     \end{enumerate}
    \item \label{general limits for semi-infinite geodesics} {\rm(}Directness{\rm)} For all $\mbf x \in \Z \times \R$, $\theta > 0$, $\sigg \in \{+,-\}$, and all $\{\tau_r\}_{r \ge m} \in \mbf T_{\mbf x}^{\theta \sig}$, 
    \[
    \lim_{n\rightarrow \infty} \f{\tau_n}{n} = \theta. 
    \]
    \end{enumerate}
    \end{theorem}
    \begin{remark}[A look ahead]  
In future work, we will use the joint distribution of Busemann functions to build on these results and strengthen parts of Theorem~\ref{existence of semi-infinite geodesics intro version}. Specifically, Part~\ref{itm:convergence of geodesics}\ref{itm:limits in theta} can be made stronger in the following way: There exists an event of full probability on which, for all $\theta > 0, m \le r \in \Z,s < t \in \R, S \in \{L,R\}$ and $\sigg \in \{+,-\}$, there exists $\ve > 0$ such that 
\[
\tau_{(m,s),r}^{\gamma \sig,S} = \tau_{(m,s),r}^{\theta -,S} \text{ for all }\theta - \ve < \gamma < \theta,\qquad\text{and}\qquad \tau_{(m,s),r}^{\delta \sig,S} = \tau_{(m,s),r}^{\theta +,S} \text{ for all }\theta < \delta < \theta + \ve.
\]
This is used to strengthen Part~\ref{itm:monotonicity of semi-infinite jump times}\ref{itm:strong monotonicity in t} to show that, on this event, for each $m \in \Z$, $m \le r \in \Z, s < t \in \R, \theta > 0$, and $\sigg \in \{+,-\}$,
\[
\tau_{(m,s),r}^{\theta \sig,R} \le \tau_{(m,t),r}^{\theta \sig,L}.
\]
Part~\ref{itm:monotonicity of semi-infinite jump times}\ref{itm:monotonicity in theta} cannot be strengthened to compare $\tau_{(m,t),r}^{\gamma \sig,R}$ and $\tau_{(m,t),r}^{\theta \sig,L}$ for general $\gamma < \theta$. Specifically, there exists $\gamma < \theta$ such that
\[
\tau_{(m,t),r}^{\gamma -,R} = \tau_{(m,t),r}^{\gamma +,R} > \tau_{(m,t),r}^{\theta -,L} = \tau_{(m,t),r}^{\theta +,L}.
\]
\end{remark}
    The following theorem shows that the Busemann geodesics give control over all semi-infinite geodesics.
    \begin{theorem}  \label{thm:convergence and uniqueness}
   The following hold on the full probability event $\Omega_2$.
    \begin{enumerate}[label=\rm(\roman{*}), ref=\rm(\roman{*})]  \itemsep=3pt
    \item {\rm(}Control of finite geodesics{\rm)} \label{control of finite geodesics} Let $\theta > 0, (m,t) \in \Z \times \R$, and let $\{t_n\}$ be any sequence that has direction $\theta$. For all $n$ sufficiently large so that $n \ge m$ and $t_n \ge t$, let $t = t_{n,m - 1} \le t_{n,m} \le \cdots \le t_{n,n} = t_n$ be any sequence that defines a {\rm(}finite{\rm)} geodesic between $(m,t)$ and $(n,t_n)$. Then, for each $r \ge m$,
    \[
    \tau_{(m,s),r}^{\theta-,L} \le \liminf_{n \rightarrow \infty} t_{n,r} \le \limsup_{n \rightarrow \infty} t_{n,r} \le \tau_{(m,s),r}^{\theta+,R}. 
    \]
    \item{\rm(}Control of semi-infinite geodesics{\rm)} \label{all semi-infinite geodesics lie between leftmost and rightmost} If, for some $\theta > 0$ and $(m,t) \in \Z \times \R$, any other geodesic {\rm(}constructed from the Busemann functions or not{\rm)} is defined by the sequence $t = t_{m - 1} \le t_m \le \cdots$, starts at $(m,t)$, and has direction $\theta$, then for all $r \ge m$, 
    \[
    \tau_{(m,t),r}^{\theta -,L} \le t_r \le \tau_{(m,t),r}^{\theta +,R}.  
    \]
    \item{\rm(}Convergence of finite geodesics{\rm)} \label{convergence to unique semi-infinite geodesic}  Assume that $\mbf T_{\mbf (m,t)}^\theta$ contains a single element  $\{\tau_r\}_{r \ge m - 1}$. If $\{t_n\}$ is a $\theta$-directed sequence and, for each $n$, the sequence $t =t_{n,m - 1} \le t_{n,m} \le \cdots \le t_{n,n}$ defines a finite geodesic between $(m,t)$ and $(n,t_n)$, then
    \[
    \lim_{n \rightarrow \infty} t_{n,r} = \tau_r\qquad \text{for all } r \ge m. 
    \]
    \end{enumerate}
    \end{theorem}
    \begin{remark}
    In Part~\ref{convergence to unique semi-infinite geodesic}, the assumption of uniqueness holds, for example, on the event $\Omega_{(m,t)}^{(\theta)}$ of Theorem~\ref{thm:general_SIG}\ref{itm:uniqueness of geodesic for fixed point and direction}. However, this assumption does not extend to all $(m,t)\in \Z \times \R$ and $\theta > 0$ simultaneously with probability one, as discussed in the following section.  
    \end{remark}

\subsection{Non-uniqueness of semi-infinite geodesics}   \label{sec:non_unique}
There are two types of non-uniqueness of semi-infinite geodesics from an initial point $\mbf x$ into an asymptotic direction $\theta$. 
\smallskip 

(i) The first type, described in Theorem~\ref{thm:general_SIG}\ref{itm:size of non-uniqueness} and  in the next  Theorem~\ref{thm:non_unique_size},  is caused by the continuum time variable and does {\it not} appear in the lattice corner growth model.   It is captured by the $L/R$ distinction.  For each fixed direction $\theta>0$ and level $m \in \Z$, this happens with probability one at infinitely many locations. 
To illustrate, let $s^\star \ge 0$ be the maximizer below:  
\[
B_m(s^\star) - \h_{m +1}^\theta(s^\star) = \sup_{0 \le s < \infty}\{B_m(s) - \h_{m +1}^\theta(s)\}.
\]
By Theorem~\ref{thm:summary of properties of Busemanns for all theta}\ref{independence structure of Busemann functions on levels} and Lemma~\ref{lm:point-to-line uniqueness}, with probability one, the maximizer $s^\star$ is unique. By Theorem~\ref{distribtution of argmax BM with drift}, $s^\star > 0$ with probability one. By Theorem~\ref{thm:summary of properties of Busemanns for all theta}\ref{limits of B_m minus \h_{m + 1}},
\[
t^\star := \sup\{t < 0: B_m(t) - \h_{m + 1}^\theta(t) = B_m(s^\star) - \h_{m +1}^\theta(s^\star)\} \qquad\text{exists in }\R_{<0}.
\]
Then, both $t^\star$ and $s^\star$ are maximizers of $B_m(s) - \h_{m + 1}^\theta(s)$ on $[t^\star,\infty)$. This gives at least two distinct sequences in the set $\mbf T_{(m,t^\star)}^\theta$, with $\tau_{(m,t^\star),m}^{\theta,L}=t^\star$ and $\tau_{(m,t^\star),m}^{\theta,R}=s^\star$. 

This presents a new type of non-uniqneness that is not present in discrete last-passage percolation with exponential weights. However, when $\theta$ is fixed,   $\theta$-directed geodesics can disagree only for a finite amount of time, because  Theorem~\ref{thm:general_SIG}\ref{itm:coalescence} forces them to eventually come back together. 

\smallskip 

(ii) The second type of non-uniqueness of semi-infinite geodesics is captured by the $\theta \pm$ distinction.   Hence, it happens with probability zero at a fixed $\theta$ and thereby requires investigation of the full Busemann process and the full collection of all semi-infinite geodesics.  In contrast to the first type of non-uniqueness, this bears some similarity to the behavior present in discrete last-passage percolation shown in~\cite{Geometry_of_Geodesics}.

In future work, we show that there exists a random countable set of directions $\theta$ such that, out of every initial point, there are two $\theta$-directed geodesics. These geodesics may initially  stay together for a while, but eventually they separate for good and never come back together. Furthermore, there is a distinguished subset of initial  points at which geodesics with the same direction split immediately. This set will be shown to have  almost surely Hausdorff dimension $\f{1}{2}$. 

The following theorem clarifies the non-uniqueness described by (i) above. Fix $\theta > 0$. On a full probability event where the $\theta\pm$ distinction is not present, define the following sets:
\begin{align*}
\NU_0^\theta &= \{(m,t) \in \Z \times \R: \tau_{(m,t),r}^{\theta,L} < \tau_{(m,t),r}^{\theta ,R} \text{ for some }r \ge m\}, \qquad \text{and} \\
\NU_1^\theta &= \{(m,t) \in \NU_0^\theta: \tau_{(m,t),m}^{\theta,L} < \tau_{(m,t),m}^{\theta,R} \}.
\end{align*}
Since $\theta > 0$ is fixed, by Theorem~\ref{thm:convergence and uniqueness}\ref{all semi-infinite geodesics lie between leftmost and rightmost}, $\NU_0^\theta$ is almost surely the set of points $\mbf x \in \Z \times \R$ such that the $\theta$-directed semi-infinite geodesic from $\mbf x$ is not unique.  Its subset $\NU_1^\theta$ is the set of   initial points from which  two $\theta$-directed geodesics separate on the first level.  
    \begin{theorem} \label{thm:non_unique_size}
    There exists a full probability event $\wt \Omega^{(\theta)}$ on which the following hold. 
    \begin{enumerate} [label=\rm(\roman{*}), ref=\rm(\roman{*})]  \itemsep=3pt
    \item \label{decomposition} The sets $\NU_0^\theta$ and $\NU_1^\theta$ are countably infinite and can be written as 
    \begin{align*}
    \NU_0^\theta &= \{(m,t) \in \Z \times \R: t = \tau_{(m,t),r}^{\theta,L} < \tau_{(m,t),r}^{\theta,R} \text{ for some }r \ge m\}, \qquad\text{and}\\[1em]
    \NU_1^\theta &= \{(m,t) \in \NU_0^\theta: t = \tau_{(m,t),m}^{\theta,L} < \tau_{(m,t),m}^{\theta,R} \}.
    \end{align*}
    For each $(m,t) \in \Z \times \R$ and $\theta > 0$,  at most one geodesic in $\mbf T_{(m,t)}^{\theta}$ passes horizontally through $(m,t + \ve)$ for some $\ve > 0$.
    \item \label{non-discrete or dense} The set $\NU_1^\theta$ is neither discrete nor dense in $\Z \times \R$. More specifically, for each point $(m,t) \in \NU_1^\theta$ and every $\ve > 0$, there exists $s \in (t - \ve,t)$ such that $(m,s) \in \NU_1^\theta$. For each $(m,t) \in \NU_1^\theta$, there exists $\delta > 0$ such that, for all $s \in (t,t+\delta)$, $(m,s) \notin \NU_0^\theta$. 
    \end{enumerate}
    \end{theorem}
    \begin{remark}
    Part~\ref{decomposition} states that, on $\wt \Omega^{(\theta)}$, if there exist multiple $\theta$-directed geodesics out of $(m,t)$, then these geodesics separate one by one from the upward vertical ray at $(m,t)$. 
    The set $\NU_1^\theta$ is the subset of $\NU_0^\theta$ such that two geodesics separate immediately at the initial point. See Figures~\ref{fig:NU_0} and~\ref{fig:NU_1}.
    \end{remark}
    \begin{figure}[h!]
 \centering
            \begin{tikzpicture}
            \draw[gray,thin] (0,0)--(10,0);
            \draw[gray,thin] (0,1)--(10,1);
            \draw[gray,thin] (0,2)--(10,2);
            \draw[gray,thin] (0,3)--(10,3);
            \draw[gray,thin] (0,4)--(10,4);
            \draw[gray,thin] (0,5)--(10,5);
            \draw[red,ultra thick] plot coordinates {(0,2)(5,2)(5,3)(6,3)(7,3)(7,4)(8,4)(8,5)};
            \draw[red,ultra thick,->] plot coordinates {(0,0)(0,4)(3,4)(3,5)(10,5)};
            \filldraw[black] (5,2) circle (2pt) node[anchor = north] {$(r,\tau_{(m,t),r}^{\theta,R})$};
            \filldraw[black] (0,2) circle (2pt) node[anchor = east] {$(r,t) = (r,\tau_{(m,t),r}^{\theta,L})$};
            \filldraw[black] (0,0) circle (2pt) node[anchor = north] {$(m,t)$};
            \end{tikzpicture}
            \caption{\small In this figure, $(m,t) \in \NU_0^\theta \setminus \NU_1^\theta$. The two geodesics split on the vertical line containing the initial point, but they must come back together to coalesce.}
            \label{fig:NU_0}
            \end{figure}
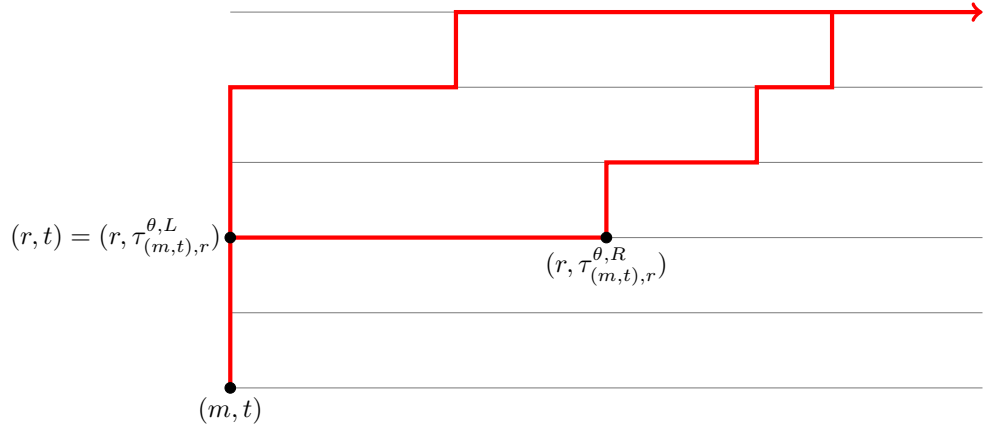
            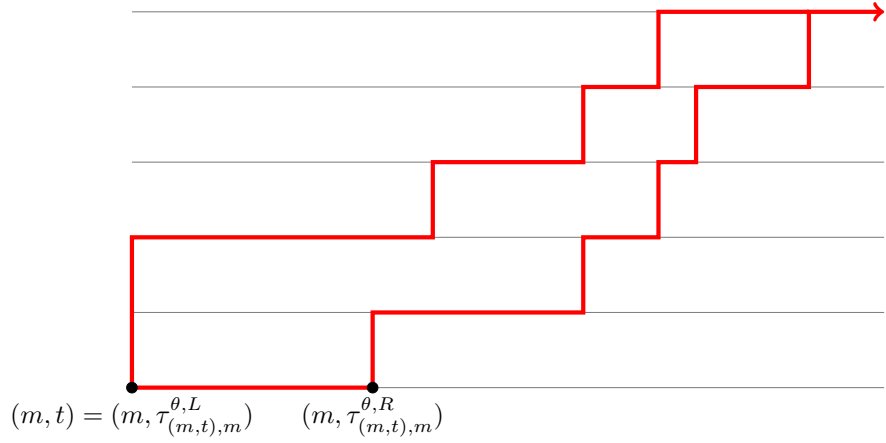
\begin{figure}[h!]
            \centering
            \begin{tikzpicture}
            \draw[gray,thin] (0,0)--(10,0);
            \draw[gray,thin] (0,1)--(10,1);
            \draw[gray,thin] (0,2)--(10,2);
            \draw[gray,thin] (0,3)--(10,3);
            \draw[gray,thin] (0,4)--(10,4);
            \draw[gray,thin] (0,5)--(10,5);
            \draw[red,ultra thick,->] plot coordinates {(0,0)(3.2,0)(3.2,1)(6,1)(6,2)(7,2)(7,3)(7.5,3)(7.5,4)(9,4)(9,5)(10,5)};
            \draw[red,ultra thick] plot coordinates {(0,0)(0,2)(4,2)(4,3)(5,3)(6,3)(6,4)(7,4)(7,5)(9,5)};
            \filldraw[black] (0,0) circle (2pt) node[anchor = north] {$(m,t) = (m,\tau_{(m,t),m}^{\theta,L})$};
            \filldraw[black] (3.2,0) circle (2pt) node[anchor = north] {$(m,\tau_{(m,t),m}^{\theta,R})$};
            \end{tikzpicture}
            \caption{\small In this figure, $(m,t) \in \NU_1^\theta$. The two geodesics split immediately from the initial point, but later coalesce. }
            \label{fig:NU_1}
            \bigskip
        \end{figure}

\subsection{Dual geodesics and coalescence} \label{sec:intro_dual_geod} To prove the coalescence of Theorem~\ref{thm:general_SIG}\ref{itm:coalescence}, we use the dual field $\mbf X^\theta$ of independent Brownian motions from~\eqref{definition of dual weights from Busemanns} and their southwest-directed semi-infinite geodesics. We use this to construct the BLPP analogue of Pimentel's dual tree~\cite{pimentel2016} and then adapt the argument of~\cite{Timo_Coalescence}. 

Since Brownian motion is symmetric, in distribution, about the origin, there exist $\theta$-directed dual southwest semi-infinite geodesics for the environment $\mbf X^\theta$. These are constructed in a very similar manner as the northeast geodesics in Definition~\ref{def:semi-infinite geodesics}. Specifically, for $(m,t) \in \Z \times \R$, let $\mbf T_{(m,t)}^{\theta, \star}$ be the set of sequences $t  = \tau_{m}^\star \ge \tau_{m - 1}^\star \ge \cdots$ satisfying
\begin{equation} \label{eqn:def of southwest geodesics}
\h_{r - 1}^{\theta}(\tau_{r - 1}^\star) - X_{r}^{\theta}(\tau_{r - 1}^\star) = \sup_{-\infty < s \le \tau_{r}^\star}\{ \h_{r - 1}^{\theta}(s) - X_{r}^{\theta}(s)\}\qquad\text{for each }r \le m.
\end{equation}
Define the leftmost and rightmost sequences similarly, by $\tau_{(m,t),r}^{\theta,L \star}$ and $\tau_{(m,t),r}^{\theta,R \star}$. These sequences define southwest semi-infinite paths, similar as for the northeast paths. We graphically represent southwest paths on the plane, where the continuous coordinate is not changed, but the discrete coordinate is shifted down by $\f{1}{2}$. That is, $\text{for }m \in \Z, \text{ denote }m^\star = m-\f{1}{2}$, and for $\mbf x = (m,t) \in \Z \times \R, \text{ denote } \mbf x^\star = \bigl(m - \f{1}{2},t\bigr)$. Then, for $\{\tau_r^\star\}_{r \le m} \in \mbf T_{(m,t)}^{\theta,\star}$, the southwest path consists of horizontal and vertical line segments, where $\tau_r^\star$ denotes the position of the vertical segment connecting levels $(r + 1)^\star$ and $r^\star$. Specifically, the path consists of the points
\[
\bigcup_{r = -\infty}^m \{(r^\star,u): u \in [\tau_{r - 1}^\star,\tau_r^\star]\} \cup \bigcup_{r = -\infty}^{m - 1} \{(v,\tau_r^\star): v \in [(r - 1)^\star,r^\star]\}.
\]
Figure~\ref{fig:reg axes and dual axes} shows the regular axes and the dual axes together, with a southwest dual geodesic traveling on this dual plane. Each element of $\mbf T_{\mbf (m,t)}^{\theta \star}$ is a southwest semi-infinite geodesic for the dual environment $\mbf X^\theta$. This fact is recorded in Theorem~\ref{existence of backwards semi-infinite geodesics}.  

\begin{figure}[t]
\begin{tikzpicture}
\draw[black,thick] (-0.5,0) -- (15,0);
\draw[black,thick] (-0.5,1)--(15,1);
\draw[black,thick] (-0.5,2)--(15,2);
\draw[black,thick] (-0.5,3)--(15,3);
\draw[dashed] (-0.5,-0.5)--(15,-0.5);
\draw[dashed] (-0.5,0.5)--(15,0.5);
\draw[dashed] (-0.5,1.5)--(15,1.5);
\draw[dashed] (-0.5,2.5)--(15,2.5);
\draw[blue, ultra thick,->] (14.5,2.5)--(12,2.5)--(12,1.5)--(7,1.5)--(7,0.5)--(3,0.5)--(3,-0.5)--(1,-0.5)--(1,-1);
\filldraw[black] (14.5,2.5) circle (2pt);
\node at (-1.3,1.5) {\small $(m - 1)^\star$};
\node at (-1.3,2) {\small $m - 1$};
\node at (-1.3,2.5) {\small $m^\star$};
\node at (-1.3,3) {\small $m$};
\node at (-1.3,0.5) {\small $\vdots$};
\node at (14.5,-1) {\small $t$};
\node at (12,-1) {\small $\tau_{(m,t),m - 1}^\star$};
\node at (7,-1) {\small $\tau_{(m,t),m - 2}^\star$};
\node at (4,-1) {$\cdots$};
\end{tikzpicture}
\caption{\small The original discrete levels (solid) and the dual levels (dashed). Dual levels are labeled with a $\star$.}
\label{fig:reg axes and dual axes}
\end{figure}
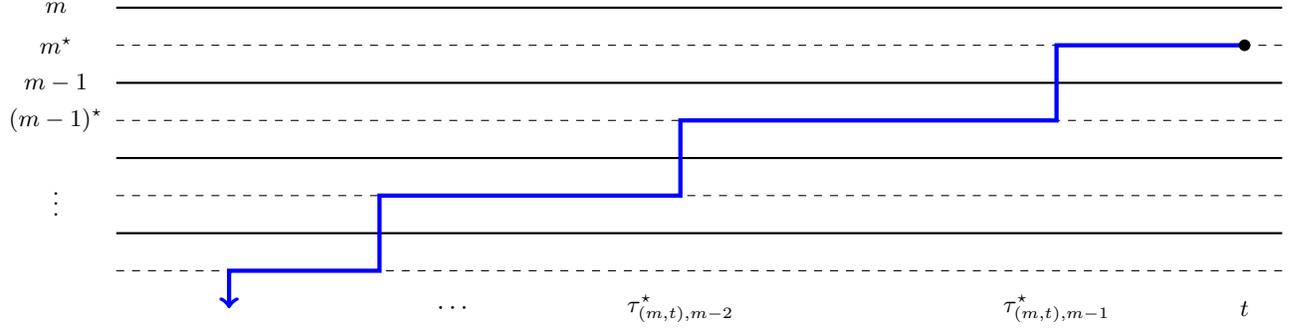

Since $\mbf X^\theta$ is an environment of i.i.d.\ Brownian motions, the following theorem allows us to conclude Part~\ref{itm:coalescence} of Theorem~\ref{thm:general_SIG} from Part~\ref{itm:no bi-infinite geodesics in given direction}. Full details of this connection are found in the proofs. Refer to Figure~\ref{fig:bi-infinite dual path from disjoint geodesics} for clarity.
\begin{figure}[t]
\centering
\begin{tikzpicture}
\draw[gray,thin] (-0.5,0) -- (15,0);
\draw[gray,thin] (-0.5,1)--(15,1);
\draw[gray,thin] (-0.5,2)--(15,2);
\draw[gray,thin] (-0.5,3)--(15,3);
\draw[gray,thin] (-0.5,4)--(15,4);
\draw[gray,thin] (-0.5,4)--(15,4);
\draw[gray,thin] (-0.5,5)--(15,5);
\draw[gray,thin,dashed] (-0.5,0.5)--(15,0.5);
\draw[gray,thin,dashed] (-0.5,1.5)--(15,1.5);
\draw[gray,thin,dashed] (-0.5,2.5)--(15,2.5);
\draw[gray,thin,dashed] (-0.5,3.5)--(15,3.5);
\draw[gray,thin,dashed] (-0.5,4.5)--(15,4.5);
\draw[red, ultra thick,->] (1,0)--(3,0)--(3,1)--(6,1)--(7,1)--(7,2)--(9,2)--(11,2)--(11,3)--(13,3)--(14.5,3)--(14.5,4)--(15,4);
\draw[red, ultra thick,->] (-0.5,1)--(0.5,1)--(0.5,2)--(2.5,2)--(2.5,3)--(5,3)--(5,4)--(12,4)--(12,5)--(15,5);
\draw[blue,thick,<->] (15,4.5)--(13,4.5)--(13,3.5)--(6,3.5)--(6,2.5)--(4.5,2.5)--(4.5,1.5)--(1.5,1.5)--(1.5,0.5)--(0.5,0.5)--(0.5,-0.5)--(-0.5,-0.5);
\filldraw[black] (1,0) circle (2pt);
\node at (1.3,-0.5) {$\mbf x$};
\filldraw[black] (-0.5,1) circle (2pt) node[anchor = east] {$\mbf y$};
\end{tikzpicture}
\caption{\small The outcome of Theorem \ref{thm:bi-infinite geodesic between disjoint geodesics}:  two disjoint semi-infinite northeast paths (red/thick) and the dual bi-infinite path (blue/thin)}
\label{fig:bi-infinite dual path from disjoint geodesics}
\end{figure}
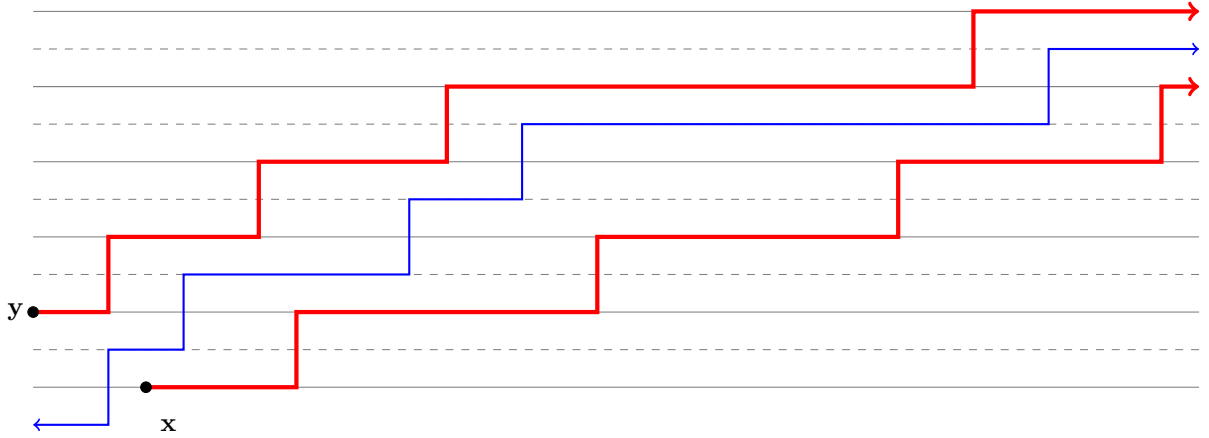

\begin{theorem} \label{thm:bi-infinite geodesic between disjoint geodesics}
Fix $\theta > 0$. With probability one, if for any $\mbf x,\mbf y \in \Z \times \R$, the rightmost semi-infinite geodesics in $\mbf T_{\mbf x}^{\theta }$ and $\mbf T_{\mbf y}^{\theta }$ are disjoint (i.e. the paths share no points), then there exists a bi-infinite, upright path defined by jump times $\cdots \le \tau_{-1}^\star \le \tau_0^\star \le \tau_1^\star \le \cdots$ {\rm(}where $\tau_i^\star$ denotes the jump time from level $i$ to level $i + 1${\rm)} that satisfies the following:
\begin{enumerate} [label=\rm(\roman{*}), ref=\rm(\roman{*})]  \itemsep=3pt 
    \item For any point $\mbf x^\star$ along the path, the portion of that path to the south and west of $\mbf x^\star$ is the leftmost semi-infinite geodesic in the set $\mbf T_{\mbf x}^{\theta, \star}$. Specifically, when shifted back up by $\f{1}{2}$, the path is a bi-infinite geodesic for the environment $\mbf X^\theta$. 
    \item The sequence $\{\tau_n^\star\}_{n \in \Z}$ satisfies
    \[
    \lim_{n \rightarrow \infty} \f{\tau_n^\star}{n} = \theta = \lim_{n \rightarrow \infty} \f{\tau_{-n}^\star}{-n}.
    \]
\end{enumerate}  

The analogous result holds if we assume that two leftmost $\theta$-directed semi-infinite geodesics are disjoint. In this case,  the portion of the path to the south and west of each of its points  is the rightmost dual semi-infinite geodesic.  
\end{theorem}

 \section{Connections to other models} \label{section:connections}
 \subsection{Connection to infinite-length polymer measures for the Brownian polymer}
 Recalling the definitions at the beginning of Section~\ref{section:def of BLPP}, for $(m,s) \le (n,t) \in \Z \times \R$, the point-to-point partition function of the Brownian polymer with unit temperature is
 \[
 Z_{(m,s),(n,t)}(\mbf B) = \int e^{\E(\mbf s_{m,n})} \mbf 1\{\mbf s_{m,n - 1} \in \Pi_{(m,s),(n,t)}\}d \mbf s_{m,n - 1}.
 \]
 The associated quenched polymer measure on $\Pi_{(m,s),(n,t)}$ is
\[
 Q^{\mbf B}(\tau_m \in d s_m,\ldots, \tau_{n - 1} \in d s_{n - 1}) = \f{1}{Z_{(m,s),(n,t)}(\mbf B)}\exp \Bigl(\sum_{k = m}^n B_k(s_{k - 1},s_k)\Bigr)1\{\mbf s_{m,n - 1} \in \Pi_{(m,s),(n,t)}\}d \mbf s_{m,n - 1}.
 \]
 Brownian last-passage percolation is the zero-temperature analogue of the Brownian or O'Connell-Yor polymer. This is made precise by the limit
 \[
 \lim_{\beta \rightarrow \infty} \f{1}{\beta} \log Z_{(m,s),(n,t)}(\beta \mbf B) = L_{(m,s),(n,t)}(\mbf B),
\]
which, by convergence of $L^p$ norms, holds in a deterministic sense as long as $\mbf B$ is a field of continuous functions $\R\rightarrow \R$. The parameter $\beta$ is used to denote inverse temperature.

Alberts, Rassoul-Agha, and Simper~\cite{blpp_utah} showed the existence of Busemann functions and infinite-length limits of the quenched measures when the right endpoint $(n,t_n)$ satisfies $t_n/n \rightarrow \theta$ for some fixed $\theta > 0$. The Busemann functions are defined as 
\[
\widehat  \B^\theta(\mbf x,\mbf y) := \lim_{n \rightarrow \infty} \log \f{ Z_{\mbf x,(n,t_n)}(\mbf B)}{Z_{\mbf y,(n,t_n)}(\mbf B)}. 
\]
Similarly as in~\eqref{horizontal Busemann simple expression}, define $\widehat \h_m^\theta(t) = \widehat \B^\theta((m,0),(m,t))$. Using the ``Proof of Theorem 2.5, assuming Theorem 3.1" on page 1937 of~\cite{blpp_utah} and an analogous construction of the Busemann functions for all initial points as in Section~\ref{section:Busemann construction} of the present paper, it can be shown that, the process $t \mapsto \widehat h_m^\theta(t)$ is a two-sided Brownian motion with drift, independent of $\{B_r\}_{r \le m}$. By Theorems~\ref{thm:summary of properties of Busemanns for all theta}\ref{independence structure of Busemann functions on levels} and~\ref{thm:dist of Busemann functions}\ref{BM_drift}, the same is true for the zero temperature Busemann process $t \mapsto h_m^\theta(t)$. 
An infinite, up-right path under the quenched infinite-length polymer measure is a continuous-time Markov chain, starting from a point $(m,t)$ and defined by jump times $t = \tau_{m - 1} \le \tau_m \le \cdots$. By Equation (2.3) in~\cite{blpp_utah}, the quenched conditional distribution of $\tau_{r}$ given $\tau_{r - 1}$ is
\begin{align*}
Q^{\mbf B}_{(m,s),(n,t_n)}(\tau_{r} \in d s_r| \tau_{r - 1} = s_{r - 1})
&= e^{B_{r}(s_{r - 1},s_r)}\f{Z_{(r + 1,s_r),(n,t_n)}(\mbf B)}{Z_{(r,s_{r - 1}),(n,t_n)}(\mbf B)} \mbf 1\{s_{r - 1} \le s_r\}\, ds_r.
\end{align*}
The proof of existence of infinite length measures requires a rigorous tightness argument, but to motivate the connection to BLPP, we formally take limits as $n \rightarrow \infty$ to yield the conditional measure
\begin{align*}
&\;\exp\bigl(B_r(s_{r - 1},s_r) + \widehat \B^\theta((r +1,s_r),(r,s_{r - 1}))\bigr)\mbf 1\{s_{r - 1} \le s_r\}\, ds_r \\[1em]
&= \exp\bigl(B_r(s_{r - 1},s_r) - \widehat \h_{r + 1}(s_r) + \widehat \B^\theta((r + 1,0),(r ,s_{r - 1})\bigr)\mbf 1\{s_{r - 1} \le s_r\}\, ds_r.
\end{align*}
Now, note that the term $\widehat \B^\theta((r + 1,0),(r ,s_{r - 1})$ does not depend on $s_r$. We can then think of the connection between BLPP and the O'Connell-Yor polymer in the following sense: In the positive temperature case, given the environment $\mbf B$, the transition density from $\tau_{r - 1}$ to $\tau_r$ is given by
\[
C\exp\bigl(B_r(s) - \widehat \h_{r + 1}^\theta(s)\bigr)\mbf 1\{\tau_{r - 1} \le s\}\,ds,
\]
where $C$ is a normalizing constant depending on $r$ and $\tau_{r - 1}$. On the other hand, by Definition~\ref{def:semi-infinite geodesics}, in the zero-temperature case, given the environment $\mbf B$ and the previous jump $\tau_{r - 1}$, the jump $\tau_r$ is chosen in a deterministic fashion by maximizing $B_r(s) - h_{r +1}^\theta(s)$ over $s \in [\tau_{r - 1},\infty)$.

\subsection{Connection to semi-infinite geodesics in discrete last-passage percolation} \label{sec:discreteLPP}
The corner growth model, or discrete last-passage percolation, is defined as follows. Let $\{Y_{\mbf x}\}_{\mbf x \in \Z^2}$ be a collection of nonnegative i.i.d random variables, each associated to a vertex on the integer lattice. For $\mbf x \le \mbf y \in \Z \times \Z$, define the last-passage time as
\[
G_{\mbf x,\mbf y} = \sup_{\mbf x_\centerdot \in \Pi_{\mbf x,\mbf y}} \sum_{k = 0}^{|\mbf y - \mbf x|_1} Y_{\mbf x_k}, 
\]
where $\Pi_{x,y}$ is the set of up-right paths $\{\mbf x_k\}_{k = 0}^{n}$ that satisfy $\mbf x_0 = \mbf x,\mbf x_{n} = \mbf y$, and $\mbf x_k - \mbf x_{k - 1} \in \{\mbf e_1,\mbf e_2\}$. Under the assumption that $Y_0$ has finite second moment, Theorem 3.1 and Corollary 3.1 of~\cite{glynn1991} introduced BLPP as a universal scaling limit of the corner growth model, where one variable is scaled, and the other is held constant. That is, if $Y_0$ is normalized to have unit mean and variance,
\[
\{L_{(m,s),(n,t)}: (m,s) \le (n,t) \in \Z \times \R\}
\]
is the functional limit, as $k \rightarrow \infty$, of the properly interpolated version of the process
\[
\bigl\{\f{1}{\sqrt k} G_{(m,\floor{sk}),(n,\floor{tk})} - (t - s)k: (m,s) \le (n,t) \in \Z \times \R \bigr\}.
\]

The most tractable case of discrete last-passage percolation is the case where $Y_0$ has the exponential distribution with rate $1$. In this case, Busemann functions exist and are indexed by a direction vector $\mbf u$. They are defined by
\[
U^{\mbf u}(\mbf x,\mbf y) := \lim_{n \rightarrow \infty} G_{(\mbf x,\mbf z_n)} - G_{(\mbf y,\mbf z_n)},
\]
where $\mbf z_n$ satisfies $\mbf z_n/n \rightarrow \mbf u$ for a fixed direction $\mbf u$. For a given $\omega \in \Omega$, $\mbf x \in \Z^2$, and direction $\mbf u$, a semi-infinite geodesic $\mbf \gamma$ is defined by the sequence $\{\mbf \gamma_k^{\mbf u,\mbf x}\}_{k \in \Z_{\ge 0}}$. At each step, a choice is made to move upward or to the right. First, set $\mbf \gamma^{\mbf u,\mbf x}_0 = \mbf x$, and for $k \ge 0$,
\be \label{eqn:cgm geod}
\mbf \gamma_{k + 1}^{\mbf u,\mbf x} = \begin{cases}
\mbf \gamma_{k}^{\mbf u,\mbf x} + \mbf e_1, &\text{if}\quad U^{\mbf u}\bigl(\mbf \gamma_{k}^{\mbf u,\mbf x},\mbf \gamma_{k}^{\mbf u,\mbf x} + \mbf e_1\bigr) \le U^{\mbf u}\bigl(\mbf \gamma_{k}^{\mbf u,\mbf x},\mbf \gamma_{k}^{\mbf u,\mbf x} + \mbf e_2\bigr), \\[1em]
\mbf \gamma_{k}^{\mbf u,\mbf x} + \mbf e_2, &\text{if}\quad U^{\mbf u}\bigl(\mbf \gamma_{k}^{\mbf u,\mbf x},\mbf \gamma_{k}^{\mbf u,\mbf x} + \mbf e_2\bigr) < U^{\mbf u}\bigl(\mbf \gamma_{k}^{\mbf u,\mbf x},\mbf \gamma_{k}^{\mbf u,\mbf x} + \mbf e_1\bigr).
\end{cases}
\ee
In the case of exponential weights, this sequence is a semi-infinite geodesic with direction $\mbf u$. This construction is inherently discrete, so it does not extend directly to the case of BLPP. However, this construction is equivalent to taking a sequence of maximizers of the appropriate function, analogous to Definition~\ref{def:semi-infinite geodesics}. The following is a discrete analogue of Theorem~\ref{thm:summary of properties of Busemanns for all theta}\ref{general queuing relations Busemanns} that holds in the case of exponential weights. 
\begin{equation} \label{CGM queue relations}
U^{\mbf u}((m,r),(m,r + 1)) = \max_{\substack{m \le k < \infty \\k \in \Z}} \Big\{\sum_{i = m}^k Y_{(i,r)} + U^{\mbf u}((k,r + 1),(m,r + 1)) \Big\}.
\end{equation}

Now, suppose that, starting at the point $(m,r)$, the semi-infinite geodesic constructed in~\eqref{eqn:cgm geod} makes an $\mbf e_2$ step from $(k,r)$ to $(k,r + 1)$ for some $k \ge m$. Then, we show that $k$ is maximal for~\eqref{CGM queue relations}. Indeed, by~\eqref{eqn:cgm geod}, if the path $\mbf \gamma$ makes an $\mbf e_2$ step from $(k,r)$ to $(k,r + 1)$, then  
\begin{align*}
&U^{\mbf u}((i,r),(i + 1,r)) \le U^{\mbf u}((i,r),(i,r + 1)) \qquad\text{ for }\qquad m \le i \le k - 1, \qquad\text {and } \\[1em]
&U^{\mbf u}((k,r),(k + 1,r)) >U^{\mbf u}((k,r),(k,r + 1)).
\end{align*}
Using the identity $Y_{\mbf x} = U^{\mbf u}(\mbf{x,x + e_1}) \wedge U^{\mbf u}(\mbf{x,x + e_2})$
and additivity of the Busemann functions,
\begin{align*}
&\quad\sum_{i = m}^{k}Y_{(i,r)} + U^{\mbf u}((k,r + 1),(m,r + 1)) \\[1em]
&=\sum_{i = m}^{k - 1} U^{\mbf u}((i,r),(i + 1,r)) + U^{\mbf u}((k,r),(k, r + 1)) + U^{\mbf u}((k,r + 1),(m,r + 1)) \\[1em]
&= U^{\mbf u}((m,r),(m,r + 1)),
\end{align*}
so $k$ is indeed maximal in for the right hand-side of~\eqref{CGM queue relations}. The inductive step follows in the same manner. Hence, we see that the construction of semi-infinite geodesics in Definition~\ref{def:semi-infinite geodesics} is a continuous analogue of the procedure for the discrete case when viewed from the perspective of maximizers.

\subsection{Connection to queuing theory} \label{section:queue}
Fix $\theta > 0$, and consider the almost surely unique $\theta$-directed semi-infinite geodesic (Theorem~\ref{thm:general_SIG}\ref{itm:uniqueness of geodesic for fixed point and direction}) starting from $(0,0)$ and defined by the sequence of jump times $\{\tau_r\}_{r \ge -1}$ with $\tau_{-1} = 0$.  For each $r \ge 0$, associate a queuing station as follows.  For $s <t$, let $\f{1}{\sqrt \theta}(t - s) - X_{r + 1}^\theta(s,t)$ denote the service available in the interval $(s,t]$ and let $\f{1}{\sqrt \theta}(t - s) - h_r^\theta(s,t)$ denote the arrivals to the queue in the interval $(s,t]$. Here, the parameter $\f{1}{\sqrt \theta}$ dictates the rate of service. See Appendix~\ref{section:queue and stationary} for more details about the queuing setup. By~\eqref{eqn:dual weights reverse relations}, $h_{r + 1}^\theta = \Da(h_r^\theta,X_{r + 1})$, so as in the proof of Theorem~\ref{O Connell Yor BM independence theorem queues}, $\f{1}{\sqrt \theta} (t - s) - h_{r + 1}^\theta(s,t)$ gives the departures from the $r$th station in the interval $(s,t]$, and the departures process from the $r$th station becomes the arrivals process for the $r + 1$st station. In other words, once a customer is served at the $r$th station, they move into the queue at the $r + 1$st station. By Equations~\eqref{380} and~\eqref{eqn:dual weights reverse relations}, for all $r \in \Z$ and $t \in \R$, 
\[
v_{r + 1}^\theta(t) = \sup_{t \le u < \infty}\{B_r(t,u) - h_{r + 1}^\theta(t,u)\}= \sup_{-\infty < u \le t}\{X_{r +1}^\theta(t,u) - h_{r}^\theta(t,u)\} = \Qa(h_r^\theta,X_{r+1}^\theta)(t). 
\]
In queuing terms, $v_{r + 1}^\theta(t)$ gives the length of the $r$th queue at time $t$. 
 As $\tau_r$ is the maximizer of $B_r(u) - h_{r + 1}^\theta(u)$ over $u \in [\tau_{r - 1},\infty)$, $\tau_r$ is the first time $u \ge \tau_{r - 1}$ such that $v_{r + 1}^\theta(u) = 0$ (see Lemma~\ref{lemma:equality of busemann to weights of BLPP}). In queuing terms, $\tau_r$ is the first time greater than or equal to $\tau_{r - 1}$ at which the queue is empty. Thus, the semi-infinite geodesic represents the movement of a customer through the infinite series of queuing stations: the customer starts at station $0$ at time $0$ and is served at the $r$th station at a time no later than $\tau_r$. 
 
\section{Construction and proofs for the Busemann process} \label{section:Busemann construction}
The remainder of the paper is devoted to the proofs of the theorems. This section constructs the global Busemann process, with the proof of Theorem~\ref{thm:summary of properties of Busemanns for all theta} as the ultimate result. Section~\ref{section:main proofs} proves all results about the semi-infinite geodesics, culminating in the proof of Theorem~\ref{thm:general_SIG}.

\subsection{Construction of Busemann functions for a fixed direction}

\begin{definition}[Definition of Busemann functions for fixed $\theta$ and a countable dense set of points] \label{def:Busemann functions on countable set}

For a fixed $\theta > 0$, and $\mbf x,\mbf y \in \Z \times \Q$, let 
\[
\Omega_0^{(\theta)} = \bigcap_{\mbf x,\mbf y \in \Z \times \Q} \Omega_{\mbf x,\mbf y}^{(\theta)},
\]
where the $\Omega_{\mbf x,\mbf y}^{(\theta)}$ are the events of Theorem~\ref{thm:existence of Busemann functions for fixed points}. Then, $\Pp(\Omega_0^{(\theta)}) = 1$. On this event, for $\mbf x,\mbf y \in \Z \times \Q$, define $\B^\theta(\mbf x,\mbf y)$ by~\eqref{eqn:limit definition of Busemann functions}. By definition in terms of limits, it is clear that the Busemann functions are additive, that is, for $ \omega \in \Omega_0^{(\theta)}$ and $\mbf x,\mbf y,\mbf z \in \Z \times \Q$, 
\begin{equation} \label{eqn:additivity for countable dense set}
\B^\theta(\mbf x,\mbf y) + \B^\theta(\mbf y,\mbf z) = \B^\theta(\mbf x,\mbf y).
\end{equation}
Hence, the entire collection of Busemann functions is constructed from the collection of horizontal and vertical Busemann functions $\{\h_m^\theta,\vv_m^\theta\}_{m \in \Z}$ (Equations~\eqref{horizontal Busemann simple expression} and~\eqref{vertical Busemann simple expression}). 
\end{definition}
Set $Y(t) = -B_0(t) + \f{1}{\sqrt \theta} t$. By Lemma~\ref{lemma:equality of Busemann function distribution at rationals}, as elements of the space of functions $\Q\rightarrow \R$, equipped with the standard product $\sigma-$algebra,  
\[
    \bigl\{\h_0^\theta(t): t \in \Q \bigr\}
    \overset{d}{=} \bigl\{Y(t): t \in \Q \bigr\} \qquad \text{and} \qquad \bigl\{\vv_{1}^\theta(t): t \in \Q\bigr\} \overset{d}{=} \bigl\{\Qa(Y,B_1)(t): t \in \Q\bigr\}.
\]
For $\omega \in \Omega_0^{(\theta)}$, $m \in \Z$, and $t \in \Q$,
\[
h_m^\theta(t) = \lim_{n \rightarrow \infty} L_{(m,0),(n,n\theta)}(\mbf B) - L_{(m,t),(n,n\theta)}(\mbf B) = \lim_{n \rightarrow \infty} L_{(m,0),(n + m,(n + m)\theta)}(\mbf B) - L_{(m,t),(n + m,(n + m)\theta)}(\mbf B).
\]
Since the environment $\mbf B = \{B_m\}_{m \in \Z}$ is a field of i.i.d. two-sided Brownian motions, the right-hand side implies that $h_m^\theta(t)$ has the same distribution for all $m$. By this same reasoning, for each $m \in \Z$,
\[
\{\h_m^\theta(t), \vv_{m + 1}^\theta(t): t \in \Q\} \deq \{\h_0^\theta(t),\vv_1^\theta(t): t \in \Q\}.
\]
Since Brownian motion satisfies $\lim_{s \rightarrow \infty} \f{B(s)}{s} = 0$ almost surely, for all $m \in \Z$, the following limits also hold almost surely:
\[
\lim_{\Q \ni s \rightarrow  \pm \infty} \bigl[B_m(s) - \h_{m + 1}^\theta(s)\bigr]= \mp \infty.
\]
Set
    \begin{align}
       \Omega^{(\theta)} = &\bigcap_{m,r \in \Z}\bigl\{(s,t) \mapsto \B^\theta((m,s),(r,t)) \text{ is uniformly continuous on all bounded subsets of }\Q \times \Q\bigr\} \nonumber \\[1em]
       &\qquad\qquad\cap\bigcap_{m \in \Z}\{\lim_{\Q \ni s \rightarrow \pm \infty}\bigl[B_m(s) - \h_{m + 1}^\theta(s)\bigr] = \mp \infty\}.
    \end{align}
    By almost sure continuity of the functions $Y$ and $\Qa(Y,B_1)$ (Lemma~\ref{queue length is continuous function of t}), $\Pp(\Omega^{(\theta)}) = 1$. 
\begin{definition}[Definition of Busemann functions for fixed $\theta$, arbitrary points] \label{def:Busemann functions, fixed theta}
On the event $\Omega^{(\theta)}$, for arbitrary $\mbf x,\mbf y \in \Z \times \R$, define $\B^\theta(\mbf x,\mbf y)$ such that, for each $m,r \in \Z$,
\[
(s,t) \mapsto \B^\theta((m,s),(r,t))
\]
is the unique continuous extension of this function from $\Q \times \Q$ to $\R^2$.
\end{definition} The following 
lemma states properties for a fixed $\theta$, as a
 precursor to the more general Theorems~\ref{thm:summary of properties of Busemanns for all theta} and Theorem~\ref{thm:dist of Busemann functions}.
\begin{lemma} \label{lemma:summary of Busemanns fixed theta}
Let $\theta>0$. Then, the following hold. 
\begin{enumerate} [label=\rm(\roman{*}), ref=\rm(\roman{*})]  \itemsep=3pt 
    \item \label{itm:Busemann limits, fixed theta} On the event $\Omega^{(\theta)}$, for all $\mbf x,\mbf y \in \Z \times \R$ and all sequences $\{t_n\}$ satisfying $t_n/n \rightarrow \theta$,
    \[
    \B^\theta(\mbf x,\mbf y) = \lim_{n \rightarrow \infty} \bigl[L_{\mbf x,(n,t_n)} - L_{\mbf y,(n,t_n)}\bigr]. 
    \]
    \item \label{itm:fixed theta additivity} On the event $\Omega^{(\theta)}$, whenever $\mbf x,\mbf y,\mbf z \in \Z \times \R$,
    \[
    \B^\theta(\mbf x,\mbf y) + \B^\theta(\mbf y,\mbf z) = \B^\theta(\mbf x,\mbf z). 
    \]
    \item \label{itm:fixed theta monotonicity} For $0 < \gamma < \theta < \infty$, on the event $\Omega^{(\theta)} \cap \Omega^{(\gamma)}$, for all $m \in \Z$ and $s < t \in \R$,
\[
0 \le \vv_{m}^\gamma(s) \leq \vv_m^\theta(s), \text{ and } 
B_m(s,t) \leq \h_m^\theta(s,t) \leq \h_m^\gamma(s,t). 
\]
\item \label{itm:fixed theta limit of B_m - \h_{m + 1}} On the event $\Omega^{(\theta)}$, for each $m \in \Z$,
\[
\lim_{s \rightarrow \pm \infty} B_m(s) - \h_{m + 1}^\theta(s) = \mp \infty. 
\]
\item For every $m \in \Z$, the process $t \mapsto \h_m^\theta(t)$ is a two-sided Brownian motion with drift $\f{1}{\sqrt \theta}$. For each $m \in \Z$ and $t \in \R$, $\vv_m^\theta(t) \sim \operatorname{Exp}\bigl(\f{1}{\sqrt \theta}\bigr)$. \label{itm: general distribution of h v busemann functions} 
\item \label{item: fixed theta independence of busemann on higher levels} For any $m \in \Z$, $\{\h_r^\theta\}_{r > m}$ is independent of $\{B_r\}_{r \le m}$. 
\item \label{itm:fixed theta Busemann queuing relations} There exists a full-probability event $\Omega_1^{(\theta)} \subseteq \Omega^{(\theta)}$  on which the following hold for all $m\in\Z$ and $t \in \R:$
\[
\vv_{m + 1}^\theta(t) = Q(\h_{m +1}^\theta,B_m)(t)\qquad\text{and}\qquad \h_m^\theta(t) = D(\h_{m + 1}^\theta,B_m)(t). 
\]
\end{enumerate}
\end{lemma}

\begin{proof}
\textbf{Part~\ref{itm:Busemann limits, fixed theta}}: On $\Omega^{(\theta)}$, let $(m,s),(r,t) \in \Z \times \R$, and let $\{t_n\}$ be any sequence with $t_n/n \rightarrow \theta$. Note that, whenever $t_1 < t_2, m \in \Z$ and $\mbf y \ge (m,t_2)$,
    \[
    L_{(m,t_1),\mbf y} \ge B_m(t_1,t_2) + L_{(m,t_2),\mbf y}.
    \]
    Then, let $q_1,q_2 \in \Q$ be such that $q_1 < s$ and $q_2 > t$. Then, 
    \[
    L_{(m,s),(n,t_n)} - L_{(r,t),(n,t_n)} \le L_{(m,q_1),(n,t_n)} - B_m(q_1,s) - L_{(r,q_2),(n,t_n)} - B_r(t,q_2),
    \]
    and therefore,
    \[
    \limsup_{n \rightarrow \infty} \bigl[L_{(m,s),(n,t_n)} - L_{(r,t),(n,t_n)}\bigr] \le \B^{\theta}((m,q_1),(r,q_2))- B_m(q_1,s)- B_r(t,q_2).
    \]
    Taking $q_1 \nearrow s$ and $q_2 \searrow t$ and using the continuity of Brownian motion and the Busemann functions,
    \[
    \limsup_{n \rightarrow \infty} \bigl[L_{(m,s),(n,t_n)} - L_{(r,t),(n,t_n)}\bigr] \le \B^{\theta }((m,s),(r,t)).
    \]
    A similar procedure gives the appropriate lower bound.

\medskip \noindent \textbf{Part~\ref{itm:fixed theta additivity}}: This follows from the additivity of the process on the countable dense set (Equation~\eqref{eqn:additivity for countable dense set}) and the construction of Definition~\ref{def:Busemann functions, fixed theta} as the unique continuous extension.

\medskip

\noindent \textbf{Part~\ref{itm:fixed theta monotonicity}}: By Lemma~\ref{lemma:bounds for differences of Brownian LPP}, for $s, t \in \R$ and all $n$ such that $n \geq m$ and $n\gamma \ge s \vee t$, 
\begin{align*} 
 0 &\le L_{(m,s),(n,n\gamma)} - L_{(m + 1,s),(n,n\gamma)} \leq L_{(m,s),(n,n\theta)} - L_{(m + 1,s),(n,n\theta)} \\
\text{and}\qquad  
B_m(s,t) &\leq  L_{(m,s),(n,n\theta)} - L_{(m,t),(n,n\theta)} \leq L_{(m,s),(n,n\gamma)} - L_{(m,t),(n,n\gamma)}.
\end{align*} 
The proof is complete by Part~\ref{itm:Busemann limits, fixed theta}, taking limits as $n \rightarrow \infty$.

\medskip

\noindent \textbf{Part~\ref{itm:fixed theta limit of B_m - \h_{m + 1}}}: This follows from the definition of $\Omega^{(\theta)}$ and continuity.
\medskip

\noindent \textbf{Part~\ref{itm: general distribution of h v busemann functions}}:
Observe that $t \mapsto \h_m^{\theta}(t)$ is a continuous process with the correct finite-dimensional distributions by Lemma~\ref{lemma:equality of Busemann function distribution at rationals}, taking limits when necessary. Part~\ref{itm:Busemann limits, fixed theta} and Theorem~\ref{thm:existence of Busemann functions for fixed points} guarantee that $\vv_m^\theta(t) \sim\operatorname{Exp}\bigl(\f{1}{\sqrt \theta}\bigr)$ for all $t \in \R$. 

\medskip
\noindent \textbf{Part~\ref{item: fixed theta independence of busemann on higher levels}}:
 By Part~\ref{itm:Busemann limits, fixed theta},
\[
\{\h_r^\theta(t): r > m, t \in \R \} = \bigl\{\lim_{n \rightarrow \infty} \bigl[L_{(r,0),(n,n\theta)} - L_{(r,t),(n,n\theta)}\bigr]:r > m, t \in \R\bigr\}.
\]
The right-hand side is a function of $\{B_r\}_{r > m}$, which is independent of $\{B_r\}_{r \le m}$.

\medskip
\noindent \textbf{Part~\ref{itm:fixed theta Busemann queuing relations}}:
We start with the first statement. We need to show that
\begin{equation} \label{inductive equality for Busemann functions}
\vv_{m + 1}^\theta(t) = \sup_{t \leq s < \infty} \{B_m(t,s) - \h_{m + 1}^\theta(t,s)  \}.
\end{equation}
By Part~\ref{itm:Busemann limits, fixed theta},
\begin{align} \label{D}
\sup_{t \leq s < \infty} \{B_m(t,s) - \h_{m + 1}^\theta(t,s)  \} = \sup_{t \leq s < \infty} \{B_m(t,s) + \lim_{n \rightarrow \infty}\bigl[L_{(m + 1,s),(n,n\theta)} - L_{(m + 1,t),(n,n\theta)}\bigr]  \}.
\end{align}
On the other hand,
\be \label{C}
\vv_{m +1}^\theta(t) = \lim_{n \rightarrow \infty} \bigl[L_{(m,t),(n,n\theta)} - L_{(m + 1,t),(n,n\theta)}\bigr].
\ee 
For each $s \ge t$ and all $n$ sufficiently large so that $n\theta \geq s$ and $n \ge m + 1$, 
\[
B_m(t,s) + L_{(m + 1,s),(n,n\theta)} - L_{(m + 1,t),(n,n\theta)} \leq L_{(m,t),(n,n\theta)} - L_{(m + 1,t),(n,n\theta)}.
\]
Taking limits as $n \rightarrow \infty$ and comparing~\eqref{D} and~\eqref{C} establishes \[\vv_{m +1}^\theta(t) \ge \sup_{t \leq s < \infty} \{B_m(t,s) - \h_{m + 1}^\theta(t,s)  \}.\] Next, by Part~\ref{itm: general distribution of h v busemann functions}, $\vv_{m +1}^\theta(t)\sim \operatorname{Exp}(\frac{1}{\sqrt \theta})$. By Parts~\ref{itm: general distribution of h v busemann functions} and~\ref{item: fixed theta independence of busemann on higher levels} and Lemma~\ref{lemma:sup of BM with drift},
\begin{align*}
\sup_{t \leq s < \infty} \{B_m(t,s) - \h_{m + 1}^\theta(t,s)  \}  \overset{d}{=} \sup_{0 \leq s < \infty} \bigl\{\sqrt 2 B(s) - \frac{1}{\sqrt \theta}s \bigr\} 
\deq v_{m + 1}^\theta(t),
\end{align*}
where $B$ is a standard Brownian motion. Thus, since $v_{m +1}^\theta(t) \ge \sup_{t \leq s < \infty} \{B_m(t,s) - \h_{m + 1}^\theta(t,s)  \}$, equality holds with probability one for each fixed $t \in \R$. Let $\Omega_1^{(\theta)} \subseteq \Omega^{(\theta)}$, be the event of full probability on which~\eqref{inductive equality for Busemann functions} holds for all $t \in \Q$. Continuity of both sides of~\eqref{inductive equality for Busemann functions} (Definition~\ref{def:Busemann functions, fixed theta} and Lemma~\ref{queue length is continuous function of t}) extend the result to all $t \in \R$ on $\Omega_1^{(\theta)}$. The equality for $\h_m^\theta$ then follows by the definitions and additivity of the Busemann functions, as shown below.  
\begin{align*}
D(\h_{m + 1}^\theta,B_m)(t) &= \h_{m + 1}^\theta(t) + Q(\h_{m + 1}^\theta,B_m)(0) - Q(\h_{m + 1}^\theta,B_m)(t) \\[1em]
&= h_{m + 1}^\theta(t) + v_{m + 1}^\theta(0) - v_{m +1}^\theta(t) \\[1em]
&= \;\B^\theta((m + 1,0),(m + 1,t)) + \B^\theta((m,0),(m + 1,0)) - \B^\theta((m,t),(m + 1,t)) \\[1em]
&= \B^\theta((m,0),(m,t)) = \h_m^\theta(t). \qedhere
\end{align*}
\end{proof}

\subsection{Global construction of the Busemann process}
We now have the proper framework to define the global Busemann process of Theorem~\ref{thm:summary of properties of Busemanns for all theta}. Fix a countable dense subset $D$ of $(0,\infty)$. Let 
\[
\Omega_0 = \bigcap_{\theta \in D} \Omega_1^{(\theta)},
\]
 and on the event $\Omega_0$, define $\B^\theta$ for all $\theta \in D$, as in Definition~\ref{def:Busemann functions, fixed theta}. 
Then, the conclusions of Lemma~\ref{lemma:summary of Busemanns fixed theta} hold for all $\theta \in D$. 

\begin{lemma} \label{lemma:Omega 1 prob 1}
On $\Omega_0$, for each $\theta \in D,m \in \Z$, and $t \in \R$, let $\tau_{(m,t)}^{\theta}$ denote the rightmost maximizer of $B_m(s) - \h_{m + 1}^\theta(s)$ over $s \in [t,\infty)$. Such a maximizer exists by continuity and Lemma~\ref{lemma:summary of Busemanns fixed theta}\ref{itm:fixed theta limit of B_m - \h_{m + 1}}. 
Let $\Omega_1$ be the subset of $\Omega_0$ on which, for each $\theta \in D,m,N \in \Z$, and $t \in \R$, the following limits hold: 
\begin{enumerate} [label=\rm(\roman{*}), ref=\rm(\roman{*})]  \itemsep=6pt 
    \item $\lim_{D \ni \gamma \rightarrow \theta} \h_m^{\gamma}(N) = \h_m^\theta(N)$  
    \item $\lim_{D \ni \gamma \rightarrow \infty} \h_m^{\gamma}(N) = B_m(N)$. \label{itm: converge to Bm}
    \item $\lim_{D \ni \delta \searrow 0} \tau_{(m,t)}^\delta = t$. \label{itm: convergence of maxes to t} 
    \item $\lim_{D \ni \gamma \rightarrow \infty} \tau_{(m,t)}^\gamma = \infty$.  \label{itm:convergence of maxes to infinity}
\end{enumerate}
Then, $\Pp(\Omega_1) = 1$.
\end{lemma}
\begin{proof}
The almost sure uniqueness of maximizers follows from Lemma~\ref{lemma:summary of Busemanns fixed theta}\ref{item: fixed theta independence of busemann on higher levels} and the $n = m$ case of Lemma~\ref{lm:point-to-line uniqueness}. By Lemma~\ref{lemma:summary of Busemanns fixed theta}\ref{itm:fixed theta monotonicity}, $\h_m^\gamma(N) = \h_m^\gamma(0,N)$ is monotone as $D \ni \gamma \nearrow \theta$, and for $N > 0$, 
\begin{equation} \label{eqn:limits of busemann in theta monotonicity}
\lim_{D \ni \gamma \nearrow \theta} \h_m^\gamma(N) \ge \h_m^\theta(N),
\end{equation}
For negative $N$, the inequality flips. By Theorem~\ref{thm:existence of Busemann functions for fixed points}, $\h_m^\gamma(N) \sim \Nor\bigl(\f{N}{\sqrt \gamma},|N|\bigr)$. As $\gamma \nearrow \theta$, this converges in distribution to $\Nor\bigl(\f{N}{\sqrt \theta},|N|\bigr)$. Then, by~\eqref{eqn:limits of busemann in theta monotonicity},  $
\lim_{D \ni \gamma \nearrow \theta} \h_m^\gamma(N) = \h_m^\theta(N) \text{ with probability one}$. An analogous argument proves the almost sure convergence for limits from the right and the convergence to $B_m(N)$.

Next, we show that, on $\Omega_0$, Parts~\ref{itm: convergence of maxes to t} and~\ref{itm:convergence of maxes to infinity} hold if and only if they hold for all $t \in \Q$. Assume the statement holds for all $t \in \Q$.  By Lemmas~\ref{lemma:summary of Busemanns fixed theta}\ref{itm:fixed theta monotonicity} and~\ref{monotonicity of maximizers from function monotonicity}, $\tau_{(m,t)}^\delta$ is monotone as $\delta \searrow 0$, and so the limit exists. By definition of $\tau_{(m,t)}^\theta$, we have the inequality $\tau_{(m,s)}^\theta \le \tau_{(m,t)}^\theta$ whenever $s < t$. Then, for any $t \in \R$ and any $q_1,q_2 \in \Q$ with $q_1 < t < q_2$, 
\[
q_1 = \lim_{D \ni \delta \searrow 0} \tau_{(m,q_1)}^{\delta}  \le \lim_{D \ni \delta \searrow 0} \tau_{(m,t)}^{\delta} \le \lim_{D \ni \delta \searrow 0} \tau_{(m,q_2)}^{\delta} = q_2.
\]
Taking limits as $q_1 \nearrow t$ and $q_2 \searrow t$ shows the statement for all $t \in \R$. The same argument can be applied to the limits as $\gamma \rightarrow \infty$.

Lastly, we show that the limit in~\ref{itm: convergence of maxes to t} holds with probability one for fixed $t \in \R$. By Lemma~\ref{lemma:summary of Busemanns fixed theta}, Parts~\ref{itm: general distribution of h v busemann functions} and~\ref{item: fixed theta independence of busemann on higher levels}, 
\[
\bigl\{B_m(s) - \h_{m + 1}^\delta(s): s \in \R \bigr\} \deq \bigl\{\sqrt 2 B(s) - \f{s}{\sqrt \delta}:s \in \R\bigr\},
\]
where $B$ is a standard, two-sided Brownian motion. Then, by Theorem~\ref{distribtution of argmax BM with drift}, for $(m,t) \in \Z \times \R$ and $s \ge 0$, 
\[
\Pp(\tau_{(m,t)}^{\delta} > s + t) = \left(2 + \f{t}{\delta}\right) \Phi\left(-\sqrt{\f{t}{2 \delta}}\right) - \sqrt{\f{t}{\pi \delta}}\,e^{-\f{t}{4 \delta}}.
\]
Taking limits as $\delta \searrow 0$,  it follows that $\tau_{(m,t)}^{\theta}$ converges weakly to the constant $t$. Since the limit exists almost surely by monotonicity, the desired conclusion follows. A similar argument applies to show $\lim_{\gamma \rightarrow \infty} \tau_{(m,t)}^{\gamma} = \infty$ with probability one. 
\end{proof}

On $\Omega_1$, we extend the definition of $\B^\theta$ to all $\theta > 0$. We proceed similarly as shown for the exponential corner growth model in~\cite{Sepp_lecture_notes}. For BLPP, there is an additional step that must be taken to guarantee the convergence and continuity of Theorem~\ref{thm:summary of properties of Busemanns for all theta}, Parts~\ref{general uniform convergence Busemanns} and \ref{general continuity of Busemanns}. 
\begin{lemma} \label{lemma:uniform convergence of Busemann functions}
On $\Omega_1$, for each $\theta \in D$, $m \in \Z$, and $t \in \R$,
\[
\lim_{D \ni \gamma \rightarrow \theta} \h_{m}^\gamma(t) = \h_m^{\theta }(t)\qquad\text{and}\qquad\lim_{D \ni \gamma \rightarrow \theta} \vv_m^\gamma(t) = \vv_m^{\theta}(t).
\]
For each $m \in \Z$, the convergence is uniform in $t$ on compact subsets of $\R$.  Additionally, for each $m \in \Z$,
\[
\lim_{D \ni \gamma \rightarrow \infty} \h_m^\gamma(t) = B_m(t),\qquad\text{and}\qquad\lim_{D \ni \delta \searrow 0} \vv_m^\theta(t) =  0,
\]
uniformly in $t$ on compact subsets of $\R$. 
\end{lemma}

\begin{proof}
We first prove the statements for the $\h_m$. We show that 
\[
\lim_{D \ni \gamma \nearrow \theta} \h_m^\gamma(t)
\]
exists and equals $\h_m^\theta(t)$, uniformly in $t$ on compact subsets of $\R$.
The limits from the right (as well as the case of $B_m$ in place of $\h_m^\theta$) follow by analogous arguments. 

By rearranging the inequality of Lemma~\ref{lemma:summary of Busemanns fixed theta}\ref{itm:fixed theta monotonicity}, for $0 < \gamma < \theta < \infty$ and $a < b$ and any $t \in [a,b]$,
\begin{equation} \label{eqn:bounding diffs of busemann functions}
\h_m^\gamma(a) - \h_m^{\theta}(a) \le \h_m^\gamma(t) - \h_m^{\theta}(t) \le \h_m^\gamma(b) - \h_m^{\theta}(b), 
\end{equation}
and the inequality still holds if we replace $\h_m^\theta$ with $B_m$. Thus, on the event $\Omega_1$, $h_m^\gamma$ converges to $h_m^\theta$, uniformly in $t$ on compact subsets of $\R$. Now, we prove the convergence statements for $\vv_m$. By Lemma~\ref{lemma:summary of Busemanns fixed theta}\ref{itm:fixed theta Busemann queuing relations}, on $\Omega_0$, for all $t \in \R$, and all $\theta \in D$.
\begin{equation} \label{expression for v, h pulled out}
\vv_{m}^\theta(t) = Q(\h_{m}^\theta,B_{m - 1})(t) = \sup_{t \le s < \infty}\{ B_{m-1}(t,s)-\h_m^\theta(t,s)\}.
\end{equation}
Let $\omega \in \Omega_1$, and $t \in \R$. By the monotonicity of Lemma~\ref{lemma:summary of Busemanns fixed theta}\ref{itm:fixed theta monotonicity} and Lemma~\ref{monotonicity of maximizers from function monotonicity}, since $\tau_{(m,t)}^{\theta + 1}$ is a maximizer of $B_{m - 1}(s) - \h_m^{\theta + 1}(s)$ over $s \in [t,\infty)$, for all $\gamma \in D$ with $\gamma < \theta + 1$,
\[
\vv_m^\gamma(t) = \sup_{t \le s < \infty}\{B_{m -1}(t,s)-\h_{m}^\gamma(t,s) \} = \sup_{t \le s \le \tau_{(m,t)}^{\theta + 1}}\{B_{m -1}(t,s)-\h_{m}^\gamma(t,s)\}. 
\]
Since $\h_m^\gamma$ converges uniformly on compact sets to $\h_m^\theta$, $\vv_m^\gamma(t)$ converges pointwise to $\vv_m^\theta(t)$. By Lemma~\ref{lemma:summary of Busemanns fixed theta}\ref{itm:fixed theta monotonicity}, the convergence from both right and left is monotone, so by the continuity of Definition~\ref{def:Busemann functions, fixed theta} and  Dini's Theorem, the convergence is uniform.   

Lastly, for limits as $\delta \searrow 0$, we again apply Lemmas~\ref{lemma:summary of Busemanns fixed theta}\ref{itm:fixed theta monotonicity} and~\ref{monotonicity of maximizers from function monotonicity} so that, for $\delta \in D$ with $\delta \le 1$,
\begin{align*}
\vv_m^\delta(t) &= \sup_{t \le s < \infty}\{B_{m - 1}(t,s) - \h_{m}^\delta(t,s)\} = \sup_{t \le s \le \tau_{(m,t)}^\delta} \{B_{m - 1}(t,s) - \h_m^{\delta}(t,s)\} \\
&\le \sup_{t \le s \le \tau_{(m,t)}^\delta} \{B_{m - 1}(t,s) - \h_m^{1}(t,s)\},
\end{align*}
and the right-hand side converges to $0$ as $\delta \searrow 0$ by the continuity of $B_{m - 1} - \h_m^1$ and the convergence of $\tau_{(m,t)}^\delta$ to $t$ given in the definition of $\Omega_1$. Dini's Theorem again strengthens the pointwise convergence to uniform convergence. 
\end{proof}

By the additivity of Lemma~\ref{lemma:summary of Busemanns fixed theta}\ref{itm:fixed theta additivity}, for arbitrary $(m,s),(r,t) \in \Z \times \R$ with $m \le r$, 
\[
\B^\theta((m,s),(r,t)) = \h_m^\theta(s,t) + \sum_{k = m}^{r - 1} \vv_{k  + 1}(t).
\]
For $m > r$,  $\B^\theta((m,s),(r,t)) = -\B^\theta((r,t),(m,s))$, so $\B^\theta((m,s),(r,t))$ is still a sum of horizontal and vertical increments. Then, by Lemma~\ref{lemma:uniform convergence of Busemann functions},  on the event $\Omega_1$, for all $\theta \in D$ and $\mbf x,\mbf y \in \Z \times \R$,
\begin{equation} \label{eqn:contnuity of Busemann process on countable dense set} 
\lim_{D \ni\gamma \rightarrow \theta} \B^\gamma(\mbf x,\mbf y) = \B^\theta(\mbf x,\mbf y).
\end{equation}

\begin{definition} \label{general Busemann definition}
On the event $\Omega_1$, for an arbitrary $\theta > 0$ and $\mbf x,\mbf y \in \Z \times \R$, define the following. 
\begin{align*} 
&\B^{\theta-}(\mbf x,\mbf y) = \lim_{D \ni \gamma \nearrow \theta} \B^\gamma(\mbf x,\mbf y),\qquad \text{and}\qquad \B^{\theta+}(\mbf x,\mbf y) = \lim_{D \ni \delta \searrow \theta} \B^\delta(\mbf x,\mbf y). 
\end{align*}
\end{definition}
\begin{remark}
By additivity of Lemma~\ref{lemma:summary of Busemanns fixed theta}\ref{itm:fixed theta additivity} and the monotonicity of Lemma~\ref{lemma:summary of Busemanns fixed theta}\ref{itm:fixed theta monotonicity}, these limits exist for all $\mbf x,\mbf y \in \Z\times \R$ and $\theta > 0$. By~\eqref{eqn:contnuity of Busemann process on countable dense set}, on $\Omega_1$, for all $\theta \in D$ and $\mbf x,\mbf y \in \Z \times \R$, $\B^{\theta + }(\mbf x,\mbf y) = \B^{\theta-}(\mbf x,\mbf y) = \B^\theta(\mbf x,\mbf y)$.
\end{remark} 
\begin{proof}[Proof of Theorem~\ref{thm:summary of properties of Busemanns for all theta}]

\medskip
\noindent \textbf{Parts~\ref{general additivity Busemanns}--\ref{general monotonicity Busemanns}}: These follow by taking limits in Parts~\ref{itm:fixed theta additivity} and~\ref{itm:fixed theta monotonicity} of Lemma~\ref{lemma:summary of Busemanns fixed theta}. 

\medskip
\noindent \textbf{Parts~\ref{general uniform convergence Busemanns}--\ref{general continuity of Busemanns}}:
By Part~\ref{general monotonicity Busemanns}, it suffices to take limits along the countable dense set $D$. Then, Lemma~~\ref{lemma:uniform convergence of Busemann functions} establishes Parts~\ref{general uniform convergence:limits to infinity} and~\ref{general uniform convergence:limits to 0}. The monotonicity of Part~\ref{general monotonicity Busemanns} can be rearranged, just as in Equation~\eqref{eqn:bounding diffs of busemann functions}, which strengthens the pointwise convergence of the $\h_m^{\gamma \sig}$ to uniform convergence on compact sets. Thus, since  $\h_m^{\gamma}$ is continuous for $\gamma \in D$, $\h_m^{\theta \sig}$ is also continuous for $\sigg \in \{+,-\}$. Now, for the convergence of the $\vv_m^{\theta \sig}$, recall that on the event $\Omega_1$, for all $t \in \R$ and $\gamma \in D$,
\[
\vv_m^\gamma(t) = Q(\h_m^\gamma,B_{m - 1})(t) = \sup_{t \le s <\infty}\bigl\{B_{m - 1}(t,s) - \h_m^\gamma(t,s)\bigr\}.
\]
 By Part~\ref{general monotonicity Busemanns} and Lemma~\ref{monotonicity of maximizers from function monotonicity}, this supremum may be restricted to a common compact set for all $\gamma < \theta$. Then, since $\h_m^{\gamma}$ converges uniformly to $\h_m^{\theta -}$ as $D \ni \gamma \nearrow \theta$, $\vv_m^\gamma(t)$ converges pointwise to $Q(\h_m^{\theta -},B_{m - 1})(t)$ as $\gamma \nearrow \theta$. But, by definition, $\vv_m^{\theta - }(t) = \lim_{D \ni \gamma \nearrow \theta} \vv_m^\gamma(t)$, so $\vv_m^{\theta - }(t) = Q(\h_m^{\theta -},B_{m - 1})(t)$, which is continuous by continuity of $\h_m^{\theta -}$ and $B_{m - 1}$. Since the convergence is monotone, Dini's Theorem implies the convergence is uniform on compact sets. The proof for $\delta \searrow \theta$ is analogous. Part~\ref{general continuity of Busemanns} follows by continuity of the Busemann functions from Definition~\ref{def:Busemann functions, fixed theta}, the uniform convergence of Part~\ref{general uniform convergence Busemanns}, and the additivity in Part~\ref{general additivity Busemanns}.

\medskip
\noindent \textbf{Part~\ref{limits of B_m minus \h_{m + 1}}}:
 This follows from Lemma~\ref{lemma:summary of Busemanns fixed theta}\ref{itm:fixed theta limit of B_m - \h_{m + 1}} and the monotonicity of Part~\ref{general monotonicity Busemanns}.

 \medskip
 \noindent \textbf{Part~\ref{general queuing relations Busemanns}}:
 The equality for $\vv_{m + 1}^{\theta\sig}$ was shown in the proof of Part~\ref{general uniform convergence Busemanns}. The equality for $\h_m^\theta$ follows by the additivity of Part~\ref{general additivity Busemanns} and the same argument as in the proof of Lemma~\ref{lemma:summary of Busemanns fixed theta}\ref{itm:fixed theta Busemann queuing relations}.

\medskip \noindent \textbf{Part~\ref{independence structure of Busemann functions on levels}}: This follows because $\{h_r^{\theta \sig}:\theta > 0, \sigg \in \{+,-\},r > m\}$ is a function of $\{B_r\}_{r > m}$.

\medskip
\noindent \textbf{Part~\ref{busemann functions agree for fixed theta}}:
Fix $\theta > 0$. We first show that, for each $\theta > 0$ and $\mbf x,\mbf y \in \Z \times \R$, there exists an event $\Omega_{\mbf x,\mbf y}^{(\theta)}$ on which
    \begin{equation} \label{eqn:busemann equality}
    \B^{\theta -}(\mbf x,\mbf y)= \lim_{n \rightarrow \infty} \bigl[L_{\mbf x,(n,t_n)} - L_{\mbf y,(n,t_n)}\bigr] = \B^{\theta +}(\mbf x,\mbf y),
    \end{equation}
    for all sequences $\{t_n\}$ with $\lim_{n \rightarrow \infty} \f{t_n}{n} = \theta$.  By Theorem~\ref{thm:existence of Busemann functions for fixed points}, for each $\mbf x,\mbf y$, and $\theta > 0$, there exists an event of full probability on which the limit in~\eqref{eqn:busemann equality} exists and is independent of the choice of sequence. Then, by additivity, it is sufficient to show that, for each fixed $\theta > 0,m \in \Z$, and $t \in \R$, with probability one, 
    \begin{align*} 
    &\h_m^{\theta -}(t) = \lim_{n \rightarrow \infty} \bigl[L_{(m,0),(n,n\theta)} - L_{(m,t),(n,n\theta)}\bigr]= \h_m^{\theta+}(t), \text{ and} \\[1em]
    &\vv_m^{\theta -}(t) = \lim_{n \rightarrow \infty} \bigl[L_{(m - 1,t),(n,n\theta)} - L_{(m,t),(n,n\theta)}\bigr] = \vv_m^{\theta +}(t).
    \end{align*}
     We show that 
     \begin{equation} \label{horizontal limit, general theta}
     \h_m^{\theta -}(t) = \lim_{n \rightarrow \infty} \left[L_{(m,0),(n,n\theta)} - L_{(m,t),(n,n\theta)}\right] \qquad\text{a.s.},
     \end{equation}
     and the other statements follow by analogous arguments. By Theorem~\ref{thm:existence of Busemann functions for fixed points}, the right-hand side of~\eqref{horizontal limit, general theta} has distribution~$\Nor(\f{t}{\sqrt \theta},|t|)$. By definition, $
    \h_m^{\theta -}(t)$ is the limit, as $\gamma \nearrow \theta$, of $\h_m^\gamma(t)$, and $\h_m^\gamma(t)\sim \Nor(\f{t}{\sqrt \gamma},|t|)$, which converges weakly to $\Nor(\f{t}{\sqrt \theta},|t|)$. Using Theorem~\ref{lemma:summary of Busemanns fixed theta}\ref{itm:Busemann limits, fixed theta} and Lemma~\ref{lemma:bounds for differences of Brownian LPP}, for each fixed $t > 0$ and all $D \ni\gamma < \theta$,
    \[
    \h_m^{\gamma}(t) = \lim_{n \rightarrow \infty} \left[L_{(m,0),(n,n\gamma)} - L_{(m,t),(n,n\gamma)}\right] \ge \lim_{n \rightarrow \infty} \left[L_{(m,0),(n,n\theta)} - L_{(m,t),(n,n\theta)}\right] \qquad\text{a.s.}
    \]
    The inequality flips for $t < 0$. Hence, both sides in~\eqref{horizontal limit, general theta} have the same distribution, while one dominates the other, so they are equal with probability one. 
    
    \noindent Next, set
    \[
    \Omega^{(\theta)} = \bigcap_{\mbf x,\mbf y \in \Z \times \Q} \Omega_{\mbf x,\mbf y}^{(\theta)}.
    \]
     Then, using the continuity of Part~\ref{general continuity of Busemanns}, the desired conclusion follows by the same reasoning as in the proof of Lemma~\ref{lemma:summary of Busemanns fixed theta}\ref{itm:Busemann limits, fixed theta}. 
     
     \medskip \noindent \textbf{Part~\ref{itm:shift_invariance}} This follows from the construction of the Busemann functions as limits of BLPP times and the shift invariance of BLPP.
\end{proof}

\section{Proofs of the results for semi-infinite geodesics} \label{section:main proofs}
\noindent Throughout this section, $\Omega_1$ is the event defined in Lemma~\ref{lemma:Omega 1 prob 1} and referenced in Theorem~\ref{thm:summary of properties of Busemanns for all theta}.

\subsection{Key lemmas}
For two fields of Brownian motions, $\mathbf B$ and $\overline{\mathbf B}$, $\lambda > 0$, and an initial point $(m,s) \in \Z \times \R$, define the point-to-line last passage time:
\begin{equation} \label{point-to-line definition}
\overline L^\lambda_{(m,s),n}(\mathbf B,\overline{\mathbf B}) = \sup\Bigl\{\sum_{r = m}^n B_r(s_{k - 1},s_k) - \overline B_{n+1}(s_n) - \lambda s_n:\mathbf s_{m,n} \in \Pi_{(m,s),n}   \Bigr\},
\end{equation}
where $\Pi_{(m,s),n}$ denotes the set of sequences $s = s_{m - 1} \le s_m \le \cdots \le s_n$.  The following is due to Hammond~\cite{Hammond4}
\begin{lemma}[\cite{Hammond4}, Lemma B.2] \label{lm:point-to-line uniqueness}
Let $\mbf B$ be a field of independent, two-sided Brownian motions and $\overline{\mbf B}$ be any other field of Brownian motions such that, for each $n$, $\overline B_{n + 1}$ is independent of $\{B_m\}_{m \le n}$. Fix an initial point $(m,s)$ and let $n \ge m$. Then, with probability one, the quantity in~\eqref{point-to-line definition} is finite, and there is a unique sequence $\mbf s_{m,n} \in \Pi_{(m,s),n}$ that is maximal for~\eqref{point-to-line definition}.
\end{lemma}
\begin{remark}
Lemma~\ref{lm:point-to-line uniqueness} is stated slightly differently in~\cite{Hammond4}. In that paper, the function $s\mapsto-\overline B_{n + 1}(s) - \lambda s$ is replaced by an arbitrary deterministic and measurable function $h:\R \rightarrow \R \cup \{-\infty\}$ satisfying $h(t) > -\infty$ for some $t > s$ and $\limsup_{t \rightarrow \infty} h(t)/t < 0$. The assumption that $\{\overline B_{n + 1}\}_{n + 1}$ is independent of $\{B_m\}_{m \le n}$ allows us to condition on $\overline B_{n + 1}$ and obtain the desired result. 
\end{remark}
\pagebreak

\begin{lemma} \label{lemma:ptl_sig}
Let $\omega \in \Omega_1$, $(m,t) \in \Z \times \R$, $\theta > 0$, and $\sigg \in \{+,-\}$. Then, the following hold.
\begin{enumerate} [label=\rm(\roman{*}), ref=\rm(\roman{*})]  \itemsep=3pt
    \item \label{itm:geo_maxes} Let $\{\tau_r\}_{r = m - 1}^\infty$ be any sequence in $\mbf T_{(m,t)}^{\theta\sig}$. Then, for each $n \ge m$, the jump times $t = \tau_{m - 1} \le \tau_m \le \cdots \le \tau_n$ are a maximizing sequence for 
\begin{equation} \label{ptl_BLPP}
 \overline L_{(m,t),n}^\lambda(\mathbf B,\overline{\mathbf B}) := \sup\Biggl\{\sum_{r = m}^n B_r(s_{r - 1},s_r) - \h_{n +1}^{\theta\sig}(s_n): \mbf s_{m,n} \in \Pi_{(m,t),n}  \Biggr\}.
\end{equation}
Here,  $\overline L_{(m,t),n}^\lambda(\mathbf B,\overline{\mathbf B})$ is as in~\eqref{point-to-line definition}, where $\lambda = \f{1}{\sqrt \theta}$, $\mathbf B = \{B_r\}_{r \in \Z}$, and $\overline{\mathbf B} = \{\h_r^{\theta\sig}(t) - \f{t}{\sqrt \theta}: t \in \R\}_{r \in \Z}$.
\item \label{itm:maxes_geo}
 Conversely, for each $n \ge m$, whenever $t = t_{m - 1} \le t_m \le \cdots \le t_n$ is a maximizing sequence for~\eqref{ptl_BLPP}, there exists $\{\tau_r\}_{r = m - 1}^\infty \in \mbf T_{(m,t)}^{\theta \sig}$ such that $t_r = \tau_r$ for $m \le r \le n$.  
 \item \label{itm:LR_geo_max}
 For each $n \ge m$, the sequences $t = \tau_{(m,t),m - 1}^{\theta \sig,L} \le \cdots \le \tau_{(m,t),n}^{\theta \sig,L}$ and $t = \tau_{(m,t),m - 1}^{\theta \sig,R} \le \cdots \le \tau_{(m,t),n}^{\theta \sig,R}$ are, respectively, the leftmost and rightmost maximizing sequences for~\eqref{ptl_BLPP}.
 \end{enumerate}
\end{lemma}
\begin{proof}
\textbf{Part~\ref{itm:geo_maxes}:}
Recall by \eqref{semi-infinite geodesic succesive jumps}, that $\tau_m \ge t$ is a maximizer of
\[
B_m(t,s) - \h_{m + 1}^{\theta\sig}(s) \qquad \text{ over }\qquad s \in[t,\infty).
\]
Hence, the statement holds for $n = m$. Now, assume that the statement holds for some $n \ge m$. Then, $t = \tau_{m - 1} \le \cdots \le \tau_n$ satisfies
\begin{equation} \label{eqn:point-to-line maximizer}
\sum_{r = m}^n B_r(\tau_{r - 1},\tau_r) - \h_{n +1}^{\theta\sig}(\tau_n) = \sup\Biggl\{\sum_{r = m}^n B_r(s_{r - 1},s_r) - \h_{n +1}^{\theta\sig}(s_n): \mbf s_{m,n} \in \Pi_{(m,t),n}  \Biggr\}. 
\end{equation}
Using Theorem~\ref{thm:summary of properties of Busemanns for all theta}\ref{general queuing relations Busemanns} and rearranging the terms in the definition ~\eqref{definition of D} of the operator $D$, 
\begin{equation} \label{eqn:h n + 1 Busemann}
\h_{n + 1}^{\theta\sig}(s_n) = B_{n + 1}(s_n) + \sup_{0 \le s < \infty}\{B_{n + 1}(s) - \h_{n + 2}^{\theta\sig}(s)\} - \sup_{s_n \le s_{n + 1} < \infty}\{B_{n + 1}(s_{n + 1}) - \h_{n + 2}^{\theta\sig}(s_{n + 1})\},
\end{equation}
for $s_n \in \R$. Specifically, since $\tau_{n + 1}$ is a maximizer of $B_{n + 1}(s) - h_{n + 2}^{\theta\sig}(s)$ over $s \in [\tau_n,\infty)$,
\begin{equation} \label{eqn: h n + 1 Busemann substitute maximizer}
h_{n + 1}^{\theta \sig}(\tau_n) = -B_{n + 1}(\tau_n,\tau_{n + 1}) + h_{n + 2}^{\theta \sig}(\tau_{n +1}) + \sup_{0\le s < \infty}\{B_{n + 1}(s) - h_{n + 2}^{\theta \sig}(s)\}
\end{equation}
Substituting~\eqref{eqn:h n + 1 Busemann} and~\eqref{eqn: h n + 1 Busemann substitute maximizer} into~\eqref{eqn:point-to-line maximizer} and discarding the term $\sup_{0 \le s < \infty}\{B_{n + 1}(s) - \h_{n + 2}^{\theta\sig}(s)\}$ on both sides
\begin{align*}
&\qquad\sum_{r = m}^{n + 1}B_r(\tau_{r - 1},\tau_r) - h_{n + 2}^{\theta \sig}(\tau_{n +1}) \\[1em]
= &\sup\Biggl\{\sum_{r = m}^{n} B_r(s_{r - 1},s_r) - B_{n + 1}(s_n) + \sup_{s_n \le s_{n + 1} < \infty}\{B_{n + 1}(s_{n + 1}) - \h_{n + 2}^{\theta\sig}(s_{n + 1})\}: \mbf s_{m,n} \in \Pi_{(m,t),n}  \Biggr\} \\[1em]
= &\sup\Biggl\{\sum_{r = m}^{n +1 } B_r(s_{r - 1},s_r) - \h_{n +2}^{\theta\sig}(s_{n + 1}): \mbf s_{m,n + 1} \in \Pi_{(m,t),n + 1}  \Biggr\}. 
\end{align*}

\medskip \noindent \textbf{Part~\ref{itm:maxes_geo}:} We prove this part by induction. First, note that in the case $n = m$, maximizers of $B_n(t,s_m) - h_{m + 1}^{\theta\sig}(s_m)$ over $s_m \in [t,\infty)$ are precisely those that are the first jump times of the Busemann semi-infinite geodesics. Now, assume that the statement holds for $n$. We show that if $t = t_{m - 1}\le \ldots \le t_{n + 1}$ is a maximizing sequence for~\eqref{eqn:point-to-line maximizer} (with $n$ replaced by $n +1$), then $t_{n +1}$ is a maximizer of $B_{n + 1}(u) - h_{n + 2}^{\theta \sig}(u)$ over $u \in [t_n,\infty)$, and $t = t_{m - 1} \le \cdots \le t_{n}$ is a maximizing sequence for~\eqref{eqn:point-to-line maximizer}.

With this procedure mapped out, observe that, as in the proof of Part~\ref{itm:geo_maxes}, 
\begin{align*}
&\sup\Biggl\{\sum_{r = m}^{n +1 } B_r(s_{r - 1},s_r) - \h_{n +2}^{\theta\sig}(s_{n + 1}): \mbf s_{m,n + 1} \in \Pi_{(m,t),n + 1}  \Biggr\} \\[1em]
=&\sup\Biggl\{\sum_{r = m}^{n} B_r(s_{r - 1},s_r) - B_{n + 1}(s_n) + \sup_{s_n \le s_{n + 1} < \infty}\{B_{n + 1}(s_{n + 1}) - \h_{n + 2}^{\theta\sig}(s_{n + 1})\}: \mbf s_{m,n} \in \Pi_{(m,t),n}  \Biggr\} 
\end{align*}
Hence, any maximizing sequence $t = t_{m - 1}\le \cdots \le t_{n + 1}$ must satisfy 
\[
B_{n +1}(t_{n + 1}) - h_{n + 2}^{\theta\sig}(t_{n + 1}) = \sup_{t_{n} \le u < \infty}\{B_{n + 1}(u) - h_{n + 2}^{\theta \sig}(u)\}.
\]
Furthermore, $t = t_{m - 1} \le \cdots \le t_{n}$ is a maximizing sequence for 
\[
\sum_{r = m}^{n}B_r(s_{r - 1},s_r) - B_{n + 1}(s_{n}) + \sup_{s_{n} \le u <\infty}\{B_{n + 1}(u) - h_{n + 2}^{\theta \sig}(u)\}
\]
over all sequences $\mbf s_{m,n} \in \Pi_{(m,t),n}$ Subtracting off a constant, $t = t_{m - 1} \le \cdots \le t_{n}$ is also a maximizing sequence for
\begin{align*}
&\qquad \sum_{r = m}^{n} B_r(s_{r - 1},s_r) -\Big( B_{n + 1}(s_{n}) + \sup_{0 \le u <\infty}\{B_{n + 1}(u) - h_{n + 2}^{\theta \sig}(u)\} - \sup_{s_{n} \le u <\infty}\{B_{n + 1}(u) - h_{n + 2}^{\theta \sig}(u)\}\Big)  \\[1em]
&= \sum_{r = m}^{n} B_r(s_{r - 1},s_r) - h_{n + 1 }^{\theta \sig}(s_{n}),
\end{align*}
where the last line comes by~\eqref{eqn:h n + 1 Busemann}. This completes the inductive step. 

\medskip \noindent \textbf{Part~\ref{itm:LR_geo_max}:} This follows by Parts~\ref{itm:geo_maxes} and~\ref{itm:maxes_geo} since $t = \tau_{(m,t),m - 1}^{\theta \sig,L} \le \tau_{(m,t),m}^{\theta \sig,L} \le \cdots$ and $t = \tau_{(m,t),m - 1}^{\theta \sig,R} \le \tau_{(m,t),m}^{\theta \sig,R} \le \cdots$ are respectively the leftmost and rightmost sequences in $\mbf T_{(m,t)}^{\theta \sig}$. 
\end{proof}

\begin{lemma} \label{lemma:equality of busemann to weights of BLPP}
Let $\omega \in \Omega_1$, $(m,t) \in \Z \times \R, \theta > 0$, and $\sigg \in \{+,-\}$, and $\{\tau_r\}_{r \ge m - 1} \in \mbf T_{(m,t)}^{\theta \sig}$. Then, for all $r \ge m$,
\[
\vv_{r + 1}^{\theta \sig}(\tau_r) = 0,\qquad\text{and}\qquad \h_r^{\theta \sig}(u,v) = B_r(u,v) \text{ for all }u,v \in [\tau_{r - 1},\tau_r].
\]
Furthermore, the following identities hold for $r \ge m$.
     \begin{align}
     &\tau_{(m,t),r}^{\theta \sig,L} = \inf\bigl\{u \ge \tau^{\theta \sig,L}_{(m,t),r - 1}: \vv_{r + 1}^{\theta \sig}(u) = 0\bigr\} \text{ and } \label{eqn:inf_sig} \\[1em]
     &\tau^{\theta \sig,R}_{(m,t),r} = \sup\bigl\{u \ge \tau^{\theta \sig,R}_{(m,t),r - 1}: \h_r^{\theta \sig}(\tau^{\theta \sig,R}_{(m,t),r - 1},u) = B_r(\tau^{\theta \sig,R}_{(m,t),r - 1},u)\bigr\} \label{eqn:sup_sig}
     \end{align}
     More specifically, if $u \ge \tau^{\theta \sig,R}_{(m,t),r - 1}$, then $\h_r^{\theta \sig}(\tau^{\theta \sig,R}_{(m,t),r - 1},u) = B_r(\tau^{\theta \sig,R}_{(m,t),r - 1},u)$ if and only if $u \le \tau^{\theta \sig,R}_{(m,t),r}$.
\end{lemma}
\begin{remark}
The first part of Lemma~\ref{lemma:equality of busemann to weights of BLPP} says that, along any semi-infinite path in $\mbf T_{(m,t)}^{\theta \sig}$, the Busemann process agrees with the energy of the semi-infinite geodesic. That is, at every vertical jump from $(r,\tau_r)$ to $(r + 1,\tau_r)$, the vertical Busemann function $v_{r + 1}^{\theta \sig}$ equals zero. This is as it should  because, according to \eqref{E10}, a  vertical step of a path  does not collect any energy from the environment. Along each horizontal step from $(r,\tau_{r - 1})$ to $(r,\tau_r)$, the increment of the horizontal Busemann function $h_r^{\theta \sig}$ agrees with the increment of the Brownian motion $B_r$.  Equations~\eqref{eqn:inf_sig} and~\eqref{eqn:sup_sig} are more subtle. Let $s_{r - 1}$ be the time when a $\theta \sig$ geodesic jumps from level $r - 1$ to level $r$. Equation~\eqref{eqn:inf_sig} says that the leftmost $\theta \sig$ geodesic jumps from level $r$ to $r + 1$ at the \textit{first} time $u \ge s_{r - 1}$ such that $v_{r + 1}^{\theta \sig}(u) = 0$. Equation~\eqref{eqn:sup_sig} says that the rightmost geodesic jumps from level $r$ to $r + 1$ at the \textit{last} time $u \ge s_{r - 1}$ such that $B_r(s_{r -1},u) = h_r^{\theta \sig}(s_{r - 1},u)$. At all subsequent times $v > u$ the equality is lost, and we have  $B_r(s_{r-1},v) < h_r^{\theta \sig}(s_{r - 1},v)$.
\end{remark}
\begin{proof}[Proof of Lemma~\ref{lemma:equality of busemann to weights of BLPP}]
 By Theorem~\ref{thm:summary of properties of Busemanns for all theta}\ref{general queuing relations Busemanns} and the definition of $\tau_r$ as a maximizer,
 \begin{align*}
0&\le \vv_{r+1}^{\theta \sig}(\tau_r) = \sup_{\tau_r\le s<\infty}\{ B_r(\tau_r,s) - \h_{r+1}^{\theta \sig}(\tau_r, s)\} \\[1em]
&= \sup_{\tau_r\le s<\infty}\{ B_r(s) - \h_{r+1}^{\theta \sig}(s)\} - [B_r(\tau_r) - \h_{r+1}^{\theta \sig}(\tau_r)]  \\[1em]
&\le \sup_{\tau_{r-1} \le s<\infty}\{ B_r(s) - \h_{r+1}^{\theta \sig}(s)\} - [ B_r(\tau_r) - \h_{r+1}^{\theta \sig}(\tau_r) ]
=0. 
\end{align*}
 To establish~\eqref{eqn:inf_sig}, assume, by way of contradiction, that for some $\tau_{(m,t),r - 1}^{\theta \sig,L} \le u < \tau_{(m,t),r}^{\theta \sig,L}$,
 \[
 0 = \vv_{r + 1}^{\theta \sig}(u)  = \sup_{u \le s < \infty}\{B_r(u,s) - h_{r + 1}^{\theta \sig}(u,s)\}.
 \]
 Then, 
 \[
 B_m(u) - \h_{m + 1}^{\theta \sig}(u) \ge B_m(s) - \h_{m + 1}^{\theta \sig}(s)
 \]
 for all $s\ge u$, and specifically for $s = \tau_{(m,t),r}^{\theta \sig,L}$. This contradicts the definition of $\tau_{(m,t),r}^{\theta \sig,L}$ as the leftmost maximizer of $B_r(s) - h_{r + 1}^{\theta \sig}(s)$ over $s \in [\tau_{(m,t),r - 1}^{\theta \sig,L},\infty)$. 

By Theorem~\ref{thm:summary of properties of Busemanns for all theta}\ref{general queuing relations Busemanns},
\be \label{hrBe}
\h_r^{\theta \sig}(u,v) = B_r(u,v) + \sup_{u \le s < \infty} \{B_r(s) - \h_{r + 1}^{\theta \sig}(s)\} - \sup_{v \le s < \infty}\{B_r(s) - \h_{r + 1}^{\theta \sig}(s)\},
\ee
so, since $\tau_r$ maximizes $B_r(s) - \h_{r +1}^{\theta\sig}(s)$ on $[\tau_{r-1},\infty)$, the two supremum terms above are both equal to $B_r(\tau_r) - h_{r + 1}^{\theta \sig}(\tau_r)$ whenever $u,v \in [\tau_{r - 1},\tau_r]$. Therefore, $h_r^{\theta \sig}(u,v) = B_r(u,v)$. Now we  establish~\eqref{eqn:sup_sig}. If, for some $r \ge m$,~\eqref{eqn:sup_sig} fails, then by~\eqref{hrBe}, for some $u > \tau_{(m,t),r}^{\theta \sig,R}$,
\[
\sup_{\tau^{\theta \sig,R}_{(m,t),r - 1} \le s <\infty}\{B_r(s) - \h_{r + 1}^{\theta \sig}(s)\} = \sup_{u \le s <\infty}\{B_r(s) - \h_{r + 1}^{\theta \sig}(s)\}.
\]
Thus, all maximizers  of $B_r(s) - \h_{r + 1}^{\theta \sig}(s)$ over $s \in [u,\infty)$ are also maximizers over the larger set $[\tau_{(m,t),r - 1}^{\theta \sig,R},\infty)$. But this is a contradiction because $\tau_{(m,t),r}^{\theta \sig,R}$ is the rightmost maximizer over $s \in [\tau_{(m,t),r - 1}^{\theta \sig,R},\infty)$.
\end{proof}
\noindent Recall, by Remark~\ref{rmk:theta+ = theta-}, that on the event $\Omega^{(\theta)} \subseteq \Omega_1$ of Theorem~\ref{thm:summary of properties of Busemanns for all theta},  $\mbf T_{\mbf x}^{\theta +} = \mbf T_{\mbf x}^{\theta -}=\mbf T_{\mbf x}^{\theta}$ for all $\mbf x\in \Z \times \R$. 

\begin{lemma} \label{lemma: uniqueness and directedness of geodesics for fixed parameters}
    Fix $\theta > 0$ and  $\mbf x = (m,t) \in \Z \times \R$. Then there exists an event $\wt \Omega_{\mbf x}^{(\theta)} \subseteq \Omega^{(\theta)}$, of probability one, on which the set $\mbf T_{\mbf x}^\theta$ contains exactly one sequence $\{\tau_{r}\}_{r \ge m - 1}$. This sequence satisfies
    \[
    \lim_{n \rightarrow \infty} \f{\tau_n}{n} = \theta.
    \]
\end{lemma}
To prove this lemma, we need some machinery from~\cite{blpp_utah}. For two fields of Brownian motions $\mbf B$ and $\overline{\mbf B}$  and a subset $A\subset\R$, define \begin{equation} \label{point-to-line last jump restricted}
\overline L^\lambda_{(m,s),n}(\mathbf B,\overline{\mathbf B}; s_n \in A) = \sup\Bigl\{\sum_{r = m}^n B_r(s_{r -1},s_r) - \overline B_{n +1}(s_n) - \lambda s_n: \mathbf s_{m,n} \in \Pi_{(m,s),n}, \, s_n \in A     \Bigr\}.
\end{equation}
The only difference between this definition and~\eqref{point-to-line definition} is the restriction on the $s_n$.

\begin{lemma}[\cite{blpp_utah}, page 1949] \label{point to line shape theorem}
Let $\mbf B$ be a field of independent, two-sided Brownian motions and $\overline{\mbf B}$ an arbitrary field of Brownian motions.  Fix $s \in \R$, $0 \le S \le T \le \infty$, and $m\in \Z$. Then, with probability one,
\[
\lim_{n \rightarrow \infty} n^{-1} \overline L^\lambda_{(m,s),n}(\mathbf B,\overline{\mathbf B}; s + nS \leq s_n \leq s + nT) = \sup_{S \leq t \leq T} \bigl\{2 \sqrt t - \lambda t  \bigr\}.
\]
\end{lemma}
\begin{remark}
The appearance of the term $2\sqrt t$ in Lemma~\ref{point to line shape theorem} comes from the shape theorem for BLPP. Namely, the following limit holds with probability one.
\[
\lim_{n \rightarrow \infty} n^{-1} L_{(0,0),(n,nt)} = 2 \sqrt t.
\]
This almost sure convergence was first proved in~\cite{Concentration_Results} (see also~\cite{Moriarty-OConnell-2007} for an alternative proof).
If $\overline{\mbf B}_{n + 1}$ is independent of $\mbf B_m,\ldots,\mbf B_n$, then the quantity $\overline L^\lambda_{(m,s),n}(\mathbf B,\overline{\mathbf B})$ is distributed as the sum of $n - m + 1$ independent exponential random variables with rate $\lambda$ and has the interpretation as the sum of vertical increments in the increment-stationary BLPP model. (see Appendix~\ref{section:queue and stationary} of the present paper and Section 4 of~\cite{brownian_queues}). Then, in the case $S = 0$ and $T = \infty$, Lemma~\ref{point to line shape theorem} degenerates to an application of the law of large numbers, namely
\[
\lim_{n \rightarrow \infty} n^{-1} \overline L^\lambda_{(m,s),n}(\mathbf B,\overline{\mathbf B}) = \lambda^{-1}.
\]
\end{remark}
\begin{lemma}[\cite{blpp_utah}, Lemma 4.7] \label{Utah Lemma 4.7}
Let $\mbf B$ and $\overline{\mbf B}$ satisfy the same conditions of Lemma~\ref{point to line shape theorem}. Fix $\lambda > 0$, $s \in \R$, and $m \in \Z$. If $0 < \theta < \lambda^{-2}$, then there exist a nonrandom  $\ve = \ve(\lambda,\theta) > 0$ such that, with probability one, for all sufficiently large $n$,
\[
\overline L^\lambda_{(m,s),n}(\mathbf B,\overline{\mathbf B};s_n \le s + n\theta) + n\ve < \overline L^\lambda_{(m,s),n}(\mathbf B,\overline{\mathbf B}).
\]
Similarly, if $\theta > \lambda^{-2}$, there exists a nonrandom $\ve = \ve(\lambda,\theta) > 0$ such that, with probability one, for all sufficiently large  $n$,
\[
\overline L^\lambda_{(m,s),n}(\mathbf B,\overline{\mathbf B};s_n \ge s + n\theta) + n\ve < \overline L^\lambda_{(m,s),n}(\mathbf B,\overline{\mathbf B})
\]
\end{lemma}
Lemma~\ref{Utah Lemma 4.7} is slightly stronger than the result stated in~\cite{blpp_utah}, so we include a proof.
\begin{proof}
We prove the first statement, and the second follows analogously.
The unique maximum of $2\sqrt t - \lambda t$ for $t \in [0,\infty)$ is achieved at $t = \lambda^{-2}$. Then, by Lemma~\ref{point to line shape theorem} and the assumption $\theta < \lambda^{-2}$,
\[
\lim_{n \rightarrow \infty} n^{-1} \overline L^\lambda_{(m,s),n}(\mbf B, \overline{\mbf B};s_n \leq s + n\theta) = \sup_{0 \leq t \leq \theta} \bigl\{2 \sqrt t - \lambda t  \bigr\} < \sup_{0 \le t < \infty } \bigl\{2 \sqrt t - \lambda t  \bigr\}= \lim_{n \rightarrow \infty} n^{-1} \overline L^\lambda_{(m,s),n}(\mbf B, \overline{\mbf B}).
\]
Hence, for sufficiently large $n$,  
\[
\overline L^\lambda_{(m,s),n}(\mbf B, \overline{\mbf B}; s_n \leq s + n\theta) + n\ve < \overline L^\lambda_{(m,s),n}(\mbf B, \overline{\mbf B}),
\]
where 
\[
\ve = \frac12\Bigl(\;{\sup_{0 \le t < \infty } \bigl\{2 \sqrt t - \lambda t  \bigr\} - \sup_{0 \leq t \leq \theta} \bigl\{2 \sqrt t - \lambda t  \bigr\}}\Bigr). \qedhere
\]
\end{proof}

\begin{proof}[Proof of Lemma~\ref{lemma: uniqueness and directedness of geodesics for fixed parameters}]
 By Lemma~\ref{lemma:ptl_sig}\ref{itm:geo_maxes}, for a sequence $\{\tau_r\}_{r \ge m - 1} \in \mbf T_{(m,t)}^{\theta}$, the jump times $t = \tau_{m - 1} \le \cdots \le \tau_n$ are a maximizing sequence for the point-to-line last passage time~\eqref{ptl_BLPP}. By Theorem~\ref{thm:summary of properties of Busemanns for all theta}\ref{independence structure of Busemann functions on levels}, $h_{n + 1}^{\theta}$ and $\{B_r\}_{r \le n}$ are independent for each $n$. This allows us to apply the almost sure uniqueness of point-to-line last-passage maximizers in Lemma~\ref{lm:point-to-line uniqueness}.
 Let $\wt \Omega_{\mbf x}^{(\theta)}$ be the event on which these maximizers are unique and on which 
 $\tau_n/n\to\theta$ 
 for the (now) almost surely unique sequence  $\{\tau_r\}_{r \ge m -1} \in \mbf T_{(m,t)}^\theta$. It remains to show that 
 \[
 \Pp\bigl(\lim_{n \rightarrow \infty}\f{\tau_n}{n} = \theta\bigr) = 1.
 \]
 For $\gamma  < \theta = \lambda^{-2}$, Lemma~\ref{Utah Lemma 4.7} guarantees that, with probability one, for all sufficiently large $n$,
\[
\overline L^\lambda_{(m,t),n}(\mbf B,\overline{\mbf B};s_n \le s +  n\gamma) < \overline L^\lambda_{(m,t),n}(\mbf B,\overline{\mbf B}).
\]
Therefore, by Lemma~\ref{lemma:ptl_sig}\ref{itm:geo_maxes}, $\tau_n  > s + n\gamma$ for all sufficiently large $n$. Thus, for $\gamma < \theta$,
\[
\Pp\bigl(\liminf_{n\rightarrow \infty}\f{\tau_n}{n} \ge \gamma\bigr) = 1,\qquad\text{so by taking }\gamma \nearrow \theta,\qquad\Pp\bigl(\liminf_{n \rightarrow \infty}\f{\tau_n}{n} \ge \theta\bigr) = 1.
\]
A symmetric argument using the second statement of Lemma~\ref{Utah Lemma 4.7} shows that 
\[
\Pp\bigl(\limsup_{n \rightarrow \infty}\f{\tau_n}{n} \le \theta\bigr) = 1. \qedhere
\]
\end{proof}

\subsection{Proofs of Theorems~\ref{existence of semi-infinite geodesics intro version}--\ref{thm:non_unique_size}:} \label{section:most_proofs}
We now begin to prove the theorems of Section~\ref{sec:Buse_geod}. 
First, we define the event $\Omega_2$ used in the theorems. Let $\wt \Omega_{\mbf x}^{(\gamma)}$ be the events of Lemma~\ref{lemma: uniqueness and directedness of geodesics for fixed parameters}. For the countable dense set $D\subseteq (0,\infty)$ of Section~\ref{section:Busemann construction}, set
\be \label{eqn:Omega2_def}
 \Omega_2 := \bigcap_{\gamma \in D,\mbf x \in \Z \times \Q} \wt \Omega_{\mbf x}^{(\gamma)} 
\ee
Then, $\Pp(\Omega_2) = 1$ and $\Omega_2 \subseteq \bigcap_{\theta \in \Q_{>0}}\Omega^{(\theta)} \subseteq \Omega_1$.

\begin{proof}[Proof of Theorem~\ref{existence of semi-infinite geodesics intro version}]

\medskip \noindent \textbf{Part~\ref{energy of path along semi-infinte geodesic}}:
On the event $\Omega_2$, let $(m,t) \in \Z \times \R$, $\theta > 0$ and  $\sigg \in \{+,-\}$, and let $t = \tau_{m - 1} \le \tau_{m} \le \cdots$ be a sequence in $\mbf T_{(m,t)}^{\theta \sig}$. Let $\Gamma$ be the associated path. By Lemma~\ref{lemma:equality of busemann to weights of BLPP}, 
\begin{equation} \label{eqn:equality of weights to Busemann functions}
\vv_{r +1}^{\theta \sig}(\tau_r) = 0,\qquad \text{and}\qquad \h_r^{\theta \sig}(s,t) = B_r(s,t) \text{ for } \tau_{r - 1} \le s \le t \le \tau_r. 
\end{equation}
 
We take $\mbf x = \mbf y = (m,t)$ and $\mbf z =(n ,\tau_n)$. The case for general $\mbf y \le \mbf z$ along the path $\Gamma$ follows by the same argument.  By~\eqref{eqn:equality of weights to Busemann functions} and additivity of Busemann functions (Theorem~\ref{thm:summary of properties of Busemanns for all theta}\ref{general additivity Busemanns}), the energy of path $\Gamma$ between $\mbf x$ and $\mbf y$ is given by
\begin{align*}
&\sum_{r = m}^{n} \bigl(B_r(\tau_{r - 1},\tau_r)\bigr) 
= \sum_{r = m }^{n - 1} \bigl(\h_r^{\theta \sig}(\tau_{r - 1},\tau_r)+ \vv_{r + 1}^{\theta\sig}(\tau_r) \bigr) + h_n^{\theta\sig}(\tau_{n -1},\tau_n) = \B^{\theta \sig}((m,t),(n,\tau_n)).
\end{align*}
Let jump times $t = s_{m - 1} \le s_m \le s_{m + 1} \le \cdots \le s_{n} = \tau_n$ define any other path between $(m,t)$ and $(n,\tau_n)$. Then, by Theorem~\ref{thm:summary of properties of Busemanns for all theta}, Parts~\ref{general additivity Busemanns} and~\ref{general monotonicity Busemanns}, the energy of this path is
\begin{align}
\sum_{r = m}^n B_r(s_{r - 1},s_r) 
\le \sum_{r = m}^{n - 1}\bigl( \h_r^{\theta \sig}(s_{r - 1},s_r) + \vv_{r + 1}^{\theta \sig}(s_r) \bigr) + \h_{n}^{\theta \sig}(s_{n - 1},\tau_n)  = \B^{\theta \sig}((m,t),(n,\tau_n)). \label{energy along any path is less than busemann between points}
\end{align}

\medskip
\noindent \textbf{Part~\ref{Leftandrightmost}}: By Theorem~\ref{thm:summary of properties of Busemanns for all theta}\ref{general monotonicity Busemanns}, equality holds in Equation~\eqref{energy along any path is less than busemann between points} only if $B_r(s_{r - 1},s_r) = \h_r^{\theta \sig}(s_{r - 1},s_r)$ for $m \le r \le n$ and $\vv_{r + 1}^{\theta \sig}(s_r) = 0$ for $m \le r \le n - 1$. Then, the statement follows by the Equations~\eqref{eqn:inf_sig} and~\eqref{eqn:sup_sig} of Lemma~\ref{lemma:equality of busemann to weights of BLPP}.

\medskip
\noindent \textbf{Part~\ref{itm:monotonicity of semi-infinite jump times}\ref{itm:monotonicity in theta}}: The key is Theorem~\ref{thm:summary of properties of Busemanns for all theta}\ref{general monotonicity Busemanns}. We show that, for $r \ge m$,
\[
\tau_r^- := \tau_{(m,t),r}^{\theta-,L} \le \tau_{(m,t),r}^{\theta +,L} =: \tau_r^+,
\]
and all other inequalities of the statement follow by the same procedure. By definition, $\tau_{m - 1}^- = \tau_{m - 1}^+ = t$. Inductively, assume that $\tau_r^- \le \tau_r^+$ for some $r \ge m$. We use the notation $\Largsup$ to denote leftmost maximizer. Then,
\begin{align*}
\tau_{r + 1}^- &= \Largsup_{\tau_{r^-}\le s < \infty} \{B_r(s)-\h_{r +1}^{\theta-,L}(s)\} \le \Largsup_{\tau_r^+ \le s < \infty} \{B_r(s)-\h_{r +1}^{\theta-,L}(s)\} \\[1em]
&\le \Largsup_{\tau_r^+ \le s < \infty} \{B_r(s) - \h_{r +1}^{\theta+,L}(s)\} = \tau_{r + 1}^+.
\end{align*}
The first inequality above holds because $\tau_r^- \le \tau_r^+$. The second inequality is an application of Lemma~\ref{monotonicity of maximizers from function monotonicity}, using the fact that $h_m^{\theta-}(s,t) \le h_m^{\theta +}(s,t)$ for all $s \le t$ (Theorem~\ref{thm:summary of properties of Busemanns for all theta}\ref{general monotonicity Busemanns}).

\medskip \noindent \textbf{Part~\ref{itm:monotonicity of semi-infinite jump times}\ref{itm:monotonicity in t}:} Let $s < t$. We show that 
\[
\tau_{s,r} := \tau_{(m,s),r}^{\theta\sig,L} \le \tau_{(m,t),r}^{\theta \sig,L} =: \tau_{t,r},
\]
and the statement with `$L$' replaced by `$R$' has an analogous proof. Again, the base case of $r = m - 1$ follows by definition. Assume the inequality holds for some $r \ge m - 1$. Then,
\[
\tau_{s,r + 1} = L\arg\sup_{\tau_{s,r} \le u <\infty}\{ B_m(u) - \h_{m + 1}^{\theta \sig}(u)\} \le L\arg\sup_{\tau_{t,r} \le u < \infty}\{ B_m(u) - \h_{m + 1}^{\theta \sig}(u)\} = \tau_{t,r + 1}. 
\]

\medskip \noindent \textbf{Part~\ref{itm:monotonicity of semi-infinite jump times}\ref{itm:strong monotonicity in t}:} The proof of this item is postponed until the very end of Section~\ref{section:main proofs}. This item is not used in any subsequent proofs.

\medskip \noindent \textbf{Part~\ref{itm:convergence of geodesics}\ref{itm:limits in theta}}: The monotonicity of Part~\ref{itm:monotonicity of semi-infinite jump times}\ref{itm:monotonicity in theta} ensures that the limits exist and that
\begin{equation} \label{eqn: monotonicity of limit}
\tau_r := \lim_{\gamma \nearrow \theta} \tau_{(m,t),r}^{\gamma \sig, L} \le \tau_{(m,t),r}^{\theta -,L}.
\end{equation}
We prove equality in the above expression, and the other statements follow analogously. By Lemma~\ref{lemma:ptl_sig}\ref{itm:geo_maxes}, for any $n \ge m$, $t = \tau_{(m,t),m - 1}^{\gamma \sig,L} \le \ldots \le \tau_{(m,t),n}^{\gamma \sig,L}$ is a maximizing sequence for
\[
\overline L_{(m,t),n}^{\f{1}{\sqrt \gamma}}(\mathbf B,\overline{\mathbf B}) := \sup\Bigl\{\sum_{r = m}^n B_r(s_{r - 1},s_r) - \h_{n +1}^{\gamma\sig}(s_n): \mbf s_{m,n} \in \Pi_{(m,t),n}  \Bigr\}.
\]
Using the monotonicity of Part~\ref{itm:monotonicity of semi-infinite jump times}\ref{itm:limits in theta} again, for all $\gamma \le \theta$, the supremum may be restricted to the compact subset of $\Pi_{(m,t),n}$ such that $s_n \le \tau_{(m,t),n}^{(\theta + 1)-,L}$. By Theorem~\ref{thm:summary of properties of Busemanns for all theta}\ref{general uniform convergence Busemanns}\ref{general uniform convergence:limits from left}, as $\gamma \nearrow \theta$, $h_{n + 1}^{\gamma \sig}(s)$ converges to $h_{n +1}^{\theta -}(s)$, uniformly over $s \in [t,\tau_{(m,t),n}^{(\theta + 1)-,L}]$. Then, by Lemma~\ref{lemma:convergence of maximizers from converging functions}, $t = \tau_{m - 1}\le \ldots\le \tau_n$ is a maximizing sequence for 
\[
\overline L_{(m,t),n}^{\f{1}{\sqrt \theta}}(\mathbf B,\overline{\mathbf B}) := \sup\Bigl\{\sum_{r = m}^n B_r(s_{r - 1},s_r) - \h_{n +1}^{\theta-}(s_n): \mbf s_{m,n} \in \Pi_{(m,t),n}  \Bigr\}.
\]
 Part~\ref{itm:LR_geo_max} of Lemma~\ref{lemma:ptl_sig} then implies the inequality~\eqref{eqn: monotonicity of limit} must be an equality. The statement for limits as $\delta \searrow \theta$ follows by the same reasoning. 
 
 \medskip \noindent \textbf{Part~\ref{itm:convergence of geodesics}\ref{itm:limits in theta to infty}}: The proof of this item is postponed until the very end of Section~\ref{section:main proofs}. This item is not used in any subsequent proofs.
 
 \medskip \noindent \textbf{Part~\ref{itm:convergence of geodesics}\ref{itm:limits in t}}:
This follows the same proof as that of Part~\ref{itm:convergence of geodesics}\ref{itm:limits in theta}, replacing the use of Part~\ref{itm:monotonicity of semi-infinite jump times}\ref{itm:monotonicity in theta} with Part~\ref{itm:monotonicity of semi-infinite jump times}\ref{itm:monotonicity in t} and replacing the use of Lemma~\ref{lemma:convergence of maximizers from converging functions} with Lemma~\ref{lemma:convergence of maximizers from converging sets}.

\medskip \noindent \textbf{Part~\ref{general limits for semi-infinite geodesics}}:  Let $\omega \in \Omega_2$, $\theta > 0$, $(m,t) \in \Z \times \Q$, $(\tau_r)_{r \ge m - 1} \in \mbf T_{(m,t)}^{\theta \sig}$, $\ve > 0$, and let $\gamma,\delta \in D$ be such that $\theta -\ve < \gamma < \theta < \delta < \theta + \ve$.  Part~\ref{itm:monotonicity of semi-infinite jump times}\ref{itm:monotonicity in theta} and Lemma~\ref{lemma: uniqueness and directedness of geodesics for fixed parameters}  imply that
\[
\limsup_{n \rightarrow \infty} \f{\tau_n}{n} \le \limsup_{n \rightarrow \infty} \f{\tau_{(m,t),n}^{\theta +,R}}{n} \le \limsup_{n \rightarrow \infty}\f{\tau_{(m,t),n}^{\delta +,R}}{n}  = \delta < \theta + \ve.
\]
By a similar argument, on $\Omega_2$,
\[
\liminf_{n \rightarrow \infty}\f{\tau_n}{n} \ge \theta - \ve.
\]
An analogous application of Part~\ref{itm:monotonicity of semi-infinite jump times}\ref{itm:monotonicity in t} extends Theorem~\ref{existence of semi-infinite geodesics intro version}\ref{general limits for semi-infinite geodesics} to all $\mbf x \in \Z \times \R$, on the event $\Omega_2$. 
\end{proof}

\begin{proof}[Proof of Theorem~\ref{thm:convergence and uniqueness}]
\medskip \noindent \textbf{Part~\ref{control of finite geodesics}}: Let $\omega \in \Omega_2$, $\theta > 0, (m,t) \in \Z \times \R$, and $\{t_{n,r}\}_{n \ge m, m - 1 \le r \le n}$ satisfy the given assumptions. Let $0 < \gamma < \theta < \delta$ be arbitrary. Then, by assumption that $t_n/n \rightarrow \theta$ and Theorem~\ref{existence of semi-infinite geodesics intro version}\ref{general limits for semi-infinite geodesics}, there exists $N \in \Z$ such that for all $n \ge N$, 
\begin{equation} \label{eqn:squeeze finite geod}
\tau_{(m,t),n}^{\gamma-,L} < t_n < \tau_{(m,t),n}^{\delta+,R}.
\end{equation}
By Theorem~\ref{existence of semi-infinite geodesics intro version}\ref{Leftandrightmost}, the sequence $t = \tau_{(m,t),m - 1}^{\gamma-,L} \le \tau_{(m,t),m}^{\gamma-,L} \le \cdots \le \tau_{(m,t),n}^{\gamma-,L}$ defines the leftmost geodesic between $(m,t)$ and $(n,\tau_{(m,t),n}^{\gamma-,L})$, and $t = \tau_{(m,t),m - 1}^{\delta+,R} \le \tau_{(m,t),m}^{\delta+,R} \le \cdots \tau_{(m,t),n}^{\delta+,R}$ defines the rightmost geodesic between $(m,t)$ and $(n,\tau_{(m,t),n}^{\delta+,R})$.  Let $t = t_{m - 1}^L \le t_m^L \le \cdots t_n^L = t_n$ and $t = t_{m - 1}^R \le t_m^R \le \cdots t_n^R = t_n$ define the leftmost (resp. rightmost) geodesic between the points $(m,t)$ and $(n,t_n)$. Then, by~\eqref{eqn:squeeze finite geod} and Lemma~\ref{existence of leftmost and rightmost geodesics}, for all $n \ge N$ and $m \le r \le n$,
\[
\tau_{(m,t),r}^{\gamma-,L} \le t_r^L \le t_{n,r} \le t_r^R \le \tau_{(m,t),r}^{\delta+,R}.
\]
Taking limits as $n \rightarrow \infty$ produces
\[
\tau_{(m,t),r}^{\gamma-,L} \le \liminf_{n \rightarrow \infty} t_{n,r} \le \limsup_{n \rightarrow \infty} t_{n,r} \le \tau_{(m,t),r}^{\delta+,R}.
\]
Taking limits as $\gamma \nearrow \theta$ and $\delta \searrow \theta$ and using Theorem~\ref{existence of semi-infinite geodesics intro version}\ref{itm:convergence of geodesics}\ref{itm:limits in theta} completes the proof of Part~\ref{control of finite geodesics}.

\medskip \noindent \textbf{Part~\ref{all semi-infinite geodesics lie between leftmost and rightmost}}: This is an immediate consequence of Part~\ref{control of finite geodesics}, setting $t_{n,r} = t_r$ for any $\theta$-directed semi-infinite geodesic defined by the sequence $\{t_r\}_{r \ge m - 1}$. 

\medskip \noindent \textbf{Part~\ref{convergence to unique semi-infinite geodesic}:} This is also an immediate consequence of Part~\ref{control of finite geodesics}, because if $\{\tau_r\}_{r = m - 1}^\infty$ is the unique sequence in $\mbf T_{(m,t)}^\theta$, for each $r \ge m$,
$
\tau_r = \tau_{(m,t),r}^{\theta-,L} = \tau_{(m,t),r}^{\theta +,R}.
$
 \end{proof}
 
 We now define the events $\wt \Omega^{(\theta)}$ of Theorem~\ref{thm:non_unique_size}. First, let $\wt \Omega_{\mbf x}^{(\theta)}$ be the events of Lemma~\ref{lemma: uniqueness and directedness of geodesics for fixed parameters}, and for each $\mbf x \in \Z \times \R$, define the full probability events
  \be \label{eqn:omegaxtheta}
 \Omega_{\mbf x}^{(\theta)} = \wt \Omega_{\mbf x}^{(\theta)} \cap \Omega_2.
 \ee 
 By Theorem~\ref{thm:summary of properties of Busemanns for all theta}\ref{independence structure of Busemann functions on levels} and Theorem~\ref{thm:dist of Busemann functions}\ref{BM_drift}, for each $r \in \Z$, $s \mapsto B_r(s) - \h_{r + 1}^\theta(s)$ is a scaled, two-sided Brownian motion with strictly negative drift. By Theorem~\ref{thm:countable non unique maximizers}, for $\theta > 0$ and $r \in \Z$, there exists an event $\operatorname{CM}_r^{(\theta)}$, of probability one, on which the set 
 \[
 \{t \in \R: B_r(s) - \h_{r + 1}^\theta(s) \text{ over }s \in [t,\infty) \text{ has a non-unique maximum at } s =t\}
 \] 
 is countably infinite. Then, for $\theta > 0$, set
 \begin{align} \label{eqn:omega tilde space}
 \wt \Omega^{(\theta)} := \bigcap_{\mbf x \in \Z \times \Q} \Omega_{\mbf x}^{(\theta)} \cap \bigcap_{r \in \Z}\operatorname{CM}_r^{(\theta)}.
 \end{align}
 Because $\Omega_{\mbf x}^{(\theta)} = \wt \Omega_{\mbf x}^{(\theta)} \cap \Omega_2$ by definition and $\Omega_{\mbf x}^{(\theta)} \subseteq \Omega^{(\theta)}$ by  Lemma~\ref{lemma: uniqueness and directedness of geodesics for fixed parameters}, $\wt \Omega^{(\theta)} \subseteq \Omega^{(\theta)} \cap \Omega_2$ where $\Omega^{(\theta)}$ are the events of Theorem~\ref{thm:summary of properties of Busemanns for all theta}. Furthermore, $\Pp(\wt \Omega^{(\theta)}) = 1$.
 
 \begin{proof}[Proof of Theorem~\ref{thm:non_unique_size}]
 In this proof, since $\theta$ is fixed, we drop the $\pm$ distinction for $\theta$ in the superscript. 

 \medskip \noindent \textbf{Part~\ref{decomposition}}: 
By Theorem~\ref{thm:convergence and uniqueness}\ref{all semi-infinite geodesics lie between leftmost and rightmost}, on the event $\wt \Omega^{(\theta)}$, there exist multiple $\theta$-directed geodesics starting from $(m,t)$ if and only if $\tau_{(m,t),r}^{\theta,L} < \tau_{(m,t),r}^{\theta,R}$ for some $r \ge m$. If $\omega \in \wt \Omega^{(\theta)}\subseteq \Omega_2$ and $(m,t)$ is a point whose $\theta$-directed geodesic is not unique, then at most one such semi-infinite geodesic can pass through $(m,t + \ve)$ for some $\ve > 0$. Otherwise, two different geodesics would pass through $(m,q)$ for some $q \in \Q$, giving two $\theta$-directed geodesics starting from $(m,q)$. This cannot hold on the event $\Omega_{\mbf x}^{(\theta)} \supseteq \wt \Omega^{(\theta)}$.  Inductively, to get two different $\theta$-directed geodesics starting from $(m,t)$, there must be some level $r \ge m$ such that all $\theta$-directed geodesics pass through $(r,t)$ and then one geodesic passes through $(r + 1,t)$, and the other passes through $(r,t + \ve)$ for some $\ve > 0$.  Therefore, by Definition~\ref{def:semi-infinite geodesics}, on the event $\wt \Omega^{(\theta)}$,  there exists a point $(m,t) \in \Z \times \R$ whose semi-infinite geodesic in direction $\theta$ is not unique if and only if there exists $r \ge m$ such that, for $m \le k \le r - 1$, $B_k(s) - h_{k +1}^\theta(s)$ over $s \in [t,\infty)$ has a unique maximum at $s = t$, and $B_r(s) - \h_{r + 1}^\theta(s)$ over $s \in [t,\infty)$ has a non-unique maximum at $s = t$. Therefore, $t = \tau_{(m,t),r}^{\theta,L} < \tau_{(m,t),r}^{\theta,R}$ for some $r \ge m$. The countability of the sets then follows from Theorem~\ref{thm:countable non unique maximizers}\ref{countable} because, for each $k$, $B_k - h_{k + 1}^\theta$ is a (scaled) Brownian motion with negative drift (Theorem~\ref{thm:summary of properties of Busemanns for all theta}\ref{independence structure of Busemann functions on levels} and Theorem~\ref{thm:dist of Busemann functions}\ref{BM_drift}).

\medskip \noindent \textbf{Part~\ref{non-discrete or dense}}: By Theorem~\ref{thm:non_unique_size}\ref{decomposition}, $(m,t) \in \NU_1^\theta$ if and only if $B_m(s) - h_{m + 1}^\theta(s)$ has two maximizers over $s \in [t,\infty)$--one at $s = t$ and one at some $s > t$.  The result then follows from Theorem~\ref{thm:countable non unique maximizers}~\ref{non_discrete MN}.
\end{proof}

\subsection{The dual environment} \label{section:dual environment}
{\it Throughout  Sections~\ref{section:dual environment} and~\ref{section:dual geodesics}, $\theta > 0$ is fixed},  and we work on the full probability event $\Omega^{(\theta)}$ of Theorem~\ref{thm:summary of properties of Busemanns for all theta}. By Remark~\ref{rmk:theta+ = theta-}, on this event $\mbf T_{\mbf x}^{\theta} =\mbf T_{\mbf x}^{\theta +} = \mbf T_{\mbf x}^{\theta -}$ for all $\mbf x \in \Z \times \R$. Recall the dual environment $\mbf X^\theta$ of independent Brownian motions from Theorem~\ref{thm:dist of Busemann functions}

\begin{lemma} \label{lemma:Busemann limits for southwest geodesics}
    For each $\theta > 0$, there exists an event $\wt \Omega^{(\theta,\star)}$ of probability one, on which
    \[
    \lim_{n\rightarrow \infty} \bigl[ \,L_{(-n,-t_n),\mbf x}(\mbf X^\theta) - L_{(-n,-t_n),\mbf y}(\mbf X^\theta)\,\bigr]  =    \B^{\theta}(\mbf y,\mbf x) 
    \]
    for any $\mbf x,\mbf y \in \Z \times \R$ and for any sequence $\{t_n\}$ satisfying $\lim_{n \rightarrow \infty}\f{t_n}{n} = \theta$.
    To clarify, $L_{\mbf x,\mbf y}(\mbf X^\theta)$ is the last-passage process on the environment $\mbf X^\theta$ while  $\B^\theta$ is the original Busemann function  of  the environment $\mbf B$ as in Theorem \ref{thm:summary of properties of Busemanns for all theta}.
\end{lemma}
\begin{proof}
Recall from Theorem~\ref{thm:summary of properties of Busemanns for all theta}\ref{general queuing relations Busemanns} that
\[
\h_{m - 1}^{\theta } = D(\h_m^{\theta},B_{m - 1}),  \qquad  X_m^{\theta} = R(\h_m^{\theta },B_{m - 1}), \qquad\text{and} \qquad \vv_m^{\theta }= Q(\h_m^{\theta},B_{m - 1}).
\]
Further, recall~\eqref{eqn:dual weights reverse relations}, which states 
\begin{equation} \label{eqn:reverse}
\h_m^{\theta} = \Da(\h_{m - 1}^{\theta},X_m^{\theta}),\qquad B_{m - 1} = \Ra(\h_{m - 1}^{\theta },X_m^{\theta}), \qquad\text{and}\qquad \vv_m^{\theta} = \Qa(\h_{m - 1}^{\theta },X_m^{\theta }).
\end{equation}
Recall that  $\wt f(t)=-f(-t)$. Apply Lemma~\ref{lemma:time reversal equality of queuing maps} to deduce that 
 \be\label{382} 
\wt \h_{m - 1}^{\theta } = \Da(\wt \h_m^{\theta},\wt B_{m - 1}),\qquad\wt X_m^{\theta } = \Ra(\wt \h_m^{\theta},\wt B_{m - 1}),\qquad\text{and}\qquad - \wt \vv_m^{\theta }= \Qa(\wt \h_m^{\theta },\wt B_{m - 1}).
\ee 
From independence and matching marginals  (Part Theorem~\ref{thm:summary of properties of Busemanns for all theta}\ref{independence structure of Busemann functions on levels} and Theorem~\ref{thm:dist of Busemann functions}, Parts~\ref{BM_drift} and~\ref{mutual independence of the X_m}) follows 
  \be \label{eqn:equal_dist}   (\,\wt \h_k^{\theta},\wt B_{k - 1}, \wt B_{k - 2}, \wt B_{k - 3},\dotsc)  \deq  (\h_{m - 1}^\theta, X_m^\theta, X_{m+1}^\theta, X_{m+2}^\theta,\dotsc)    \qquad 
   \forall k,m\in\Z.   \ee
 Thus, iterating mappings \eqref{382} backward in the index $m$  and   mappings \eqref{eqn:reverse} forward in the index $m$ gives this equality in distribution: 
\[
\bigl\{\wt \h_{-(m - 1)}^{\theta}, -\wt \vv_{-(m - 1)}^{\theta}, \wt X_{-(m - 1)}^{\theta } ,\wt B_{-m} \bigr\}_{m \in \Z} \deq \bigl\{\h_{m - 1}^\theta,\vv_{m}^\theta,B_{m - 1}, X_m^\theta\bigr\}_{m \in \Z}
\]
Apply Theorem~\ref{thm:summary of properties of Busemanns for all theta}\ref{busemann functions agree for fixed theta}  to the environment $\mbf{\wt X_{-}^\theta} = \{\wt X_{-m}^\theta\}_{m \in \Z}$  in the process on the left above, to deduce that there exists a full probability event $\wt \Omega^{(\theta,\star)}$ on which,
 for any $t \in \R$ and $m \in \Z$, 
\begin{align*}
&\lim_{n \rightarrow \infty} L_{(-m,-t),(n,t_n)}(\mbf{\wt X_{-}^\theta})- L_{(-m,0),(n,t_n)}(\mbf{\wt X_{-}^\theta}) = -\wt \h_{-(-m)}^\theta(-t) = \h_m^\theta(t) \qquad\text{and}\\[1em]
&\lim_{n \rightarrow \infty}L_{(-m  ,-t),(n,t_n)}(\mbf{\wt X_{-}^\theta}) -  L_{(-m + 1 ,-t),(n,t_n)}(\mbf{\wt X_{-}^\theta})  = -\wt \vv_{-(-m)}^\theta(-t)  = \vv_{m}^{\theta}(t).
\end{align*}

Adding together horizontal and vertical steps for the general case, the proof is complete by noting that, for any $m,k \in \Z$ and $s,t \in \R$, 
\[
\lim_{n \rightarrow \infty} L_{(-m,-s),(n,t_n)}(\mbf{\wt X_{-}^\theta}) - L_{(-k,-t),(n,t_n)}(\mbf{\wt X_{-}^\theta}) = \lim_{n \rightarrow \infty} L_{(-n,-t_n),(m,s)}(\mbf X^\theta) -L_{(-n,-t_n),(k,t)}(\mbf X^\theta)  
\]
because
\begin{align*}
    &L_{(-m,-s),(n,t_n)}(\mbf{\wt X_{-}^\theta}) = \sup\Bigl\{\sum_{r = -m}^n \wt X_{-r}^\theta(s_{r - 1},s_r): -s = s_{-m - 1} \le \cdots \le s_n = t_n   \Bigr\} \\[1em]
    &= \sup \Bigl\{ \sum_{r = -m}^n  X_{-r}^\theta(-s_r,-s_{r - 1}): -s = s_{-m - 1} \le \cdots \le s_n = t_n    \Bigr\} \\[1em]
    &= \sup \Bigl\{\sum_{r = -m}^n X_{-r}^\theta(-s_r,-s_{r - 1}): -t_n = -s_n \le -s_{n - 1} \le \cdots \le -s_{m - 1} = s   \Bigr\} \\[1em]
    &= \sup \Bigl\{\sum_{r = -n}^m X_{r}^\theta(\wt s_{r - 1},\wt s_r): -t_n = \wt s_{-n - 1} \le \wt s_{-n} \le \cdots \le \wt s_{m} = s   \Bigr\}
    = L_{(-n,-t_n),(m,s)}(\mbf X^\theta). 
\end{align*}
To get the second-to-last line above, simply set $\wt s_k = -s_{-k - 1}$.
\end{proof}

\subsection{Dual geodesics} \label{section:dual geodesics}
Recall the definition of the sets $\mbf T_{\mbf x}^{\theta \star}$ from Section~\ref{sec:intro_dual_geod}. Analogous results as for the original northwest semi-infinite geodesics hold, as demonstrated by the following theorem. 
\begin{theorem} \label{existence of backwards semi-infinite geodesics}
 Fix $\theta > 0$. Then, for every $\mbf x \in \Z \times \R$, every semi-infinite path in $\mbf T_{\mbf x}^{\theta,\star}$ is a semi-infinite geodesic for Brownian last-passage percolation with environment $\mbf X^{\theta}$.   Specifically, the following hold. 
 \begin{enumerate} [label=\rm(\roman{*}), ref=\rm(\roman{*})]  \itemsep=3pt 
     \item \label{southwest energy of path along semi-infinte geodesic} On the full probability event $\Omega^{(\theta)}$, let $\mbf x \in \Z \times \R$, and let $\Gamma^\star$ be any semi-infinite path in $\mbf T_{\mbf x}^{\theta,\star}$. Then, for any $\mbf y \le \mbf z \in \Z \times \R$ with $\mbf y^\star,\mbf z^\star$ lying along the semi-infinite path, the energy of the portion of that path between $\mbf y^\star$ and $\mbf z^\star$, in the environment $\mbf X^\theta$, is 
     \[
      L_{\mbf y,\mbf z}(\mbf X^{\theta}) =   \B^{\theta}(\mbf y,\mbf z), 
     \]
     and this energy is maximal among all paths between $\mbf y$ and $\mbf z$ in the environment $\mbf X^\theta$. To be clear, $\B^\theta$ is the original Busemann function for the environment $\mbf B$.
     \item On $\Omega^{(\theta)}$, for all $(m,t) \in \Z \times \R$ and $r \le m$, 
     \begin{align*}
     &\tau_{(m,t),r - 1}^{\theta,R \star} = \sup\bigl\{u \le \tau^{\theta,R \star}_{(m,t),r }: \vv_{r }^{\theta}(u) = 0\bigr\} \text{ and } \\[1em]
     &\tau^{\theta,L \star}_{(m,t),r - 1} = \inf\bigl\{u \le \tau^{\theta,L \star}_{(m,t),r}: \h_r^{\theta}(u,\tau^{\theta,L \star}_{(m,t),r}) = X_r^{\theta}(u,\tau^{\theta,L \star}_{(m,t),r})\bigr\}
     \end{align*}
    More specifically, if $u \le \tau^{\theta,L \star}_{(m,t),r }$, then $\h_r^{\theta }(u,\tau^{\theta,L \star}_{(m,t),r}) = X_r^{\theta}(u,\tau^{\theta,L \star}_{(m,t),r })$ if and only if $u \ge \tau^{\theta,L \star}_{(m,t),r - 1}$.  \label{southwest identity for rightmost and leftmost geodesics}
     \item On $\Omega^{(\theta)}$, if, for some $\mbf z \in \Z \times \R$, $\mbf x^\star \ge \mbf y^\star$ lie along the leftmost semi-infinite geodesic in $\mbf T_{\mbf z}^{\theta,\star} $, then the portion of the path between $\mbf x^\star$ and $\mbf y^\star$, shifted back up by $\f{1}{2}$ to lie on integer levels, is the leftmost geodesic between $\mbf x$ and $\mbf y$ in the environment $\mbf X^\theta$. The analogous statement  holds for the rightmost geodesics. \label{southwest Leftandrightmost}
     \item \label{dual geodesics same dist as NW geodesics} The following distributional equality holds.
    \begin{align*}
    &\bigl\{\bigl(\tau_{(m,t),r}^{\theta,R},\tau_{(m,t),r}^{\theta,L}\bigr): (m,t) \in \Z, r \ge m \bigr\} \\[1em]
    &\qquad\qquad\qquad\qquad \deq\bigl\{\bigl(-\tau_{(-m,-t),-(r + 1)}^{\theta,L \star},-\tau_{(-m,-t),-(r + 1)}^{\theta,R \star}\bigr): (m,t) \in \Z, r \ge m     \bigr\} 
    \end{align*}
     \item \label{southwest geodesics limits} There exists an event of full probability, $\Omega^{(\theta,\star)} \subseteq \Omega^{(\theta)}$, on which, for every $(m,t) \in \Z \times \R$ and every sequence $\{\tau^\star_{r}\}_{r \le m} \in \mbf T_{(m,t)}^{\theta, \star}$,
    \[
    \lim_{n \rightarrow \infty} \f{\tau_{-n}^\star}{-n} = \theta. 
    \]
 \end{enumerate}
\end{theorem}
\begin{proof}
\textbf{Parts~\ref{energy of path along semi-infinte geodesic}--\ref{southwest Leftandrightmost}}: 
On $\Omega^{(\theta)}$, let $(m,t) \in \Z \times \R$, and  $\{\tau_r^\star\}_{r \le m} \in \mbf T_{(m,t)}^{\theta,\star}$. By~\eqref{eqn:reverse} and the definitions~\eqref{reverse definition of Q}--\eqref{reverse definition of R}, for all $u,v \in \R$ and $r \in \Z$, 
\begin{align}
\vv_{r}^{\theta}(u) &= \sup_{-\infty < s \le u} \{X_r^\theta(s,u) - \h_{r - 1}^\theta(s,u) \}, \label{eqn:v_backwards}\\[1em]
\h_{r }^{\theta}(u,v) &= \h_{r - 1}^{\theta}(u,v) + \sup_{-\infty < s \le v}\{X_r^{\theta}(s,v) - \h_{r - 1}^{\theta}(s,v)\} - \sup_{-\infty < s \le u}\{X_r^{\theta}(s,u) - \h_{r - 1}^{\theta}(s,u)\} \nonumber \\[1em]
&= X_{r}^{\theta}(u,v) + \sup_{-\infty < s \le v}\{\h_{r - 1 }^{\theta}(s) -X_{r}^{\theta}(s)\} - \sup_{-\infty < s \le u}\{\h_{r - 1 }^{\theta}(s) -X_{r}^{\theta }(s)\}, \text{ and} \label{eqn:equality of busemann to dual weight}\\[1em]
B_{r - 1}(u,v) &=X_r^{\theta}(u,v) + \sup_{-\infty < s \le u}\{X_r^{\theta}(s,u) - \h_{r - 1}^{\theta}(s,u)\} - \sup_{-\infty < s \le v}\{X_r^{\theta}(s,v) - \h_{r - 1}^{\theta}(s,v)\} \nonumber \\[1em]
&= \h_{r - 1}^{\theta}(u,v) + \sup_{-\infty < s \le u}\{ \h_{r - 1}^{\theta}(s)-X_{r}^{\theta}(s) \} - \sup_{-\infty < s \le v}\{ \h_{r - 1}^{\theta}(s)-X_{r}^{\theta}(s) \} . \label{eqn:equality of busemann to BM dual}
\end{align}
 As $\tau_{r -1}^\star$ is a maximizer of $h_{r - 1}(s) - X_r^\theta(s)$ over $s \in (\infty,\tau_r^\star]$, from~\eqref{eqn:v_backwards}, it follows that $\vv_r^\theta(\tau_{r - 1}^\star) = 0$ for each $r \le m$. By~\eqref{eqn:equality of busemann to dual weight}, $\h_r^{\theta}(u,v) = X_r^\theta(u,v)$ for $u,v \in [\tau_{r - 1}^\star,\tau_r^\star]$. In general, $\vv_r^\theta(u) \ge 0$ and $\h_r^{\theta}(u,v) \ge X_r^{\theta}(u,v)$ for $u \le v$. Then, Parts~\ref{southwest energy of path along semi-infinte geodesic}-\ref{southwest Leftandrightmost} follow just as for the analogous statements in the proofs of Theorem~\ref{existence of semi-infinite geodesics intro version}, Parts~\ref{energy of path along semi-infinte geodesic}-\ref{Leftandrightmost} and Lemma~\ref{lemma:equality of busemann to weights of BLPP}.

\medskip \noindent \textbf{Parts~\ref{dual geodesics same dist as NW geodesics}--\ref{southwest geodesics limits}}:
By~\eqref{eqn:equal_dist},  $\{\wt \h_{r - 1}^\theta, \wt X_{r}^\theta\}_{r \in \Z}$ has the same distribution as $\{\h_{-(r - 1)}^\theta,B_{-r}\}_{r \in \Z}$. Furthermore, the leftmost (rightmost) maximizers of $\h_{r - 1}^\theta(s) - X_r^\theta(s)$ for $s \in (-\infty,t]$ are the negative of the rightmost (resp. leftmost) maximizers of $\wt X_{r}^\theta(s) - \wt \h_{r - 1}^\theta(s)$ for $s \in [-t,\infty)$, establishing Part~\ref{dual geodesics same dist as NW geodesics}. Part~\ref{southwest geodesics limits} then follows from Theorem~\ref{existence of semi-infinite geodesics intro version}\ref{general limits for semi-infinite geodesics}.
\end{proof}

\begin{theorem} \label{strong crossing theorem for geodesics and dual geodesics}
On the event $\Omega^{(\theta)}$, the following hold for all $s \le t \in \R$ and $m \in \Z$. 
 \begin{enumerate} [label=\rm(\roman{*}), ref=\rm(\roman{*})]  \itemsep=3pt 
 \item If $\tau_{(m,s),m}^{\theta,R} < t$, then also $\tau_{(m,s),m}^{\theta,R} < \tau_{(m + 1,t),m}^{\theta,L \star}$. See Figure~\ref{fig:dual left horizontal misses right vertical} for clarity.  \label{itm:general dual left horizontal does not cross right vertical} 
    
    \item If $t \le \tau_{(m,s),m}^{\theta, R}$, then $\tau_{(m + 1,t),m}^{\theta ,L^\star} \le s$. See Figure~\ref{fig:weak dual left vertical misses right vertical} for clarity. \label{itm:weak dual left vertical does not cross right horizontal}
    
    \item If $ \tau_{(m,s),m}^{\theta ,L} \le t$, then also $\tau_{(m,s),m}^{\theta ,L} \le \tau_{(m + 1,t),m}^{\theta ,R \star}$. See Figure~\ref{fig:weak dual right horizontal misses left vertical} for clarity. \label{itm:weak  dual right horizontal does not cross left vertical}
    \item If $t < \tau_{(m,s),m}^{\theta , L}$, then $\tau_{(m,+1,t),m}^{\theta ,R \star} < s$. See Figure~\ref{fig:dual right vertical misses left vertical} for clarity. \label{itm:general dual right vertical does not cross left horizontal}
\end{enumerate}
\end{theorem}
\begin{remark} \label{rmk:geod to left}
 All the statements of Theorem~\ref{strong crossing theorem for geodesics and dual geodesics} only mention the first jump time for northeast geodesics and the first time of descent for southwest geodesics. However, for any $\mbf x \in \Z \times \R$ and any point $\mbf y$ along the rightmost (resp. leftmost) semi-infinite geodesic in $\mbf T_{\mbf x}^\theta$, the rightmost (resp. leftmost) semi-infinite geodesic in $\mbf T_{\mbf y}^\theta$ agrees with the remainder of the original semi-infinite geodesic started from $\mbf x$. Thus, the results of Theorem~\ref{strong crossing theorem for geodesics and dual geodesics} can be extended by induction. For example, Part~\ref{itm:general dual left horizontal does not cross right vertical} implies that if a leftmost dual southwest geodesic starts strictly to the right and below a rightmost northeast geodesic, it remains to the right and below the northeast geodesic. See Figure~\ref{fig:dual left horizontal misses right vertical}.
\end{remark}
\begin{remark}
Theorem~\ref{strong crossing theorem for geodesics and dual geodesics} is the analogue of Lemma 4.4 in~\cite{Timo_Coalescence}, which states a result for exponential last-passage percolation. In words, Theorem~\ref{strong crossing theorem for geodesics and dual geodesics} says that northeast Busemann geodesics do not cross dual southwest geodesics. In this sense, the theorem gives an analogue to Pimentel's dual tree~\cite{pimentel2016}. 
In the setting of exponential LPP, in Section 5 of~\cite{Timo_Coalescence}, it is shown that semi-infinite geodesics in the original environment are competition interfaces for geodesics in the dual environment with boundary conditions given by the Busemann process. This gives some intuition on why the northeast geodesics do not cross the southwest dual geodesics. Due to the general non-uniqueness of geodesics, the construction of competition interfaces is somewhat delicate for BLPP and will be studied in future work. 
\end{remark}
\begin{proof}
\noindent \textbf{Part~\ref{itm:general dual left horizontal does not cross right vertical}}: 
By Lemma~\ref{lemma:equality of busemann to weights of BLPP}, if $u \ge s$, then
\begin{equation} \label{geodesic crossing: first biconditional for busemann increments to equal B}
B_m(s,u) =  \h_m^{\theta}(s,u) \text{ if and only if } u \le \tau_{(m,s),m}^{\theta,R}. 
\end{equation}
By Theorem~\ref{existence of backwards semi-infinite geodesics}\ref{southwest identity for rightmost and leftmost geodesics}, if $u \le t$, then
\begin{equation} \label{geodesic crossing: biconditional for X increments to equal h increments}
X_{m + 1}^{\theta}(u,t) = \h_{m + 1}^{\theta}(u, t) \text{ if and only if } u \ge \tau_{(m + 1,t),m }^{\theta,L \star}.
\end{equation}
By~\eqref{eqn:equality of busemann to dual weight} and~\eqref{eqn:equality of busemann to BM dual},
\begin{align*} 
&X_{m +1}^{\theta }(u,v) = \h_{m + 1}^{\theta }(u,v) \iff B_m(u,v) = \h_m^{\theta }(u,v) \\[1em]
\iff &\sup_{-\infty \le w \le u}\{\h_m^{\theta }(w) - X_{m + 1}^{\theta }(w)\} = \sup_{-\infty \le w \le v}\{\h_m^{\theta }(w) - X_{m + 1}^{\theta }(w)\}.
\end{align*}
This along with~\eqref{geodesic crossing: biconditional for X increments to equal h increments} implies that if $u \le t$,
\begin{equation} \label{geodesic crossing: second biconditional for busemann increments to equal B}
B_m^{\theta }(u,t) = \h_m^{\theta }(u,t) \text{ if and only if } u \ge \tau_{(m + 1,t),m }^{\theta ,L \star}.
\end{equation}

By assumption, $\tau_{(m,s),m}^{\theta ,R} < t$. Equation~\eqref{geodesic crossing: first biconditional for busemann increments to equal B} and the monotonicity of Theorem~\ref{thm:summary of properties of Busemanns for all theta}\ref{general monotonicity Busemanns} imply that
\begin{equation} \label{geodesic crossing: busemann increment greater than B}
B_m(s,t) < \h_m^{\theta }(s,t).
\end{equation}

Assume, by way of contradiction, that $\tau_{(m + 1,t),m}^{\theta ,L^\star} \le \tau_{(m,s),m}^{\theta ,R}$. Then, there exists $u \le \tau_{(m,s),m}^{\theta ,R}$ with $ u \ge s \vee \tau_{(m,t),m}^{\theta ,L^\star}$. Then, for such $u$, by~\eqref{geodesic crossing: first biconditional for busemann increments to equal B} and~\eqref{geodesic crossing: second biconditional for busemann increments to equal B},
\[
B_r(s,u) =  \h_r^{\theta }(s,u)\qquad\text{and}\qquad B_r^{\theta }(u,t) = \h_r^{\theta }(u, t). 
\]
Adding these two equations gives us a contradiction to~\eqref{geodesic crossing: busemann increment greater than B}.

\medskip \noindent \textbf{Part~\ref{itm:weak dual left vertical does not cross right horizontal}}: Let $t \le \tau_{(m,s),m}^{\theta ,R}$, and assume by way of contradiction, that $\tau_{(m + 1,t),m}^{\theta , L\star} > s$. This also implies that $s < t$ since $t \ge \tau_{(m + 1,t),m}^{\theta , L\star}$. By Equation~\eqref{geodesic crossing: second biconditional for busemann increments to equal B}, $B_m(s,t) < \h_m^{\theta }(s,t)$. However, since $s < t \le \tau_{(m,s),m}^{\theta ,R}$, Equation~\eqref{geodesic crossing: first biconditional for busemann increments to equal B}, implies that $B_m(s,t) = \h_m^{\theta }(s,t)$, giving the desired contradiction. 

\medskip \noindent \textbf{Part~\ref{itm:weak  dual right horizontal does not cross left vertical}:} Assume that $\tau_{(m,s),m}^{\theta ,L} \le t$. By Theorem~\ref{existence of backwards semi-infinite geodesics}\ref{southwest identity for rightmost and leftmost geodesics},
\[
\tau_{(m + 1,t),m}^{\theta ,R \star} = \sup\{u \le t: \vv_{m + 1}^{\theta }(u) = 0\}. 
\]
By Lemma~\ref{lemma:equality of busemann to weights of BLPP}, 
$
\vv_{m + 1}^{\theta }(\tau_{(m,s),m}^{\theta ,L}) = 0$, so the desired conclusion follows.

\medskip \noindent \textbf{Part~\ref{itm:general dual right vertical does not cross left horizontal}}: Assume that $t < \tau_{(m,s),m}^{\theta ,L}$. By Lemma~\ref{lemma:equality of busemann to weights of BLPP},
\[
\tau_{(m,s),m}^{\theta ,L} = \inf\{u \ge s: \vv_{m + 1}^{\theta }(u) = 0\}.
\]
By assumption, $\tau_{(m + 1,t),m}^{\theta ,R \star} \le t < \tau_{(m,s),m}^{\theta ,L}$. By Theorem~\ref{existence of backwards semi-infinite geodesics}\ref{southwest identity for rightmost and leftmost geodesics},
$\vv_{m + 1}^{\theta }(\tau_{(m + 1,t),m}^{\theta ,R \star}) = 0$, so $\tau_{(m + 1,t),m}^{\theta ,R \star} < s$.
\end{proof}

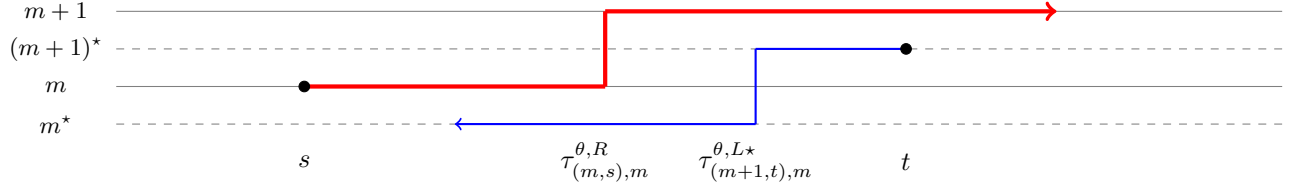
\begin{figure}
\begin{tikzpicture}
\draw[gray,thin] (-0.5,0) -- (15,0);
\draw[gray,thin] (-0.5,1)--(15,1);
\draw[gray,thin,dashed] (-0.5,-0.5)--(15,-0.5);
\draw[gray,thin,dashed] (-0.5,0.5)--(15,0.5);
\draw[red, ultra thick] (2,0)--(6,0);
\draw[red, ultra thick] (6,0)--(6,1);
\draw[red, ultra thick,->] (6,1)--(12,1);
\draw[blue,thick] (10,0.5)--(8,0.5);
\draw[blue,thick] (8,0.5)--(8,-0.5);
\draw[blue,thick,->] (8,-0.5)--(4,-0.5);
\filldraw[black] (2,0) circle (2pt);
\filldraw[black] (10,0.5) circle (2pt);
\node at (-1.3,-0.5) {\small $m^\star$};
\node at (-1.3,0) {\small $m$};
\node at (-1.3,0.5) {\small $(m + 1)^\star$};
\node at (-1.3,1) {\small $m + 1$};
\node at (2,-1) {$s$};
\node at (6,-1) {$\tau_{(m,s),m}^{\theta ,R}$};
\node at (8,-1) {$\tau_{(m + 1,t),m}^{\theta ,L \star}$};
\node at (10,-1) {$t$};
\end{tikzpicture}
\caption{\small A rightmost northeast geodesic (red/thick) lies strictly above and to the left of a leftmost dual southwest geodesic (blue/thin).}
\label{fig:dual left horizontal misses right vertical}
\end{figure}

\begin{figure}
\begin{tikzpicture}
\draw[gray,thin] (-0.5,0) -- (15,0);
\draw[gray,thin] (-0.5,1)--(15,1);
\draw[gray,thin,dashed] (-0.5,-0.5)--(15,-0.5);
\draw[gray,thin,dashed] (-0.5,0.5)--(15,0.5);
\draw[red, ultra thick] (5,0)--(9,0);
\draw[red, ultra thick] (9,0)--(9,1);
\draw[red, ultra thick,->] (9,1)--(12,1);
\draw[blue,thick] (5,0.5)--(7,0.5);
\draw[blue,thick] (5,0.5)--(5,-0.5);
\draw[blue,thick,->] (5,-0.5)--(1,-0.5);
\filldraw[black] (5,0) circle (2pt);
\filldraw[black] (7,0.5) circle (2pt);
\node at (-1.3,-0.5) {\small $m^\star$};
\node at (-1.3,0) {\small $m$};
\node at (-1.3,0.5) {\small $(m + 1)^\star$};
\node at (-1.3,1) {\small $m + 1$};
\node at (4.5,-1) {$\tau_{(m + 1,t),m}^{\theta ,L \star} \le s$};
\node at (9,-1) {$\tau_{(m,s),m}^{\theta ,R}$};
\node at (7,-1) {$t$};
\end{tikzpicture}
\caption{\small A rightmost northeast geodesic (red/thick) lies weakly to the right and below a leftmost dual southwest geodesic (blue/thin).}
\label{fig:weak dual left vertical misses right vertical}
\end{figure}
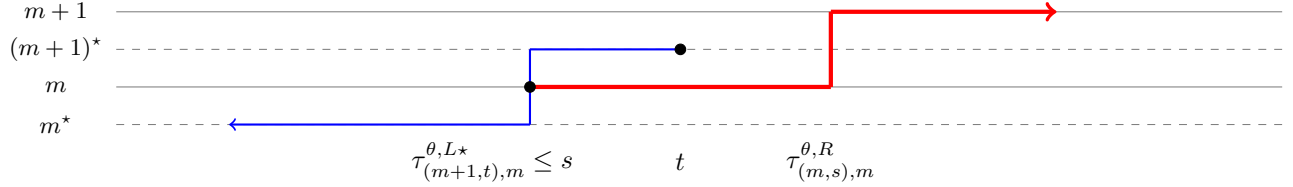

\begin{figure}
\begin{tikzpicture}
\draw[gray,thin] (-0.5,0) -- (15,0);
\draw[gray,thin] (-0.5,1)--(15,1);
\draw[gray,thin,dashed] (-0.5,-0.5)--(15,-0.5);
\draw[gray,thin,dashed] (-0.5,0.5)--(15,0.5);
\draw[red, ultra thick] (2,0)--(6,0);
\draw[red, ultra thick] (6,0)--(6,1);
\draw[red, ultra thick,->] (6,1)--(12,1);
\draw[blue,thick,->] (9,0.5)--(6,0.5)--(6,-0.5)--(3,-0.5);
\filldraw[black] (2,0) circle (2pt);
\filldraw[black] (9,0.5) circle (2pt);
\node at (-1.3,-0.5) {\small $m^\star$};
\node at (-1.3,0) {\small $m$};
\node at (-1.3,0.5) {\small $(m + 1)^\star$};
\node at (-1.3,1) {\small $m + 1$};
\node at (2,-1) {\small $s$};
\node at (6,-1) {\small $\tau_{(m,s),m}^{\theta ,L} \le \tau_{(m + 1,t),m}^{\theta ,R \star}$};
\node at (9,-1) {$t$};
\end{tikzpicture}
\caption{\small A leftmost northeast geodesic (red/thick) lies weakly above and to the left of a rightmost dual southwest geodesic (blue/thin).}
\label{fig:weak dual right horizontal misses left vertical}
\end{figure}
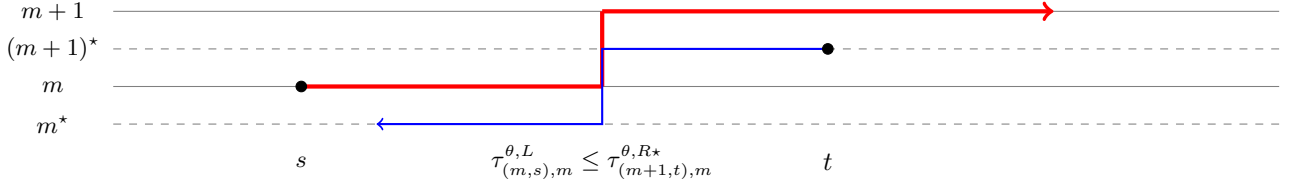
\begin{figure}
\begin{tikzpicture}
\draw[gray,thin] (-0.5,0) -- (15,0);
\draw[gray,thin] (-0.5,1)--(15,1);
\draw[gray,thin,dashed] (-0.5,-0.5)--(15,-0.5);
\draw[gray,thin,dashed] (-0.5,0.5)--(15,0.5);
\draw[red, ultra thick] (5,0)--(9,0);
\draw[red, ultra thick] (9,0)--(9,1);
\draw[red, ultra thick,->] (9,1)--(12,1);
\draw[blue,thick] (3,0.5)--(7,0.5);
\draw[blue,thick] (3,0.5)--(3,-0.5);
\draw[blue,thick,->] (3,-0.5)--(1,-0.5);
\filldraw[black] (5,0) circle (2pt);
\filldraw[black] (7,0.5) circle (2pt);
\node at (-1.3,-0.5) {$m^\star$};
\node at (-1.3,0) {$m$};
\node at (-1.3,0.5) {$(m + 1)^\star$};
\node at (-1.3,1) {$m + 1$};
\node at (5,-1) {$s$};
\node at (9,-1) {$\tau_{(m,s),m}^{\theta ,L}$};
\node at (3,-1) {$\tau_{(m + 1,t),m}^{\theta ,R \star}$};
\node at (7,-1) {$t$};
\end{tikzpicture}
\caption{\small A leftmost northeast geodesic (red/thick) lies strictly to the right and below a rightmost dual southwest geodesic (blue/thin).}
\label{fig:dual right vertical misses left vertical}
\end{figure}
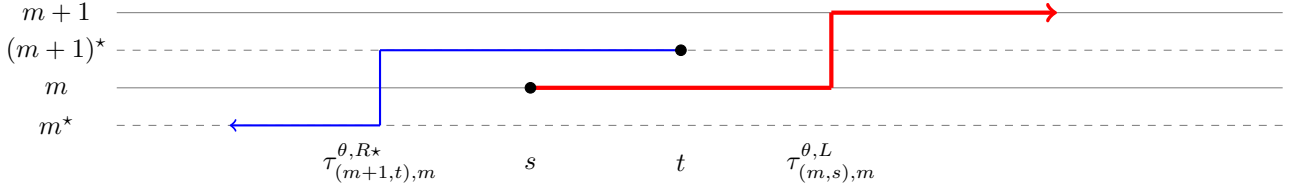

We use these non-intersection properties to prove Theorem~\ref{thm:bi-infinite geodesic between disjoint geodesics}. 

\begin{proof}[Proof of Theorem~\ref{thm:bi-infinite geodesic between disjoint geodesics}]
We follow a similar compactness argument as in the proof of Lemma 4.6 in~\cite{Timo_Coalescence}, appropriately modified for the semi-discrete setting of BLPP. The full probability event of this theorem is $\Omega^{(\theta)} \cap \Omega^{(\theta,\star)} \cap \Omega_2$, where $\Omega^{(\theta,\star)}$ is the event of Theorem~\ref{existence of backwards semi-infinite geodesics}\ref{southwest geodesics limits}. On this event, let $\mbf x = (m,t), \mbf y = (k,t)$. Assume that $\Gamma_1$ and $\Gamma_2$ are disjoint, where $\Gamma_1$ denotes the path defined by the jump times $t = \tau_{(k,t),k - 1}^{\theta ,R} \le  \tau_{(k,t),k}^{\theta ,R} \le \cdots$, and $\Gamma_2$ denotes the path defined by the jump times $s = \tau_{(m,s),m - 1}^{\theta ,R} \le \tau_{(m,s),m}^{\theta ,R} \le \cdots$. By Theorem~\ref{existence of semi-infinite geodesics intro version}\ref{general limits for semi-infinite geodesics}, these paths satisfy
\be\label{678} 
\lim_{n \rightarrow \infty} \f{\tau_{(m,s),n}^{\theta ,R}}{n} = \lim_{n \rightarrow \infty}\f{\tau_{(k,t),n}^{\theta ,R}}{n} = \theta,
\ee
so both the space and time coordinates of these up-right paths go to $\infty$. Assume, without loss of generality, that $\Gamma_2$ lies above and to the left of $\Gamma_1$. Then, by picking $t_0$ large enough, there exists $k_0^1$ and $k_0^2$ such that $(k_0^1,t_0) \in \Gamma_1$ and $(k_0^2,t_0) \in \Gamma_2$. For each $t_0$, there are only finitely many choices of $k_0^1$ and $k_0^2$, so assume that $k_0^1$ is the largest such choice and $k_0^2$ is the smallest such choice. Then, $k_0^1 < k_0^2$ by the assumption that $\Gamma_2$ lies above and to the left of $\Gamma_1$. See Figure~\ref{fig:discrete time for constructing bi-infinte path} for clarity.  For $i = 1,2,\ldots$, define $t_i = t_0 + i$. For each time $i$, similarly define $k_i^1$ and $k_i^2$ so that $k_i^1 < k_i^2$ and $(k_i^j,t_i)$ lies on $\Gamma_j$ for $j = 1,2$. Recall the notation $m^\star = m - \f{1}{2}$. Then, for $i = 0,1,2,\ldots$, there exists $k_i \in \Z$ with $k_i^1 < k_i^\star < k_i^2$. 
This gives us an infinite sequence $(k_i^\star,t_i)_{i \ge 0}$ such that, for each $i$, the point $\bigl(k_i^\star,t_i\bigr)$, lies between the paths $\Gamma_1$ and $\Gamma_2$ (see Figure~\ref{fig:discrete time for constructing bi-infinte path}). 
Starting from each of these points $\bigl(k_i^\star,t_i\bigr)$, let $\Gamma_i^\star$ be the leftmost dual geodesic path in $\mbf T_{(k_i,t_i)}^{\theta, \star }$. Each $\Gamma_i^\star$ is an infinite, down-left path satisfying the limit condition of Theorem~\ref{existence of backwards semi-infinite geodesics}\ref{southwest geodesics limits}, so for each $j < i$, each of the paths $\Gamma_i^\star$  intersects the vertical line $t =t_j$. Let $k_{i,i}^\star = k_i^\star$, and for $j < i$, let $k_{i,j}$ be the maximal integer such that $\bigl(k_{i,j}^\star,t_j\bigr)$ lies on the path $\Gamma_i^\star$. Since $\bigl(k_i^\star,t_i\bigr)$ lies between $\Gamma_1$ and $\Gamma_2$, and $\Gamma_1,\Gamma_2$ are rightmost semi-infinite geodesics constructed from the Busemann functions, Theorem~\ref{strong crossing theorem for geodesics and dual geodesics}\ref{itm:general dual left horizontal does not cross right vertical}--\ref{itm:weak dual left vertical does not cross right horizontal} implies that for $i \ge 0$, $\Gamma_i^\star$ lies strictly below and to the right of $\Gamma_2$, and weakly above and to the left
of $\Gamma_1$ (see also Remark~\ref{rmk:geod to left}).  

Since $k_j^1 < k_{i,j}^\star < k_j^2$ for $0 \le j \le i$, there are only finitely many values of $k_{i,j}$ for each value of $j$. Then, there exists a subsequence $k_{i_r}$ such that for some $N_0 \in \Z$ with $k_0^1 < N_0^\star < k_0^2$ and all $\ell$, $k_{i_\ell,0} = N_0$. Take a further subsequence $k_{i_{r_\ell}}$ such that for some $k_1^1 < N_1^\star < k_1^2$ and every element of this subsequence, $k_{i_{r_\ell},1} = N_1$. Continuing in this way, there exists a sequence $(N_i,t_i)$ of elements of $\Z \times \R$, that $(N_i^\star,t_i)$ lies between the paths $\Gamma_1$ and $\Gamma_2$, and such that the leftmost southwest dual semi-infinite geodesic starting from $(N_i^\star,t_i)$ passes through $(N_j^\star,t_j)$ for $0 \le j \le i$. Since we chose the leftmost paths in $\mbf T_{(N_i,t_i)}^{\theta,\star }$ at each step, the paths are consistent at all space-time points. Hence, this construction gives a bi-infinite path such that, for any point along the path, the part of the path to the southwest of that point is a semi-infinite southwest dual geodesic.  To get the last part of the theorem about asymptotic direction of the paths, we already showed the direction to the southwest. The direction to the northeast follows by~\eqref{678} because the path lies between $\Gamma_1$ and $\Gamma_2$.

The proof for the disjoint left geodesics is analogous, replacing the use of leftmost southwest dual geodesics with rightmost southwest geodesics, and now, the bi-infinite path lies weakly below and to the right of $\Gamma_2$ and strictly above and to the left of $\Gamma_1$.
\end{proof}

\begin{figure}[ht!]
\begin{tikzpicture}
\draw[gray,thin] (-0.5,0) -- (15,0);
\draw[gray,thin] (-0.5,1)--(15,1);
\draw[gray,thin] (-0.5,2)--(15,2);
\draw[gray,thin] (-0.5,3)--(15,3);
\draw[gray,thin] (-0.5,4)--(15,4);
\draw[gray,thin] (-0.5,4)--(15,4);
\draw[gray,thin] (-0.5,5)--(15,5);
\draw[gray,thin,dashed] (-0.5,-0.5)--(15,-0.5);
\draw[gray,thin,dashed] (-0.5,0.5)--(15,0.5);
\draw[gray,thin,dashed] (-0.5,1.5)--(15,1.5);
\draw[gray,thin,dashed] (-0.5,2.5)--(15,2.5);
\draw[gray,thin,dashed] (-0.5,3.5)--(15,3.5);
\draw[gray,thin,dashed] (-0.5,4.5)--(15,4.5);
\draw[gray,thin,dashed] (1,-0.5)--(1,5);
\draw[gray,thin,dashed] (4,-0.5)--(4,5);
\draw[gray,thin,dashed] (7,-0.5)--(7,5);
\draw[gray,thin,dashed] (10,-0.5)--(10,5);
\draw[gray,thin,dashed] (13,-0.5)--(13,5);
\node at (1,-1) {$t_0$};
\node at (4,-1) {$t_1$};
\node at (7,-1) {$t_2$};
\node at (10,-1) {$t_3$};
\node at (13,-1) {$t_4$};
\filldraw[black] (1,0) circle (2pt) node[anchor = north] {\small $(k_0^1,t_0)$};
\filldraw[black] (1,2) circle (2pt) node[anchor = south] {\small $(k_0^2,t_0)$};
\draw[red, ultra thick,->] (1,0)--(3,0)--(3,1)--(6,1)--(6.5,1)--(6.5,2)--(9,2)--(11,2)--(11,3)--(13,3)--(15,3)--(15,4);
\draw[red, ultra thick,->] (-0.5,1)--(0.5,1)--(0.5,2)--(2.5,2)--(2.5,3)--(5,3)--(5,4)--(12,4)--(12,5)--(15,5);
\draw[blue,thick,->] (10,2.5)--(6.5,2.5)--(6.5,1.5)--(4.5,1.5)--(4.5,1.5)--(1.5,1.5)--(1.5,0.5)--(0.5,0.5)--(0.5,-0.5)--(-0.5,-0.5);
\filldraw[black] (10,2.5) circle (2pt) node[anchor = south] {\small $(k_3^\star,t_3)$};
\node at (14,5.5) {$\Gamma_2$};
\node at (15.5,3.5) {$\Gamma_1$};
\end{tikzpicture}
\caption{\small Constructing a backwards semi-infinite path from each discrete time point. The upper red/thick path is $\Gamma_2$ and the lower red/thick path is $\Gamma_1$.}
\label{fig:discrete time for constructing bi-infinte path}
\end{figure}

\subsection{Proof of Lemma~\ref{lemma:midpoint prob for BLPP}, Theorem~\ref{thm:general_SIG}, and Parts~\ref{itm:monotonicity of semi-infinite jump times}\ref{itm:strong monotonicity in t} and~\ref{itm:convergence of geodesics}\ref{itm:limits in theta to infty} of Theorem~\ref{existence of semi-infinite geodesics intro version} }

 \begin{proof}[Proof of Lemma~\ref{lemma:midpoint prob for BLPP}]
    We follow the procedure of the proof of Theorem 4.12 in the arXiv version of~\cite{Sepp_lecture_notes}. We take $(m,t) = \mbf 0$, and the general case follows analogously. We simplify further: for any sequence $\{t_k\}_{k \in \Z}$ satisfying 
    \[
    \lim_{n \rightarrow \infty} \f{t_n}{n} = \theta \qquad\text{and}\qquad\lim_{n \rightarrow \infty} \f{t_{-n}}{-n} =\eta,
    \]
    set $\wt t_k = - t_{-(k + 1)}$. Then, $\{\wt t_k\}_{k \in \Z}$ is a sequence satisfying 
    \[
     \lim_{n \rightarrow \infty} \f{\wt t_n}{n} = \eta \qquad\text{and}\qquad\lim_{n \rightarrow \infty} \f{t_{-n}}{-n} =\theta. 
    \]
    Furthermore, a geodesic path between $(-n,t_{-n})$ and $(n,t_n)$ for the original field of Brownian motions $\mbf B = \{B_r\}_{r \in \Z}$, when reflected through the origin, becomes a geodesic path between $(-n,\wt t_{-n})$ and $(n,\wt t_n)$ for the field of Brownian motions $\{\wt B_{-r}\}_{r \in \Z}$. Then, without loss of generality, assume that
    \begin{equation} \label{expectation assumption for midpoint problem}
    \Ee\bigl[\B^\eta((1,-1),\mbf 0) \bigr] \ge \Ee\bigl[\B^\theta((1,-1),\mbf 0)\bigr].
    \end{equation}

    If a geodesic between $(-n,t_{-n})$ and $(n,t_n)$ passes through $\mbf 0$, then the geodesic cannot pass through any point above and to the left of $\mbf 0$. Then, for any $\mbf x \in \Z_{> 0} \times \R_{< 0}$,
    \[
    L_{(-n,t_{-n}),\mbf 0} + L_{\mbf 0,(n,t_n)} \ge L_{(-n,t_{-n}),\mbf x} + L_{\mbf x,(n,t_n)}
    \]
     Specifically, for all $k \in \Z_{> 0}$, 
    \[
     L_{(-n,t_{-n}),\mbf 0} - L_{(-n,t_{-n}),(k,-k)} \ge L_{(k,-k),(n,t_n)} - L_{\mbf 0,(n,t_n)}.
    \]
    For fixed $\mbf x,\mbf y \in \Z \times \R$, let $\wt \B^\eta(\mbf x,\mbf y)$ denote the almost sure limit
    \[
    \lim_{n \rightarrow \infty}  L_{(-n,t_{-n}),\mbf x} - L_{(-n,-t_{-n}),\mbf y}.
    \]
    By Theorem~\ref{thm:dist of Busemann functions}\ref{mutual independence of the X_m} and Lemma~\ref{lemma:Busemann limits for southwest geodesics}, the limit exists almost surely, and $\wt \B^\theta(\mbf x,\mbf y) \deq \B^\theta(\mbf y,\mbf x)$. Let $\Omega_{\Z}$ be the full probability event on which, for every $k \in \Z_{> 0}$ and for every sequence $\{t_n\}_{n \in \Z}$ satisfying
    \[
    \lim_{n \rightarrow \infty} \f{t_n}{n} = \theta\qquad\text{and}\qquad\lim_{n \rightarrow \infty}\f{t_{-n}}{-n} = \eta,
    \]
    we have
    \begin{align*}
    &\lim_{n \rightarrow \infty} L_{(-n,t_{-n},\mbf 0)}(\mbf B) - L_{(-n,t_{-n},(k,-k)}(\mbf B) = \wt \B^\eta(\mbf 0,(k,-k)),\qquad\text{and} \\[1em] &\lim_{n \rightarrow \infty} L_{(k,-k),(n,t_n)}(\mbf B) - L_{\mbf 0,(n,t_n)}(\mbf B) = \B^\theta((k,-k),\mbf 0).
    \end{align*}
    Then, we have the following inclusion of sets:
    \begin{align}
        &\Bigg\{\text{There exists a sequence } \{t_n\}_{n \in \Z}  \text{ satisfying } \lim_{n \rightarrow \infty} \f{t_n}{n} = \theta \text{ and }\lim_{n \rightarrow \infty} \f{t_{-n}}{-n} = \eta \nonumber \\[1em]
         &\qquad\text{such that, for all } n \in \Z_{> 0}, \text{ some geodesic between } (-n,t_{-n}) \text{ and } (n,t_n) \text{ passes through } \mbf 0\Bigg\}\nonumber \\[1em]
         &\subseteq \bigcap_{m = 1}^\infty \Bigg\{\text{There exists a sequence } \{t_n\}_{n \in \Z}  \text{ satisfying } \lim_{n \rightarrow \infty} \f{t_n}{n} = \theta \text{ and }\lim_{n \rightarrow \infty} \f{t_{-n}}{-n} = \eta \text{ such that } \nonumber\\[1em]
         &\quad\text{for all }n \in \Z_{> 0} \text{ and } k = 1,\ldots,m,\; L_{(-n,t_{-n}),\mbf 0}(\mbf B) - L_{(-n,t_{-n}),(k,-k)}(\mbf B) \ge L_{(k,-k),(n,t_n)}(\mbf B) - L_{\mbf 0,(n,t_n)}(\mbf B) \Bigg\}\nonumber \\[1em]
         &\subseteq \bigcap_{m = 1}^\infty \bigl\{\wt \B^\eta(\mbf 0,(k,-k)) \ge \B^\theta((k,-k),\mbf 0),\;\; k = 1,\ldots,m \bigr\} \cup  \Omega_{\Z}^C. \nonumber 
    \end{align}
    We show that this last event has probability $0$. 
    Define the process $\{S_k\}_{k \ge 0}$ by $S_0 = 0$ and for $k \ge 1$,
    \[
    S_k  = \B^\eta(\mbf 0,(k,-k)) - \B^\theta((k,-k),\mbf 0) = \sum_{i = 1}^k  \bigl(\wt \B^\eta((i - 1,- i + 1),(i,-i)) - \B^\theta\bigl((i,-i),(i - 1,-i +1)\bigr)\bigr).
    \]
    The sequences
    \[
    \bigl\{\B^\theta((i,-i),(i - 1,-i +1))\bigr\}_{i \in \Z_{> 0}} \qquad\text{and}\qquad \bigl\{\wt \B^\eta((i - 1,-i + 1),(i,-i))\bigr\}_{i \in \Z_{> 0}} 
    \]
    are independent because they are constructed from disjoint increments of the field of independent Brownian motions $\mbf B$. Furthermore,
    \[
    \bigl\{\B^\theta((i,-i),(i - 1,-i +1))\bigr\}_{i \in \Z_{> 0}} = \bigl\{v_{i}^\theta(-i) + h_{i - 1}^\theta(-i,-i + 1)   \bigr\}_{i \in \Z_{>0}}
    \]
    is an i.i.d. collection of random variables by Theorem~\ref{thm:dist of Busemann functions}\ref{Burke property}. By reflection applied to the Busemann functions $\wt B^\eta$, the process $\{S_k\}_{k \ge 0}$ has independent increments.  By Theorem~\ref{thm:dist of Busemann functions}, the increments of this random walk are independent sums of normal and exponential random variables and therefore have finite first and second moment. Let $\mu$ be the mean and $\sigma^2$ the variance of $S_1$. Then, $\mu \le 0$ by the assumption~\eqref{expectation assumption for midpoint problem}. Thus, for $m \ge 1$, 
    \begin{align*}
    &\Pp\bigl(\wt \B^\eta\bigl(\mbf 0,(k,-k)\bigr)\ge \B^\theta\bigl((k,-k),\mbf 0\bigr),\;\; k = 1,\ldots,m\bigr) = \Pp(S_k \ge 0,\;\; k = 1,\ldots,m) \\[1em]
    = \;&\Pp\biggl(\f{S_{\floor{u m}} - u m \mu}{\sqrt{m \sigma^2}} \ge -\f{u m \mu}{\sqrt{m \sigma^2}}\;\text{for }u \in [0,1]\biggr)  \le \Pp\biggl(\inf_{u \in [0,1]}\f{S_{\floor{u m}} - u m \mu}{\sqrt{m \sigma^2}} \ge 0\biggr)\overset{m\rightarrow \infty}{\longrightarrow} 0.
    \end{align*}
    The convergence in the last step holds because
    \[
    \inf_{u \in [0,1]}\f{S_{\floor{u m}} - u m \mu}{\sqrt{m \sigma^2}} \overset{m\rightarrow \infty}{\Longrightarrow} \inf_{u \in [0,1]} B(u) 
    \]
    where the convergence holds in distribution by Donsker's Theorem, and $B$ is a standard Brownian motion.
 \end{proof}

We are now ready to prove all parts of the main theorem of this paper. 
 
 \begin{proof}[Proof of Theorem~\ref{thm:general_SIG}]
 Unless specified otherwise, the full probability event of the parts of this theorem is the event $\Omega_2$ defined in Equation~\eqref{eqn:Omega2_def}.
 
 \medskip \noindent \textbf{Part~\ref{itm:SIG_existence}:} This is a direct corollary of Theorem~\ref{existence of semi-infinite geodesics intro version}\ref{itm:SIG_existence}.
 
 \medskip \noindent \textbf{Part~\ref{itm:uniqueness of geodesic for fixed point and direction}:} 
 Let $\Omega_{\mbf x}^{(\theta)}$ be the event defined in Equation~\eqref{eqn:omegaxtheta}. By Lemma~\ref{lemma: uniqueness and directedness of geodesics for fixed parameters}
 on $\Omega_{\mbf x}^{(\theta)}$, there is a unique element of $\mbf T_{\mbf x}^{\theta}$. Setting $\mbf x = (m,t)$, $\tau_{(m,t),r}^{\theta-,L} = \tau_{(m,t),r}^{\theta+,R}$ for all $r \ge m$, and the desired conclusion follows from Theorem~\ref{thm:convergence and uniqueness}\ref{all semi-infinite geodesics lie between leftmost and rightmost}.
 
 \medskip \noindent \textbf{Part~\ref{itm:size of non-uniqueness}:} By Theorems~\ref{thm:summary of properties of Busemanns for all theta}\ref{busemann functions agree for fixed theta} and~\ref{thm:convergence and uniqueness}\ref{all semi-infinite geodesics lie between leftmost and rightmost}, on the full probability event $\Omega^{(\theta)} \cap \Omega_2$, there exists multiple $\theta$-directed semi-infinite geodesics from $(m,t)$ if and only if $\tau_{(m,t),r}^{\theta,L} < \tau_{(m,t),r}^{\theta,R}$ for some $r \ge m$, or in other words, if and only if $\mbf T_{(m,t)}^{\theta}$ contains more than one element. Since the event $\wt \Omega^{(\theta)}$ constructed in Equation~\eqref{eqn:omega tilde space} is contained in $\Omega^{(\theta)}$, Theorem~\ref{thm:non_unique_size}\ref{decomposition} implies that this set of points whose $\theta$-directed geodesic is not unique is countable. Then, since $\wt \Omega^{(\theta)} \subseteq \bigcap_{\mbf x \in \Z \times \Q} \Omega_{\mbf x}^{(\theta)}$, if two $\theta$-directed geodesics from $(m,t)$ pass through $(m,t + \ve)$ for some $\ve > 0$, then they both pass through $(m,q)$ for some $(m,q) \in \Z \times \Q$. Further, the portions of the geodesics after $(m,q)$ are both $\theta$-directed semi-infinite geodesics from $(m,q)$ and are therefore the same geodesic by Part~\ref{itm:uniqueness of geodesic for fixed point and direction}.

\medskip  \noindent \textbf{Part~\ref{itm:all geodesics are directed}}: Assume to the contrary, that, for some $\omega \in \Omega_2$, there exists a point $(m,t) \in \Z \times \R$ and a sequence $t = t_{m - 1} \le t_m \le t_{m + 1} \le \cdots$ defining a semi-infinite geodesic $\Gamma$, starting from $(m,t)$ and satisfying
\[
0 \le \underline \theta := \liminf_{n\rightarrow \infty}\f{t_n}{n} < \limsup_{n\rightarrow \infty} \f{t_n}{n} =: \overline \theta \le \infty.
\]
Choose some $\theta \in (\underline \theta,\overline \theta)$. By Theorem~\ref{existence of semi-infinite geodesics intro version}\ref{general limits for semi-infinite geodesics},
\[
\lim_{n\rightarrow \infty}\f{\tau_{(m,t),n}^{\theta-,L}}{n} = \theta
\]
Therefore, there are infinitely many values of $n$ with $t_n < \tau_{(m,t),n}^{\theta - ,L}$ and another infinitely many values of $n$ with $t_n > \tau_{(m,t),n}^{\theta-,L}$. Let $n_1$ be the minimal index such that $t_{n_1} < \tau_{(m,t),n_1}^{\theta-,L}$, and let $n_2$ be the minimal index larger than $n_1$ such that $t_{n_2} > \tau_{(m,t),n_2}^{\theta-,L}$. See Figure~\ref{fig:double crossing of geodesics} for clarity. By planarity, the points $(n_1,t_1)$ and $(n_2,\tau_{(m,t),n_2 - 1}^{\theta -,L})$ lie on both geodesics, and between the two points, the path $\Gamma$ lies strictly to the left of the leftmost geodesic in $\mbf T_{(m,t)}^{\theta-}$. This contradicts Theorem~\ref{existence of semi-infinite geodesics intro version}\ref{Leftandrightmost} which states the leftmost geodesic in $\mbf T_{(m,t)}^{\theta-}$ is the leftmost geodesic between any two of its points. 

\begin{figure}[t]
\centering
\begin{tikzpicture}
\draw[gray,thin, dashed] (-2,-2)--(13,-2);
\draw[gray,thin, dashed] (-2,-1)--(13,-1);
\draw[gray,thin, dashed] (-2,0) -- (13,0);
\draw[gray,thin, dashed] (-2,1) --(13,1);
\draw[gray, thin, dashed] (-2,2)--(13,2);
\draw[gray,thin, dashed] (-2,3)--(13,3);
\draw[red, ultra thick,->] (-1.5,-2)--(-0.5,-2)--(-0.5,-1)--(1.5,-1)--(1.5,0)--(6,0)--(6,1)--(7,1)--(7,2)--(9,2)--(9,3)--(10,3);
\draw[blue, thick,->] (-1.5,-2)--(2,-2)--(2,-1)--(2.5,-1)--(2.5,0.05)--(3.5,0.05)--(3.5,1)--(4,1)--(4,2)--(11,2)--(11,3)--(13.5,3);
\filldraw[black] (-1.5,-2) circle (2pt) node[anchor = north] {\small $(m,t)$};
\filldraw[black] (9,2) circle (2pt) node[anchor = north] {\small $(n_2,\tau_{(m,t),n_2 }^{\theta-,L})$};
\filldraw[black] (11,2) circle (2pt) node[anchor = north] {\small $(n_2,t_{n_2})$};
\filldraw[black] (3.5,0) circle (2pt) node[anchor = north] { \small$(n_1,t_{n_1})$};
\filldraw[black] (7,2) circle (2pt) node[anchor = south] {\small $(n_2,\tau_{(m,t),n_2 - 1}^{\theta -,L})$};
\filldraw[black] (6,0) circle (2pt) node[anchor = north] {\small $(n_1,\tau_{(m,t),n_1}^{\theta -,L})$};
\end{tikzpicture}
\caption{\small Crossing of geodesic paths. The blue/thin path is the semi-infinite geodesic defined by the jump times $t \le t_m \le t_{m + 1} \le \cdots$, and the red/thick path is the semi-infinite geodesic defined by the times $t \le \tau_{(m,t),m}^{\theta -,L} \le \tau_{(m,t),m + 1}^{\theta -,L} \le \cdots$.}
\label{fig:double crossing of geodesics}
\end{figure}

\medskip \noindent \textbf{Part~\ref{itm:only vertical or horizontal geodesics are trivial}}:
Next, assume that for some $\omega \in \Omega_2$, there exists $(m,t) \in \Z \times \R$, and  a sequence $t = t_{m - 1} \le t_m \le \cdots$ that defines a semi-infinite geodesic $\Gamma$, starting from $(m,t)$ and satisfying
\[
\lim_{n\rightarrow \infty} \f{t_n}{n} = 0.
\]
To show that $t_r = t$ for all $r \ge m$, it is sufficient to show that $t_m =t$. For then, this semi-infinite geodesic travels vertically to $(m + 1,t)$, the remaining part of the geodesic is a semi-infinite geodesic starting at $(m + 1,t)$, and the result follows by induction. By Theorem~\ref{existence of semi-infinite geodesics intro version}\ref{general limits for semi-infinite geodesics}, for every $\theta > 0$ and any $(m,t) \in \Z \times \R$,
\[
\lim_{n\rightarrow \infty} \f{\tau_{(m,t),n}^{\theta+,R}}{n} = \theta.
\]
Hence, for all $\theta > 0$ and all sufficiently large $n$, $t_n < \tau_{(m,t),n}^{\theta+,R}$. However, by an analogous argument as in Part~\ref{itm:all geodesics are directed}, again using Theorem~\ref{thm:existence of Busemann functions for fixed points}\ref{Leftandrightmost}, $t_r \le \tau_{(m,t),r}^{\theta+,R}$ for all $r \ge m$. Specifically, the inequality holds for $r = m$. By monotonicity of Theorem~\ref{existence of semi-infinite geodesics intro version}\ref{itm:monotonicity of semi-infinite jump times}\ref{itm:monotonicity in theta}, 
$\lim_{\theta \searrow 0} \tau_{(m,t),m}^{\theta +,R}$ exists.
By definition of the event $\Omega_1$ from Lemma~\ref{lemma:Omega 1 prob 1}, this limit equals $t$ on the event $\Omega_1 \supseteq \Omega_2$. The case where $\lim_{n \rightarrow \infty}\f{t_n}{n} = \infty$ is handled similarly.

 
\medskip \noindent  \textbf{Part~\ref{itm:no bi-infinite geodesics in given direction}:} Let $\Omega^{(\theta,\eta)}$ be a full probability event on which, for every $(m,q) \in \Z \times \Q$, there exists no sequences satisfying the conditions~\eqref{eqn:limit condition for bi-infintie geodesics} and such that, for every $n \in \Z$, there exists a geodesic between $(-n,\tau_{-n})$ and $(n,\tau_n)$ that passes through $(m,q)$. Such an event exists by Lemma~\ref{lemma:midpoint prob for BLPP}. Then, on this event, 
if a bi-infinite geodesic that satisfies~\eqref{eqn:limit condition for bi-infintie geodesics} exists, it cannot 
 pass through  $(m,q)$ for any $q \in \Q$. But then this bi-infinite geodesic cannot ever move horizontally, and so there must exist $t \in \R$ such that the bi-infinite geodesic consists only of points $(r,t)$. Now~\eqref{eqn:limit condition for bi-infintie geodesics} fails. 
 
 \medskip \noindent \textbf{Part~\ref{itm:coalescence}:}
 Define
\[
\widehat \Omega^{(\theta)} = \wt \Omega^{(\theta)} \cap  \Omega^{(\theta,\star)} \cap \Omega^{(\theta,\theta)}(\mbf X^\theta),
\]
where $\wt \Omega^{(\theta)}$ is the event defined in~\eqref{eqn:omega tilde space}, $\Omega^{(\theta,\star)}$ is the event of Theorem~\ref{existence of backwards semi-infinite geodesics}\ref{southwest geodesics limits}, and $\Omega^{(\theta,\theta)}(\mbf X^\theta)$ is the event of Part~\ref{itm:no bi-infinite geodesics in given direction}, applied to the random environment $\mbf X^\theta$ of i.i.d Brownian motions  (Lemma~\ref{thm:dist of Busemann functions}\ref{mutual independence of the X_m}).

By the construction of the $\mbf T_{\mbf x}^{\theta }$ in terms of the variational formula (Definition~\ref{def:semi-infinite geodesics}), if for some $\mbf x,\mbf y \in \Z \times \R$, the rightmost semi-infinite geodesics in $\mbf T_{\mbf x}^{\theta }$ and $\mbf T_{\mbf y}^{\theta }$ ever intersect, they agree above and to the right of the point of intersection. The same is true of leftmost geodesics. Because $\wt \Omega^{(\theta)} \subseteq \Omega^{(\theta)} \cap \Omega_2$ (see discussion after~\eqref{eqn:omega tilde space}), 
\[
\widehat \Omega^{(\theta)} \subseteq \Omega^{(\theta)} \cap \Omega^{(\theta,\star)} \cap \Omega_2.
\]
If, by way of contradiction, for some $\omega \in \widehat \Omega^{(\theta)}$, there exists $\mbf x$ and $\mbf y \in \Z \times \R$ such that the rightmost geodesics in $\mbf T_{\mbf x}^{\theta}$ and $\mbf T_{\mbf y}^{\theta}$ are disjoint, then by Theorem~\ref{thm:bi-infinite geodesic between disjoint geodesics}, there exists a bi-infinite geodesic for the environment $\mbf X^\theta$ that is defined by jump times $\{\tau_r\}_{r \in \Z}$ and satisfies
\[
\lim_{n \rightarrow \infty} \f{\tau_n}{n} = \theta = \lim_{n \rightarrow \infty} \f{\tau_{-n}}{-n}.
\]
This is a contradiction since $\omega \in \Omega^{(\theta,\theta)}(\mbf X^\theta)$. Therefore, for all $\omega \in \widehat \Omega^{(\theta)}$, whenever $(m,s),(k,t) \in \Z \times \R$, the rightmost geodesics in $\mbf T_{(m,s)}^\theta$ and $\mbf T_{(r,t)}^\theta$ coalesce. The same is true by replacing ``right" with ``left." In other words, for all sufficiently large $r$, 
\begin{equation} \label{eqn:coalRL}
\tau_{(m,s),r}^{\theta,R} = \tau_{(k,t),r}^{\theta,R} \qquad\text{and}\qquad \tau_{(m,s),r}^{\theta,L} = \tau_{(k,t),r}^{\theta,L}.
\end{equation}
Now, let $s = s_{m - 1} \le s_m \le \cdots$ and $t = t_{k -1 } \le t_k \le \cdots$ be any sequences defining $\theta$-directed semi-infinite geodesics starting from $(m,s)$ and $(k,t)$, respectively. We show that these semi-infinite geodesics coalesce. Without loss of generality, assume that $m \le k$ and $s_{k -1} < t_{k - 1} = t $. The other cases are handled similarly. Then, by Theorems~\ref{thm:convergence and uniqueness}\ref{all semi-infinite geodesics lie between leftmost and rightmost} and~\ref{existence of semi-infinite geodesics intro version}\ref{itm:monotonicity of semi-infinite jump times}\ref{itm:strong monotonicity in t}, for all $r \ge k$,
\[
\tau_{(k,s_{k - 1}),r}^{\theta,L} \le s_r \le \tau_{(k,s_{k - 1}),r}^{\theta,R} \le \tau_{(k,t),r}^{\theta,L} \le t_r \le \tau_{(k,t),r}^{\theta,R}.
\]
Then, by~\eqref{eqn:coalRL}, these inequalities are all equalities for all sufficiently large $r$, and the geodesics coalesce. 
\end{proof}

We conclude this section by completing the remaining parts of Theorem~\ref{existence of semi-infinite geodesics intro version}.
\begin{proof}[Proof of remaining parts of Theorem~\ref{existence of semi-infinite geodesics intro version}:]

\textbf{Part~\ref{itm:monotonicity of semi-infinite jump times}\ref{itm:strong monotonicity in t}:} Let $\omega \in \wt \Omega^{(\theta)}$, $m \in \Z$, and $s < t \in \R$. We use a modified induction, in the following manner.
\begin{enumerate} [label=\rm(\Alph{*}), ref=\rm(\Alph{*})]  \itemsep=3pt 
    \item First, note by definition that $s = \tau_{(m,s),m - 1}^{\theta,R} < \tau_{(m,t),m - 1}^{\theta,L} = t$.
    \item For each $r \ge m$, we assume that $\tau_{(m,s),r - 1}^{\theta,R} < \tau_{(m,t),r - 1}^{\theta,L}$. \label{itm:inductive asumption}
    \item Under assumption~\ref{itm:inductive asumption}, we show that if $\tau_{(m,s),r}^{\theta,R} \ge \tau_{(m,t),r}^{\theta,L}$, then $\tau_{(m,s),k}^{\theta,R} = \tau_{(m,t),k}^{\theta,L}$ for all $k \ge r$.  \label{itm:inductive conclusion}
    \end{enumerate}
    \medskip
    
    \noindent With this procedure mapped out, assume that, for some $r \ge m$, 
    \[
    \tau_{(m,s),r - 1}^{\theta,R} < \tau_{(m,t),r - 1}^{\theta,L}\qquad\text{and}\qquad \tau_{(m,s),r}^{\theta,R} \ge \tau_{(m,t),r}^{\theta,L}.
    \]
    By definition, $\tau_{(m,s),r}^{\theta,R}$ is a maximizer of $B_r(u) - h_{r + 1}^\theta(u)$ over $u \in [\tau_{(m,s),r-1}^{\theta,R},\infty)$. But since 
    \[
    \tau_{(m,s),r}^{\theta,R} \ge \tau_{(m,t),r}^{\theta,L} \ge \tau_{(m,t),r - 1}^{\theta,L},
    \]
    $\tau_{(m,s),r}^{\theta,R}$ is also a maximizer over $u \in [\tau_{(m,t),r - 1}^{\theta,L},\infty)$. By definition, $\tau_{(m,t),r}^{\theta,L}$ is another maximizer over this set, so
    \[
    B_r(\tau_{(m,s),r}^{\theta,R}) - h_{r + 1}^\theta(\tau_{(m,s),r}^{\theta,R}) = B_r(\tau_{(m,t),r}^{\theta,L}) - h_{r + 1}^\theta(\tau_{(m,t),r}^{\theta,L}),
    \]
   and both $\tau_{(m,s),r}^{\theta,R}$ and $\tau_{(m,t),r}^{\theta,L}$ are maximizers of $B_r(u) - h_{r + 1}^{\theta}(u)$ over $u \in [q,\infty)$ for any  rational $q \in [\tau_{(m,s),r -1}^{\theta,R},\tau_{(m,t),r - 1}^{\theta,L}]$. By assumption~\ref{itm:inductive asumption}, such a rational $q$ exists. Then, the sequence $\{\tau_{(m,s),k}^{\theta,R}\}_{k \ge r}$ and the sequence $\{\tau_{(m,t),k}^{\theta,L}\}_{k \ge r}$ both define jump times for a semi-infinite geodesic starting from $(r,q)$. Since $\omega \in \wt \Omega^{(\theta)}$, there is a unique $\theta$-directed semi-infinite geodesic starting from this point, and conclusion~\ref{itm:inductive conclusion} holds.

\medskip \noindent \textbf{Part~\ref{itm:convergence of geodesics}\ref{itm:limits in theta to infty}}: By the monotonicity of Theorem~\ref{existence of semi-infinite geodesics intro version}\ref{itm:monotonicity of semi-infinite jump times}\ref{itm:monotonicity in theta}, it suffices to show that
\[
\lim_{\theta \searrow 0}\tau_{(m,t),r}^{\theta +,R} = t \qquad\text{and}\qquad\lim_{\theta \rightarrow \infty} \tau_{(m,t),r}^{\theta-,L} = \infty\qquad\text{for } r \ge m.
\]
 The proof of Theorem~\ref{thm:general_SIG}\ref{itm:only vertical or horizontal geodesics are trivial}, established the statement for $r = m$. For the limits as $\theta \rightarrow \infty$, this implies the statement holds for all $r \ge m$ since $\tau_{(m,t),m}^{\theta -,L} \le \tau_{(m,t),r}^{\theta-,L}$. By the monotonicity of Theorem~\ref{existence of semi-infinite geodesics intro version}\ref{itm:monotonicity of semi-infinite jump times}\ref{itm:monotonicity in theta}, the limit
\[
\tau_{(m,t),r}^{0} := \lim_{\theta \searrow 0} \tau_{(m,t),r}^{\theta +,R}
\]
exists and satisfies $\tau_{(m,t),r}^{0} \le \tau_{(m,t),r}^{\theta+,R}$ for all $\theta > 0$. Furthermore, since $\{\tau_{(m,t),r}^{\theta+,R}\}_{r \ge m - 1}$ is a nondecreasing sequence for each $\theta > 0$, the sequence $\{\tau_{(m,t),r}^{0}\}_{r \ge m - 1}$ is also nondecreasing. By Theorem~\ref{existence of semi-infinite geodesics intro version}\ref{general limits for semi-infinite geodesics},
\[
0 \le \limsup_{n \rightarrow \infty} \f{\tau_{(m,t),n}^{0}}{n} \le \lim_{n \rightarrow \infty} \f{\tau_{(m,t),n}^{\theta +,R}}{n} = \theta\qquad\text{ for all }\theta > 0 \qquad\Longrightarrow\qquad \lim_{n \rightarrow \infty} \f{\tau_{(m,t),n}^{0}}{n}= 0.
\]
Therefore, by Theorem~\ref{thm:general_SIG}\ref{itm:only vertical or horizontal geodesics are trivial}, if the sequence of jump times $t = \tau_{(m,t),m -1}^{0,R} \le \tau_{(m,t),m}^{0,R} \le \cdots$ defines a semi-infinite geodesic starting from $(m,t)$, the desired conclusion follows. It is sufficient to show that, for any $n \ge m$, the sequence $t = \tau_{(m,t),m - 1}^{0} \le \tau_{(m,t),m}^{0} \le \cdots \le \tau_{(m,t),n}^{0}$ defines jump times for a finite geodesic between $(m,t)$ and $(n,\tau_{(m,t),n}^{0})$. By Theorem~\ref{existence of semi-infinite geodesics intro version}\ref{energy of path along semi-infinte geodesic}, For each $\theta > 0$, the sequence $t = \tau_{(m,t),m - 1}^{\theta+,R} \le \cdots \le \tau_{(m,t),n}^{\theta+,R}$ is a maximizing sequence for
\[
L_{(m,t),(n,\tau_{(m,t),n}^{\theta +,R})} = \sup\Bigl\{\sum_{r = m}^n B_r(s_r,s_{r - 1}): t = s_{m - 1} \le s_m \le \cdots \le s_n = \tau_{(m,t),n}^{\theta +,R}    \Bigr\},
\]
so since $\lim_{\theta \searrow 0}\tau_{(m,t),n}^{\theta +,R} = \tau_{(m,t),n}^0$, Lemma~\ref{lemma:convergence of maximizers from converging sets} completes the proof. 
\end{proof}

\appendix
 \section{Deterministic facts about continuous functions}

\begin{lemma} \label{monotonicity of maximizers from function monotonicity}
Let $\eta^1,\eta^2:\R\rightarrow \R$ be continuous functions satisfying
\[
\limsup_{t\rightarrow \infty} \eta^i(t) =-\infty \qquad\text{for }i = 1,2.
\]
Further, assume that $\eta^1(s,u) \le \eta^2(s,u)$ for all $s < u$. For each $t\in \R$ and $i = 1,2$, let $s_t^{i,L}$ and $s_t^{i,R}$ be the leftmost and rightmost maximizers of $\eta^i$ on the set $[t,\infty)$. Then, 
\[
s_t^{1,L} \le s_t^{2,L} \qquad \text{and}\qquad s_t^{1,R} \le s_t^{2,R}.
\]
\end{lemma}
\begin{proof}
Since $s_t^{2,L}$ is a maximizer for $\eta^2$ on $[t,\infty)$, $\eta^2(s_t^{2,L}) \ge \eta^2(s)$, or equivalently, $\eta^2(s_t^{2,L},s) \le 0$  for all $s \ge t$, Then, for all $s \ge s_{t}^{2,L} \ge t$, the hypothesis of the lemma gives
\[
\eta^1(s_t^{2,L},s) \le \eta^2(s_t^{2,L},s) \le 0.
\]
Thus, the leftmost maximum of $\eta^1$ on the larger set $[t,\infty)$ can be no larger than $s_t^{2,L}$.

For rightmost maximizers, the lemma follows by similar reasoning: for $s > s_t^{2,R}$,
\[
\eta^1(s_t^{2,R},s) \le \eta^2(s_t^{2,R},s) < 0,
\]
so $\eta^1(s_t^{2,R}) > \eta^1(s)$ for all $s > s_t^{2,R}$, and no maximizer of $\eta^1(s)$ over $s \in[t,\infty)$ can be larger than $s_t^{2,R}$.
\end{proof}
\begin{lemma} \label{lemma:convergence of maximizers from converging functions}
Let $S \subseteq \R^n$, and let  $f_n:S \rightarrow \R$ be a sequence of continuous functions, converging uniformly to the function $f:S \rightarrow \R$. Assume that there exists a sequence $\{c_n\}$, of maximizers of $f_n$, converging to some $c \in S$. Then, $c$ is a maximizer of $f$. 
\end{lemma}
\begin{proof}
$f_n(c_n) \ge f_n(x)$ for all $x \in S$, so  it suffices to show that $f_n(c_n) \rightarrow f(c)$. This follows from the uniform convergence of $f_n$ to $f$, the continuity of $f$, and
\[
|f_n(c_n) - f(c)| \le |f_n(c_n) - f(c_n)| +|f(c_n) - f(c)|.  \qedhere
\]
\end{proof}
\begin{lemma} \label{lemma:convergence of maximizers from converging sets}
Let $S_n$ for $n \ge 0$ be subsets of some set $\wt S \subseteq \R^n$, on which the function $f:\wt S \rightarrow \R$ is continuous. Assume that each point $x \in S_0$ is the limit of a sequence $\{x_n\}$, where $x_n \in S_n$ for each $n$. Assume that $\{c_n\}$ is a sequence of maximizers of $f$ on $S_n$. Assume further that $c_n$ converges to some $c \in S_0$. Then, $c$ is a maximizer of $f$ on $S_0$. 
\end{lemma}
\begin{proof}
For each $x_0 \in S_0$, write $x_0 = \lim_{n\rightarrow \infty} x_n$, where $x_n \in S_n$ for each $n$. Then, $f(c_n) \ge f(x_n)$ for all $n \ge 1$, and the result follows by taking limits.
\end{proof}

\begin{lemma} \label{lemma:bounds for differences of Brownian LPP}
Let $\mbf X = \{X_m\}_{m \in \Z}$, where each $X_m:\R\rightarrow \R$ is a continuous function. Let $0 < s < t < T < u$ and $m \leq n$. Then,
\[
X_m(s,t) \le L_{(m,s),(n,u)}(\mbf X)  - L_{(m,t),(n,u)}(\mbf X) \leq L_{(m,s),(n,T)}(\mbf X)  - L_{(m,t),(n,T)}(\mbf X)
\]
Similarly, let $0 < s < t < u < \infty$ and $m < n$. Then,
\[
0 \le L_{(m,s),(n,t)}(\mbf X)  - L_{(m + 1,s),(n,t)}(\mbf X)  \leq L_{(m,s),(n,u)}(\mbf X)  - L_{(m + 1,s),(n,u)}(\mbf X) .
\]
\end{lemma}
\begin{remark}
This is a deterministic statement. The only necessary ingredient is the continuity of the $X_m$ so that each of the last-passage times is finite and has a sequence of maximizing times. 
\end{remark}
\begin{proof}
This proof follows a standard paths-crossing argument. For example, the proofs of Lemma 4.6 in~\cite{blpp_utah} and Proposition 3.8 in~\cite{Directed_Landscape} follow the same procedure. We prove the first statement, and the second is proven similarly. By definition of last-passage time, 
\[
X_m(s,t) + L_{(m,t),(n,u)} \le L_{(m,s),(n,u)}.
\]
  Since $0 < s < t < T < u$, any geodesic between $(m,s)$ and $(n,u)$ must cross any geodesic between $(m,t)$ and $(n,T)$. Let $\mbf z \in \Z \times \R$ be a point of intersection. Then,
\[
L_{(m,s),(n,u)} = L_{(m,s),\mbf z} + L_{\mbf z,(n,u)}\qquad\text{and}\qquad L_{(m,t),(n,T)} = L_{(m,t),\mbf z} + L_{\mbf z,(n,T)},
\]
so
\begin{align*}
&L_{(m,s),(n,u)} - L_{(m,t),(n,u)} \le L_{(m,s),\mbf z} + L_{\mbf z,(n,u)} - \bigl(L_{(m,t),\mbf z} + L_{\mbf z,(n,u)}\bigr) = L_{(m,s),\mbf z} - L_{(m,t),\mbf z} \\[1em]
= &L_{(m,s),\mbf z} + L_{\mbf z,(n,T)} - \bigl(L_{(m,t),\mbf z} + L_{\mbf z,(n,T)}\bigr) \le L_{(m,s),(n,T)} - L_{(m,t),(n,T)}.\qedhere
\end{align*}
\end{proof}

\section{Probabilistic results}
The following is a classical result that is often used in this paper.
\begin{lemma}[\cite{BM_handbook}, Equation 1.1.4 (1) on pg 251] \label{lemma:sup of BM with drift}
For a standard Brownian motion $B$ and $\lambda > 0$,
\[
\sup_{0 \leq s < \infty}\{\sqrt 2 B(s) - \lambda s\} \sim \operatorname{Exp}(\lambda).
\]
\end{lemma}
\begin{theorem} \label{thm:dist of busemann increment}
Let $B$ be a standard Brownian motion, and for $t > 0$, let 
\begin{align*}
D(t) &:= \sup_{0 \le s < \infty}\bigl\{\sqrt 2 B(s) - \lambda s\bigr\} - \sup_{t \le s < \infty}\bigl\{\sqrt 2 B(s) - \lambda s\bigr\}\\[1em]
&= \bigl(\sup_{0 \le s \le t}\bigl\{\sqrt 2 B(t,s) + \lambda(t- s)\bigr\} - \sup_{t \le s < \infty}\bigl\{\sqrt 2 B(t,s) + \lambda (t - s)\bigr\}\bigr)^+.
\end{align*}
Then, for all $z \ge 0$,
\begin{equation} \label{CDF of Busemann increments}
\Pp(D(t) \le z) = \Phi\bigl(\frac{z - \lambda t}{\sqrt{2 t}}\bigr) + e^{\lambda z}\left( (1 + \lambda z + \lambda^2 t)\Phi\bigl(-\frac{z + \lambda t}{\sqrt{2 t}}\bigr) - \lambda \sqrt{\frac{t}{ \pi}}e^{-\frac{(z + \lambda t)^2}{4 t}}    \right).
\end{equation}
\end{theorem}
\begin{proof}
For $x \ge 0$, by time reversal and~\cite{BM_handbook}, Equation 1.2.4 on page 251,
\begin{align*}
&\quad\;\Pp\bigl(\sup_{0 \le s \le t} \{\sqrt 2 B(t,s) + \lambda(t - s)\} \le x\bigr)
= \Pp\bigl(\sup_{0 \le s \le t}\{\sqrt 2 B(t,t - s) - \lambda(t - s - t)\} \le x\bigr) \\[1em]
&= \Pp\bigl(\sup_{0 \le s \le t}\{\sqrt 2 B(s) + \lambda s\} \le x\bigr) 
=  \Phi\bigl(\frac{x - \lambda t}{\sqrt{2 t}}\bigr) - e^{\lambda x}\Phi\bigl(\frac{-x - \lambda t}{\sqrt{2t}}\bigr).
\end{align*}
By Lemma~\ref{lemma:sup of BM with drift}, 
\[
\sup_{t \le s < \infty} \{\sqrt 2 B(t,s) + \lambda(t - s)\} \sim \operatorname{Exp}(\lambda).
\]
The conclusion of the theorem follows by a simple, but tedious convolution. It suffices to show the $t = 1$ case, and the general case follows by Brownian scaling. 
\begin{align*}
&\Pp(D(1) \le z) 
= \int_{-\infty}^0 \left(\Phi\bigl(\frac{z-y - \lambda }{\sqrt{2}}\bigr) - e^{\lambda (z - y)}\Phi\bigl(\frac{-z + y - \lambda }{\sqrt{2}}\bigr)\right)\lambda e^{\lambda y}\, dy \\[1em]
&= \int_{-\infty}^0\int_{-\infty}^{\frac{z-y - \lambda}{\sqrt 2}} \frac{\lambda}{\sqrt{2\pi}}e^{-x^2/2}e^{\lambda y}\, dx\, dy - e^{ \lambda z}\int_{-\infty}^0\int_{-\infty}^{\frac{-z + y - \lambda}{\sqrt 2}} \frac{\lambda}{\sqrt{2\pi}} e^{-x^2/2}\, dx\, dy.
\end{align*}
We now use Fubini's Theorem to switch the order of integration. This results in 
\begin{align*}
&= \int_{-\infty}^{\zfraclambda}\int_{-\infty}^0 \frac{\lambda}{\sqrt{2\pi}} e^{-x^2/2}e^{\lambda y}\, dy\, dx  + \int_{\zfraclambda}^\infty \int_{-\infty}^{z - \lambda - \sqrt 2 x} \frac{\lambda}{\sqrt{2\pi}} e^{-x^2/2}e^{\lambda y}\, dy\, dx \\[1em]
&\qquad\qquad\qquad\qquad\qquad\qquad\qquad\qquad\qquad-e^{ \lambda z}\int_{-\infty}^{\negzfraclambda}\int_{ \sqrt 2 x +z + \lambda }^0 \frac{\lambda}{\sqrt{2\pi}} e^{-x^2/2}\, dy\, dx \\[1em]
&= \int_{-\infty}^{\zfraclambda} \frac{1}{\sqrt{2\pi}} e^{-x^2/2}\, dx + \int_{\zfraclambda}^\infty \frac{1}{\sqrt{2\pi}} e^{-x^2/2}e^{z\lambda -\sqrt 2 \lambda x - \lambda^2}\, dx 
+e^{ \lambda z}\int_{-\infty}^{\frac{-z - \lambda}{\sqrt 2}} (\sqrt 2 x + z + \lambda)\frac{\lambda}{\sqtwopi} e^{-x^2/2}\, dx \\[1em]
&= \Phi\bigl(\zfraclambda\bigr) + e^{\lambda z}\int_{\zfraclambda}^\infty \frac{1}{\sqtwopi} e^{-\frac{(x + \sqrt 2 \lambda)^2}{2}}\, dx 
+ e^{\lambda z}\left(\int_{-\infty}^{\negzfraclambda} \frac{\lambda }{\sqrt{\pi}} xe^{-x^2/2}\, dx +   (\lambda z + \lambda^2) \Phi\bigl(\negzfraclambda\bigr)  \right) \\[1em]
&= \Phi\bigl(\zfraclambda\bigr) + e^{\lambda z}\int_{\frac{z + \lambda}{\sqrt 2}}^\infty \frac{1}{\sqtwopi} e^{-u^2/2}\, du
+ e^{\lambda z}\left(-\frac{\lambda}{\sqrt{\pi}}e^{-\frac{(z + \lambda)^2}{4}} +   (\lambda z + \lambda^2) \Phi\bigl(\negzfraclambda\bigr)  \right) \\[1em]
&= \Phi\bigl(\zfraclambda\bigr) + e^{\lambda z}\left( (1 + \lambda z + \lambda^2)\Phi\bigl(\negzfraclambda\bigr) - \frac{\lambda}{\sqrt \pi}e^{-\frac{(z + \lambda)^2}{4}}    \right). \qedhere
\end{align*}
\end{proof}

\begin{theorem} \label{distribtution of argmax BM with drift}
Let $B$ be a standard Brownian motion and let $\lambda > 0$. Let $T$ be the unique maximizer of $\sqrt 2B(t) - \lambda t$ for $t \in [0,\infty)$. Then, for $t \ge 0$, 
\[
\Pp(T > t) = (2 + \lambda^2 t) \Phi\left(-\lambda\sqrt{{t}/{2}}\,\right) - \lambda \sqrt{{t}/{\pi}\,}\,e^{-\f{\lambda^2 t}{4}}.
\]
\end{theorem}
\begin{remark}
 Theorem~\ref{distribtution of argmax BM with drift} should be credited to Norros and Salminen, who computed the Laplace transform of this random variable in Proposition 3.9 of~\cite{Norros-Salminen}. One can also obtain by integrating the probability density of the time of maximum of Brownian motion with drift on the interval $[0,t]$ found in~\cite{Buffet_Time_of_BMdrift_max}, Equation (1.3), and then taking $t\rightarrow \infty$. See also the discussion after Equation (1.3) equation in~\cite{Buffet_Time_of_BMdrift_max} for more historical details on this formula. Theorem~\ref{thm:dist of busemann increment} will be used in subsequent papers, so we use it to give our own proof of Theorem~\ref{distribtution of argmax BM with drift}.
\end{remark}
\begin{proof}[Proof of Theorem~\ref{distribtution of argmax BM with drift}]
Note that $T > t$ if and only if $D(t) = 0$ in the sense of Theorem~\ref{thm:dist of busemann increment}. Thus, the Lemma follows by setting $z = 0$ in the formula~\eqref{CDF of Busemann increments}.  
\end{proof}

\begin{theorem} \label{thm:countable non unique maximizers}
Let $X$ be a two-sided Brownian motion with strictly negative drift. Let 
\[
M = \{t \in \R: X(t) = \sup_{t \le s < \infty} X(s) \}.
\]
Furthermore, let  
\[
M^U = \big\{t \in M: X(t) > X(s) \text{ for all }s > t   \big\}
\]
 be the set of   $t\in M$ that are unique maximizers of $X(s)$ over $s \in [t,\infty)$.  Define
$
    M^N = M \setminus M^U
$
to be the set of $t\in M$ that are non-unique maximizers of $X(s)$ over $s \in [t,\infty)$.
Then, there exists an event of probability one, on which the following hold.
\begin{enumerate} [label=\rm(\roman{*}), ref=\rm(\roman{*})]  \itemsep=3pt 
    \item $M$ is a closed set. \label{Mclosed}
    \item \label{non_discrete} For all $\hat t \in M^U$ and $\ve > 0$, there exists $t \in M^N$ satisfying $\hat t < t < \hat t + \ve$. For all $t \in M^N$ and $\ve > 0$, there exists $\hat t \in M^U$ satisfying $t - \ve < \hat t < t$. 
    \item \label{non_discrete MN}For all $t \in M^N$ and $\ve > 0$, there exists $t^\star \in M^N$ with $t - \ve < t^\star < t$. For each $t \in M^N$, there exists $\delta > 0$ such that $M \cap (t,t+ \delta) = \varnothing$.
    \item $M^N$ is a countably infinite set. \label{countable}
\end{enumerate}
\end{theorem}
\begin{proof} 
By the $n = m$ case of Lemma~\ref{lm:point-to-line uniqueness}, for each fixed $t \in \R$, $X(s)$ is almost surely uniquely maximized for $s \in [t,\infty)$. By Theorem~\ref{distribtution of argmax BM with drift}, this maximizer is almost surely strictly greater than $t$. Let $\Omega_{\Q}$ be the full probability event on which,
\[
\lim_{s \rightarrow -\infty} X(s) = \infty \qquad\text{and}\qquad\lim_{s \rightarrow \infty} X(s) = -\infty,
\]
and such that, for every $q \in \Q$, there is a unique maximizer of $X(s)$ over $s \in [q,\infty)$ that is strictly larger than $q$.

\medskip \noindent \textbf{Part~\ref{Mclosed}}: This follows from Lemma~\ref{lemma:convergence of maximizers from converging sets}.

\medskip \noindent \textbf{Part~\ref{non_discrete}}: Let $\omega \in \Omega_{\Q}$, $\hat t \in M^U$ and $\ve > 0$, and let $q \in \Q$ satisfy $\hat t < q < \hat t + \ve$. Then, there is a unique maximizer, $s> q$ of $X(u)$ over $u \in [q,\infty)$. By assumption, $\hat t$ uniquely maximizes  $X(u)$ over the larger set $[\hat t,\infty)$, so $X(q) < X(s) < X(\hat t)$. By the intermediate value theorem, there exists a point $u$ with $\hat t < u < q$ and $X(u) = X(s)$. Then, set
\[
t = \max\{u \in (\hat t,s): X(u) = X(s)\}.
\]
Then, $t < q < s$, and 
\[
X(t) = X(s) = \sup_{t \le u <\infty}X(u)
\]
 because $X(u) < X(s)$ for all $t < u < q$. Therefore, $\hat t < t < q < \hat t + \ve$, and $t \in M^N$.

Now, let $t \in M^N$ and $\ve > 0$. Let $q \in \Q$ satisfy $t - \ve < q < t$. Then, there is a unique maximizer $\hat t$ of $X(s)$ over $s \in [q,\infty)$, so $X(\hat t) > X(s)$ for all $s \in [q,\infty) \backslash \{\hat t\}$. Specifically, $\hat t \in M^U$ and $X(\hat t) > X(t)$, so since $t$ maximizes $X(s)$ over $s \in [t,\infty)$, $\hat t$ cannot be greater than or equal to $t$. Therefore, $t - \ve < q < \hat t < t$. 

\medskip \noindent \textbf{Part~\ref{non_discrete MN}}: The first statement follows immediately from Part~\ref{non_discrete}. We now prove the second statement. We start by showing that for each $\omega \in \Omega_{\Q}$, there does not exist $t \in \R$ such that $X(s)$ has three maximizers over $s \in [t,\infty)$. If, on the contrary, such a value of $t$ exists, at least two of the maximizers must be greater than $t$, and therefore, these two maximizers are also maximizers of $X(s)$ over $s \in [q,\infty)$ for some $q \in \Q$.  Thus, on $\Omega_{\Q}$, for each $t \in M^N$, there exists a unique $\hat t \in M^U$ such that $\hat t > t$ and $(t,\hat t) \cap M = \varnothing.$

\medskip \noindent \textbf{Part~\ref{countable}}: The set $M$ is nonempty because, for any $s \in \R$, any maximizer $t$ of $X(u)$ over $u \in [s,\infty)$ lies in $M$. Then, by the first statement of Part~\ref{non_discrete MN}, the set $M^N$ is infinite. By the second statement of Part~\ref{non_discrete MN}, $M^N$ is countable.
 \end{proof}
 \begin{remark}
 Originally discovered by Taylor~\cite{taylor_1955}, it is well known that the zero set of Brownian motion almost surely has Hausdorff dimension $\f{1}{2}$. Parts~\ref{Mclosed} and~\ref{non_discrete} of Theorem~\ref{thm:countable non unique maximizers} imply that $M$ is the supporting set of the Lebesgue-Stieltjes measure defined by the nondecreasing function $t \mapsto -\sup_{t \le s < \infty} X(s)$. Using this fact, Taylor's proof (see also Theorem 4.24 in~\cite{morters_peres_2010}) can be modified to show that the set $M$ almost surely has Hausdorff dimension $\f{1}{2}$. 
 \end{remark}

\section{The Brownian queue and stationary last-passage process} \label{section:queue and stationary}
This section discusses the Brownian queue in the formulation of~\cite{brownian_queues}. Let $A$ and $S$ be two independent, two-sided Brownian motions, and let $\lambda > 0$. For $s < t$, $A(s,t)$ represents the arrivals to the queue in the time interval $(s,t]$, and $\lambda(t - s) - S(s,t)$ is the amount of service available   in  $(s,t]$.   For $t \in \R$, set
\[
q(t) = \sup_{-\infty < s \le t}\{A(s,t) + S(s,t) - \lambda(t - s)\}\qquad\text{and}\qquad d(t) = A(t) + q(0) - q(t).
\]
In queuing terms, $q(t)$ is the length of the queue at time $t$, and for $s < t$, $d(s,t)$ is the number of departures from the queue in the interval $(s,t]$. These processes are not integer-valued, but are viewed as heavy-traffic limits. We also define $e(t) = S(t) + q(0) - q(t)$. The following is due to O'Connell and Yor~\cite{brownian_queues}. Without the statements for the process $e$ the theorem is a special case of a more general result  previously   shown by Harrison and Williams~\cite{harrison1990}.
\begin{theorem}[\cite{brownian_queues}, Theorem 4] \label{O Connell Yor original}
The processes $d$ and $e$ are independent, two-sided Brownian motions. Furthermore, for each $t \in \R$, 
$\{d(s,t),e(s,t): - \infty < s \leq t\}$ is independent of $\{q(u): u \ge t\}$.
\end{theorem}
We reformulate Theorem~\ref{O Connell Yor original} in terms of the queuing mappings of~\eqref{definition of Q}--\eqref{definition of R} and~\eqref{reverse definition of Q}--\eqref{reverse definition of R}. 
\begin{theorem} \label{O Connell Yor BM independence theorem queues} Let $Y$ be a two-sided Brownian motion with drift $\lambda > 0$, independent of the two-sided Brownian motion $C$ (with no drift). Then, $\Da(Y,C)$ is a two-sided Brownian motion with drift $\lambda$, independent of the two-sided Brownian motion $\Ra(Y,C)$. Furthermore, for all $t \in \R$, $\{(\Da(Y,C)(s,t),\Ra(Y,C)(s,t)): - \infty < s \leq t\}$ is independent of $\{\Qa(Y,C)(u): u \geq t\}$. 
\end{theorem} 
\begin{proof}
Let $A(t) = \lambda t - Y(t)$ and $S(t) = C(t)$, two  independent two-sided Brownian motions. Then 
\begin{align*}
q(t) &= \sup_{-\infty < s \le t}\{A(s,t) + S(s,t) - \lambda(t - s)\} = \sup_{-\infty < s \le t}\{C(s,t) - Y(s,t)\} = \Qa(Y,C)(t). 
\end{align*}
Next observe that 
\begin{align*}
\Da(Y,C)(t) &= Y(t) + \Qa(Y,C)(t) - \Qa(Y,C)(0) = -A(t) + \lambda t + q(t) - q(0) = \lambda t -d(t),
  \\[2pt]
\text{and}\quad \Ra(Y,C)(t) &= C(t) + \Qa(Y,C)(0) - \Qa(Y,C)(t) = S(t) + q(0) - q(t) = e(t).
\end{align*}
The result follows from Theorem~\ref{O Connell Yor original}.
\end{proof}
Fix a parameter $\lambda > 0$.   Given an environment of Brownian motions $\mbf B = \{B_m\}_{m \in\Z}$, set \[
Y_0^\lambda(t) = -B_0(t) + \lambda t.
\]
For $m > 0$, recalling definitions~\eqref{reverse definition of Q}--\eqref{reverse definition of R}, set
\begin{equation} \label{eqn:stationary BLPP definitions}
q_m^\lambda(t) := \Qa(Y_{m - 1}^\lambda,B_m)(t), \;\;
Y_m^\lambda(t) := \Da(Y_{m - 1}^\lambda,B_m)(t),\;\;\text{and}\;\;
W_{m - 1}^\lambda(t) := \Ra(Y_{m - 1}^\lambda,B_m)(t).  
\end{equation}
The increment-stationary BLPP is constructed as follows:   For $(n,t) \in \Z_{> 0} \times \R$, set
\[
L^\lambda_{(n,t)} = \sup_{-\infty < s \le t}\{Y_0^\lambda(s) + L_{(1,s),(n,t)}(\mbf B)\}.
\]
One can check inductively, as   in Section 4 of~\cite{brownian_queues}, that for $m \geq 1$  and $s,t \in \R$,
    \[
    q_m^\lambda(t) = L^\lambda_{(m,t)} - L^\lambda_{(m - 1,t)},\;\;\;\text{and}\;\;\;\; Y_m^\lambda(s,t) = L^\lambda_{(m,t)} - L^\lambda_{(m,s)}. 
    \]
    The stationarity of  these increments comes from the next theorem that  follows from Theorem~\ref{O Connell Yor BM independence theorem queues} and induction. It is the zero-temperature analogue of Theorem 3.3 in~\cite{Sepp_and_Valko} and Theorem 2.11 in ~\cite{blpp_utah}.   Figure~\ref{fig:Independence structure for Brownian queue} demonstrates the independence structure, and we give credit to~\cite{blpp_utah} for a very similar picture.
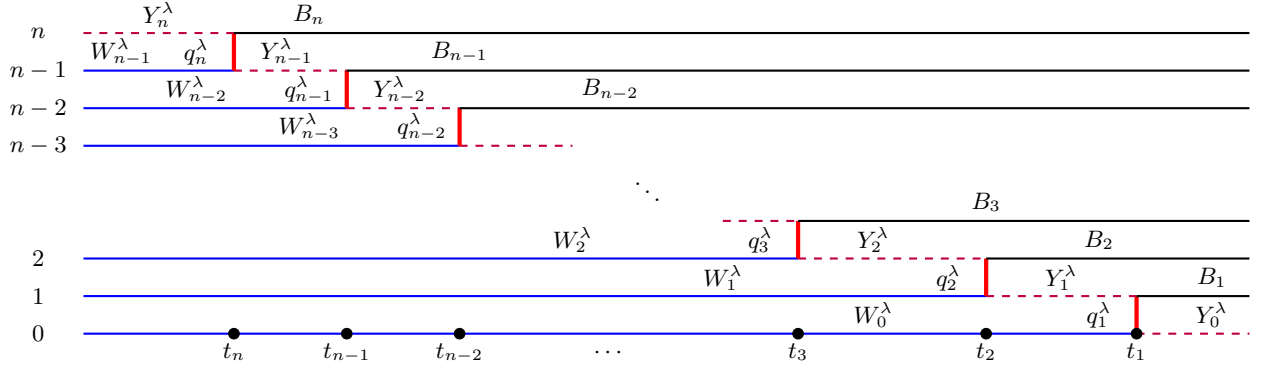
\begin{figure}[t]
\begin{tikzpicture}
\draw[blue,thick] (-0.5,0) -- (13.5,0);
\node at (-1.1,0) {\small $0$};
\node at (-1.1,0.5) {\small $1$};
\node at (-1.1,1) {\small $2$};
\node at (-1.1,2.5) {\small $n - 3$};
\node at (-1.1,3) {\small $n - 2$};
\node at (-1.1,3.5) {\small $n -1$};
\node at (-1.1,4) {\small $n$};
\draw[purple, thick, dashed] (13.5,0)--(15,0);
\node at (10,0.25) {\small $W^\lambda_0$};
\draw[red, ultra thick] (13.5,0)--(13.5,0.5);
\node at(13,0.25) {\small $q_{1}^\lambda$};
\draw[blue,thick] (-0.5,0.5) --(11.5,0.5);
\draw[red, ultra thick] (11.5,0.5)--(11.5,1);
\draw[purple,thick, dashed] (11.5,0.5)--(13.5,0.5);
\draw[black,thick] (13.5,0.5)--(15,0.5);
\node at (8,0.75) {\small$W^\lambda_{1}$};
\node at (11,0.75) {\small$q^\lambda_{2}$};
\node at(12.5,0.75) {\small$Y^\lambda_{1}$};
\node at (14.5,0.75) {\small$B_{1}$};
\node at (14.5,0.25) {\small $Y_0^\lambda$};
\draw[blue,thick] (-0.5,1) --(9,1);
\draw[purple,thick, dashed] (9,1)--(11.5,1);
\draw[black, thick] (11.5,1)--(15,1);
\draw[red, ultra thick] (9,1)--(9,1.5);
\node at (6,1.25) {\small$W^\lambda_{2}$};
\node at (10,1.25) {\small$Y^\lambda_{2}$};
\node at (8.5,1.25) {\small$q^\lambda_{3}$};
\node at (13,1.25) {\small$B_{2}$};
\draw[black,thick] (9,1.5)--(15,1.5);
\draw[purple, thick, dashed] (8,1.5)--(9,1.5);
\node at (11.5,1.75) {\small$B_{3}$};
\draw[blue,thick] (-0.5,2.5)--(4.5,2.5);
\draw[purple, thick,dashed] (4.5,2.5)--(6,2.5);
\draw[red, ultra thick] (4.5,2.5)--(4.5,3);
\node at (4,2.75) {\small$q_{n - 2}^\lambda$};
\node at(2.5,2.75) {\small$W_{n - 3}^\lambda$};
\draw[blue, thick] (-0.5,3)--(3,3);
\draw[purple,thick, dashed] (3,3)--(4.5,3);
\node at (7,2) {\small $\ddots$};
\draw[black,thick] (4.5,3)--(15,3);
\draw[red, ultra thick] (3,3)--(3,3.5);
\node at(2.5,3.25) {\small$q^\lambda_{n - 1}$};
\node at (3.7,3.25) {\small$Y^\lambda_{n - 2}$};
\node at (1,3.25) {\small$W^\lambda_{n - 2}$};
\node at (6.5,3.25) {\small$B_{n - 2}$};
\draw[blue,thick] (-0.5,3.5)--(1.5,3.5);
\draw[purple,thick, dashed] (1.5,3.5)--(3,3.5);
\draw[black, thick] (3,3.5)--(15,3.5);
\draw[red, ultra thick] (1.5,3.5)--(1.5,4);
\node at (1,3.75) {\small$q_{n}^\lambda$};
\node at (0,3.75) {\small$W_{n - 1}^\lambda$};
\node at (2.2,3.75) {\small$Y^\lambda_{n - 1}$};
\node at (4.5,3.75) {\small$ B_{n - 1}$};
\draw[purple, thick, dashed] (-0.5,4)--(1.5,4);
\draw[black, thick] (1.5,4)--(15,4);
\node at (0.5,4.25) {\small$Y_n^\lambda$};
\node at (2.5,4.25) {\small$B_n$};
\filldraw[black] (1.5,0) circle (2pt) node[anchor = north] {\small $t_n$};
\filldraw[black] (3,0) circle (2pt) node[anchor = north] {\small $t_{n -1}$};
\filldraw[black] (4.5,0) circle (2pt) node[anchor = north] {\small $t_{n - 2}$};
\node at (6.5,-0.25) {\small $\cdots$};
\filldraw[black] (9,0) circle (2pt) node[anchor = north] {\small $t_{3}$};
\filldraw[black] (11.5,0) circle (2pt) node[anchor = north] {\small $t_{2}$};
\filldraw[black] (13.5,0) circle (2pt) node[anchor = north] {\small $t_{1}$};
\end{tikzpicture}
\caption{\small Independence structure for stationary BLPP. Each process $Y^\lambda_r$ is associated to the (purple) dashed segment on level $r$, processes $B_r$ and $W^\lambda_r$ cover the remaining portions of horizontal level $r$, and the process $q_r^\lambda$ is associated to the (red) vertical edge from level $r - 1$ to $r$ at time point $t_r$.  }
\label{fig:Independence structure for Brownian queue}
\end{figure}

\begin{theorem} \label{Burke-type Theorem}
Let $Y_m^\lambda,q_m^\lambda,W_m^\lambda$ be as constructed in~\eqref{eqn:stationary BLPP definitions}. Then, the following hold. 
\begin{enumerate} [label=\rm(\roman{*}), ref=\rm(\roman{*})]  \itemsep=3pt 
    \item For all $m > 0$ and $t \in \R$, $q_m^\lambda(t) \sim \operatorname{Exp}(\lambda)$.   $\{W_m^\lambda\}_{m \ge 0}$ is a field of independent two-sided Brownian motions.   $\{Y_m^\lambda\}_{m \ge 0}$ is a field of (non-independent) two-sided Brownian motions with drift $\lambda$.
\item Let $-\infty < t_n \leq t_{n - 1} \leq \cdots \leq t_{1} < \infty$. Then, the following are mutually independent:
\begin{multline*}
\{W_0^\lambda(u,t_1):u \le t_{1}\},\;\; q_1^\lambda(t_1),\;\;
 \{Y_0^\lambda(t_1,u): u \ge t_1\}, \; \\[1em]
\{W_{r }^\lambda(u,t_{r + 1}): u \leq t_{r + 1}\},\; q_{r + 1}^\lambda(t_{r + 1}),\; \{Y_r^\lambda(t_{r + 1},u): t_{r + 1} \le u \le t_r\},\; \\[1em]\text{and}\; \{B_r(t_r,u): u \geq t_r  \},\; \text{ for }\; 1 \leq r \leq n - 1,\;\; 
\{Y_n^\lambda(u,t_n):u \le t_n\},\;\;\text{ and }\;\;
 \{B_n(t_n,u): u \ge t_n\}.
\end{multline*}
\end{enumerate} 
\end{theorem}
\begin{proof}   To prove  Part (i), begin with the assumption: 
\[
(Y_{0}^\lambda,B_{1},\ldots,B_n)\qquad\text{are mutually independent}.
\]
By Theorem~\ref{O Connell Yor BM independence theorem queues}, $Y_{1}^\lambda = \Da(Y_{0}^\lambda,B_{ 1})$ is a two-sided Brownian motion with drift $\lambda$, independent of the two-sided Brownian motion $W_{0}^\lambda = \Ra(Y_{0}^\lambda,B_{1})$.  Hence, 
\[
(W_{0}^\lambda,Y_{1}^\lambda, B_{2},\ldots,B_n) \qquad\text{are mutually independent}.
\]
Inductively, assume that $Y_{r}^\lambda$ is a two-sided Brownian motion with drift $\lambda$ and 
\[
(W_{0}^\lambda,\ldots,W_{r - 1}^\lambda,Y_{r}^\lambda, B_{r + 1},\ldots,B_n)\qquad\text{are mutually independent}.
\]
Apply Theorem~\ref{O Connell Yor BM independence theorem queues} to  $Y_{r+ 1}^\lambda = \Da(Y_{r}^\lambda,B_{r + 1})$ and  $W_r^\lambda = \Ra(Y_r^\lambda,B_{r + 1})$ to continue the induction.  Part (i) is proved. 
The proof of Part (ii) is the proof of Theorem 3.3 in~\cite{Sepp_and_Valko}.
\end{proof}

The proof of Theorem~\ref{thm:existence of Busemann functions for fixed points} is achieved through the coupling of $q_r^\lambda,Y_r^\lambda$, and $W_r^\lambda$. The following theorem is the key.
\begin{theorem}[Zero-temperature analogue of Theorem 3.1 in~\cite{blpp_utah}] \label{thm:busemann sandwich theorem}
Fix real numbers $t > s$, $\lambda > 0$, and $\gamma > \lambda^{-2} > \delta > 0$. Then, with probability one,
\be \label{eqn:sandwich inequality for h}\begin{aligned} 
&\limsup_{n \rightarrow \infty} \bigl[ \, L_{(0,t),(n,n\delta)}(\mbf W^\lambda) - L_{(0,s),(n,n\delta)}(\mbf W^\lambda)\,\bigr]   
\leq B_0(s,t) - \lambda(t - s) =  -Y_{0}^\lambda(s,t)  \\  
&\qquad\qquad \qquad\qquad
\leq \liminf_{n \rightarrow \infty} \bigl[\, L_{(0,t),(n,n\gamma)}(\mbf W^\lambda) - L_{(0,s),(n,n\gamma)}(\mbf W^\lambda)\,\bigr], 
\end{aligned}\ee
and 
\be\label{eqn:sandwich inequality for v}\begin{aligned} 
&\limsup_{n \rightarrow \infty} \bigl[\,L_{(0,t),(n,n\delta)}(\mbf W^\lambda) - L_{( 1,t),(n,n\delta)}(\mbf W^\lambda)\,\bigr]  
 \\ 
&\qquad\qquad\qquad  \le q_{1}^\lambda(t)  
\leq \liminf_{n \rightarrow \infty}  \bigl[\, L_{(0,t),(n,n\gamma)}(\mbf W^\lambda) - L_{( 1,t),(n,n\gamma)}(\mbf W^\lambda)\,\bigr] .  
\end{aligned}\ee

\end{theorem}
Here, $L_{\mbf x,\mbf y}(\mbf W^\lambda)$ is the last passage time between points $\mbf x \le \mbf y \in \Z_{> 0} \times \R$ when $\mbf W^\lambda := \{W_m^\lambda\}_{m \ge 0}$ is the random environment. ($W_m^\lambda$ need not be defined for $m < 0$ since all points $(m,t) \in \Z \times \R$ in the above expression satisfy $m \ge 0$.) While Theorem~\ref{thm:busemann sandwich theorem} is not explicitly stated in~\cite{blpp_utah}, all necessary lemmas are provided in Section 4 of that paper.

\begin{lemma} \label{lemma:equality of Busemann function distribution at rationals}
Let $\mathcal A \subseteq \R$ be any countable set. Consider the space $\R^{\A}$ of functions $\A \rightarrow \R$, equipped with the Borel product $\sigma$-algebra. Fix $\theta > 0$. Then,
\begin{align*}
    &\bigl\{\h_0^\theta(t) : t \in \mathcal A\bigr\}
    \overset{d}{=} \bigl\{-B_0(t) + \f{1}{\sqrt \theta} t: t \in \mathcal A \bigr\} \qquad \text{and} \qquad
    \bigl\{\vv_{1}^\theta(t): t \in \mathcal A\bigr\} \overset{d}{=} \bigl\{q_{1}^{\f{1}{\sqrt \theta}}(t): t \in \mathcal A\bigr\}.
\end{align*}
\end{lemma}

\begin{proof}
This proof is a multivariate extension of the ``Proof of Theorem 2.5, assuming Theorem 3.1" on page 1937 of~\cite{blpp_utah}. This method originated in the setting of the log-gamma polymer in~\cite{geor-rass-sepp-yilm-15}.  From definition \eqref{horizontal Busemann simple expression} of $\h_0^\theta$, $\h_0^\theta(0) = 0 = B_0(0)$. Hence, to prove the first statement, it is sufficient to show that
\[
\bigl\{\h_0^\theta(s,t):s,t \in \mathcal \A, s < t\bigr\} \deq \bigl\{-B_0(s,t) + \f{1}{\sqrt \theta}(t - s):s,t \in \mathcal \A, s < t\bigr\}. 
\]
Theorem~\ref{thm:existence of Busemann functions for fixed points} establishes that, for fixed $s,t \in \R$,
\[
\h_0^\theta(s,t) = \lim_{n \rightarrow \infty} \bigl[ L_{(0,s),(n,n\theta)}(\mbf B) - L_{(0,t),(n,n\theta)}(\mbf B)\bigr]  \text{ a.s.}
\]
   By Theorem~\ref{Burke-type Theorem}, $\{W_m^\lambda\}_{m \ge 0}$ is a field of independent, two-sided Brownian motions. By equality of distribution between $\{W_m^\lambda\}_{m \ge 0}$ and $\{B_m\}_{m \ge 0}$, the first inequality of~\eqref{eqn:sandwich inequality for h} implies that if $\{s_1,\ldots,s_k\}$ and $\{t_1,\ldots,t_k\}$ are finite subsets of $\mathcal A$ with $s_i < t_i$ for all $i$, and $x_1,\ldots,x_k \in \overline \R$, then whenever $\theta < \lambda^{-2}$, or equivalently, $\lambda < \frac{1}{\sqrt \theta}$,
\begin{align*}
\Pp\bigl(-\h_0^\theta(s_i,t_i) \ge x_i, \;\;1 \le i \le k\bigr)
\leq \Pp\bigl(B_0(s_i,t_i) - \lambda(t_i - s_i)\ge x_i,\;\; 1 \le i \le k\bigr).
\end{align*}
 Since $\lambda > 0$ and $s_i < t_i$ for all $i$, the right-hand side 
decreases as $\lambda$ increases, so taking $\lambda \nearrow \frac{1}{\sqrt \theta}$ gives us
\begin{align*}
\Pp\bigl(-\h_0^\theta(s_i,t_i) \ge x_i, \;\;1 \le i \le k\bigr) 
\leq \Pp\Bigl(B_0(s_i,t_i) - \f{1}{\sqrt \theta}(t_i - s_i)\ge x_i,\;\; 1 \le i \le k\Bigr).
\end{align*}
 The inequality on the right of \eqref{eqn:sandwich inequality for h} with $\lambda \searrow \frac{1}{\sqrt \theta}$ establishes the reverse inequality. 


  The same argument works for  $\vv_1^\theta$, utilizing this monotonicity for $\gamma < \delta$:
\begin{align*}
q_1^\gamma(t) &= \Qa(Y,B_1)(t) = \sup_{-\infty < s \le t}\{B_1(s,t) + B_0(s,t) - \gamma(t - s)\} \\[1em]
&\ge \sup_{-\infty < s \le t}\{B_1(s,t) + B_0(s,t) - \delta(t - s)\} \ge q_1^\delta(t). 
\qedhere
\end{align*}
\end{proof}

\section{Time reversal}
\begin{theorem} \label{bijectivity of D R joint mapping}
Let $Z,B,Y,C$ be continuous functions satisfying $Z(0) = B(0) = Y(0) = C(0) = 0$ and 
\[
\lim_{t \rightarrow \pm\infty} (B(t) - Z(t)) = \lim_{t \rightarrow \pm\infty} (C(t) - Y(t)) = \mp \infty.
\]
Then,  
\[
Y = D(Z,B) \text{ and } C = R(Z,B) \qquad\text{if and only if}\qquad Z = \Da(Y,C) \text{ and } B = \Ra(Y,C),
\]
The equalities above denote equality as functions of $t$. If either of the above equivalent conditions are satisfied, then also $Q(Z,B) = \Qa(Y,C)$.
\end{theorem}
The following lemmas will help to prove Theorem~\ref{bijectivity of D R joint mapping}.
\begin{lemma} \label{lemma:time reversal equality of queuing maps}
Let $Z,B:\R\rightarrow \R$ be continuous functions satisfying $X(0) = B(0) = 0$ and
\[
\lim_{t \rightarrow \pm \infty} (B(t) - X(t)) = \mp \infty.
\]
For a function $f:\R \rightarrow \R$, recall that we define $\wt f:\R \rightarrow \R$ by $\wt f(t) = -f(-t)$.
Then, for all $t \in \R$,
\[
Q(Z,B)(-t) = \Qa(\wt Z,\wt B)(t), \qquad -D(Z,B)(-t) = \Da(\wt Z,\wt B)(t),\qquad\text{and}\qquad -R(Z,B)(-t) = \Ra(\wt Z,\wt B)(t).
\]
\end{lemma}
\begin{proof}
This is a routine check, using the definitions. 
\end{proof}

The following is a well-known fact, but is often stated without proof, so we include full justification for the sake of completeness. For example, it appears as Equation (1.4) in~\cite{Pitman1975} and Equation (13) in~\cite{brownian_queues}.
\begin{lemma} \label{pitman Representation Lemma}
Let $f:\R \rightarrow \R$ be a continuous function such that 
\[
\lim_{t \rightarrow \pm \infty} f(t) = \pm \infty.
\]
Set $F(t) = \sup_{-\infty < s \le t} f(s)$. Then,
\begin{equation} \label{pitman Representation Equation}
\inf_{t \le s < \infty } (2 F(s) - f(s)) = F(t).
\end{equation}
\end{lemma}
\begin{proof}
The left-hand side of~\eqref{pitman Representation Equation} is
\[
\inf_{t \le s < \infty} (2F(s) - f(s)) = \inf_{t \le s < \infty} \bigl(2\sup_{-\infty <u \leq s} f(u) - f(s)\bigr).
\]
For all $s \geq t$, $\sup_{\infty < u \le s} f(u)$ is greater than or equal to both $f(s)$ and $\sup_{-\infty < u \le t} f(u)$. Therefore,
\[
2\sup_{-\infty < u \leq s} f(u) - f(s)\geq \sup_{-\infty < u \leq t} f(u)+ f(s) - f(s) = F(t).
\]
 This establishes one direction of Equation~\eqref{pitman Representation Equation}. To show the other direction, we show that there exists $s \geq t$ such that 
\[
2F(s) - f(s) = F(t). 
\]
Note that $f(t) \leq F(t)$ and that $\lim_{s \rightarrow \infty} f(s) = \infty$ by assumption. Hence, by continuity of $f$,  $f(s) = F(t)$ for some $s \geq t$. Let 
\[
s^* = \inf\{s \geq t: f(s) = F(t)\}.
\]
Then, $F(s^*) = F(t)$. Therefore,
\[
2F(s^\star) - f(s^*) = 2F(t) - F(t) = F(t). \qedhere
\]
\end{proof}

\begin{proof}[Proof of Theorem~\ref{bijectivity of D R joint mapping}]
Assume that $Z,B,Y,C$ satisfy the conditions of the Theorem. First, assume that $Z = \Da(Y,C)$ and $B = \Ra(Y,C)$.
By definitions~\eqref{reverse definition of Q}--\eqref{reverse definition of R},
\begin{align*}
Z(s) &=Y(s) + \sup_{-\infty < u \le s}\{C(u,s) - Y(u,s)\} - \sup_{-\infty < u \le 0}\{C(u,0)- Y(u,0)\} \\[1em]
&= C(s) + \sup_{-\infty < u \le s}\{Y(u)- C(u)\} - \sup_{-\infty < u \le 0}\{Y(u) - C(u)\},\qquad\text{and} \\[1em]
B(s) &=C(s) + \sup_{-\infty < u \le 0}\{C(u,0)- Y(u,0)\} - \sup_{-\infty < u \le s}\{C(u,s) - Y(u,s)\} \\[1em]
&= Y(s)+\sup_{-\infty < u \le 0}\{Y(u) - C(u)\} - \sup_{-\infty < u \le s}\{Y(u) - C(u)\}.
\end{align*}

Then,
\begin{align*}
&D(Z,B)(t) = Z(t) + Q(Z,B)(0) - Q(Z,B)(t)  \\[1em]
&= Z(t) + \sup_{0 \le s < \infty}\{B(0,s) - Z(0,s)\} - \sup_{t \le s < \infty}\{B(t,s) - Z(t,s)\} \\[1em]
&= B(t) +\sup_{0 \le s < \infty}\{B(s) - Z(s)\} - \sup_{t \le s < \infty}\{B(s) - Z(s)\} \\[1em]
&= Y(t) + \sup_{-\infty < s \le 0}\{Y(s) - C(s)\} - \sup_{-\infty < s \le t}\{Y(s) - C(s)\} \\[1em]
&\qquad\qquad\qquad+ \sup_{0 \le s < \infty}\bigl\{Y(s) - C(s) - 2\sup_{-\infty < u \le s}\{Y(u) - C(u)\} \bigr\}  \\[1em]
&\qquad\qquad\qquad\qquad\qquad\qquad\qquad\qquad\qquad-\sup_{t \le s < \infty}\bigl\{Y(s) - C(s) - 2\sup_{-\infty < u \le s}\{Y(u) - C(u)\} \bigr\}. 
\end{align*}
To show that this equals $Y(t)$, it is therefore sufficient to show that for $t \in \R$,
\[
\sup_{-\infty < s \le t}\{Y(s) - C(s)\} = \inf_{t \le s < \infty}\bigl\{2\sup_{-\infty < u \le s}\{Y(u) - C(u)\} - (Y(s) - C(s))   \bigr\},
\]
which follows from Lemma~\ref{pitman Representation Lemma}. The proof that $R(Z,B)(t) = C(t)$ follows by the same reasoning. The converse then follows by the previous case and Lemma~\ref{lemma:time reversal equality of queuing maps}: if $Y = D(Z,B)$ and $C = R(Z,B)$, $\wt Y = \Da(\wt Z,\wt B)$ and $\wt C = \Ra(\wt Z,\wt B)$, so $\wt Z = D(\wt Y,\wt C)$ and $\wt B = R(\wt Y,\wt C)$. Hence $Z = \Da(Y,C)$ and $B = \Ra(Y,C)$.

We finish by showing that $Q(Z,B) = \Qa(Y,C)$ whenever the two equivalent conditions of the theorem are met. For all $t \in \R$,
\begin{align*}
Y(t) &= D(Z,B)(t) = Z(t) + Q(Z,B)(0) - Q(Z,B)(t), \text{ and } \\[1em]
Z(t) &= \Da(Y,C) = Y(t) + \Qa(Y,C)(t) - \Qa(Y,C)(0).
\end{align*}
Putting these two equations together,
\[
Q(Z,B)(t) - \Qa(Y,C)(t) = Q(Z,B)(0) - \Qa(Z,B)(0).
\]
This is true for all $t \in \R$, so $Q(Z,B)(t) - \Qa(Y,C)(t)$ must be equal to some constant. Recall that 
\[
Q(Z,B)(t) = \sup_{t \le s <\infty}\{B(t,s) - Z(t,s)\},\qquad\text{and}\qquad\Qa(Y,C)(t) = \sup_{-\infty < s \le t}\{C(s,t) - Y(s,t)\}.
\]
By the limit conditions of the theorem, maximizers exist for each of the supremums above, so $Q(Z,B)$ and $\Qa(Y,C)$ both achieve a minimum value of $0$. Hence, the constant must be $0$. 
\end{proof}

\section*{Acknowledgements}
The authors thank Tom Alberts, Firas Rassoul-Agha, and Neil O'Connell for helpful discussions. We also thank the anonymous referees for their feedback that has improved the exposition of this paper. T.\ Sepp\"al\"ainen was partially supported by National Science Foundation grant DMS-1854619 and by the Wisconsin Alumni Research Foundation. E. Sorensen was partially supported by T. Sepp{\"a}l{\"a}inen, through National Science Foundation grants DMS-1602846 and DMS-1854619.

\bibliographystyle{alpha}
\bibliography{references_file}

\end{document}